\documentclass{article}

\usepackage[margin=1in]{geometry}
\usepackage{amsmath, amsfonts, amsthm, amssymb, cite}
\usepackage{mathtools, suffix, enumitem, bm, graphbox, pgfplots, anyfontsize}
\pgfplotsset{compat=1.18}
\mathtoolsset{showonlyrefs}
\usepackage[colorlinks]{hyperref}
\hypersetup{linkcolor=blue}
\usepackage{xcolor}

\usepackage{multirow}
\usepackage{figures/trees/tikzit}

\tikzstyle{black}=[fill=black, draw=black, shape=circle]

\tikzstyle{red}=[-, draw=red]
\tikzstyle{orange}=[-, draw={rgb,255: red,255; green,137; blue,2}]
\tikzstyle{yellow}=[-, draw={rgb,255: red,207; green,190; blue,0}]
\tikzstyle{green}=[-, draw={rgb,255: red,25; green,255; blue,140}]
\tikzstyle{cyan}=[-, draw=cyan]
\tikzstyle{blue}=[-, draw={rgb,255: red,0; green,64; blue,255}]
\tikzstyle{purple}=[-, draw={rgb,255: red,128; green,0; blue,255}]
\tikzstyle{magenta}=[-, draw={rgb,255: red,255; green,0; blue,191}]

\usepackage{wasysym}

\newtheorem{theorem}{Theorem}[section]
\newtheorem{lemma}[theorem]{Lemma}
\newtheorem{corollary}[theorem]{Corollary}

\newtheorem{proposition}[theorem]{Proposition}
\newtheorem{conjecture}[theorem]{Conjecture}

\theoremstyle{definition}

\usepackage{mleftright}
\renewcommand{\left}{\mleft}
\renewcommand{\right}{\mright}

\renewcommand{\P}{\mathbb{P}}
\newcommand{\E}{\mathbb{E}}
\newcommand{\N}{\mathbb{N}}
\newcommand{\R}{\mathbb{R}}
\newcommand{\G}{\mathbb{G}}
\newcommand{\cG}{\mathcal{G}}
\newcommand{\cB}{\mathcal{B}}

\renewcommand{\c}{\mathsf{c}}

\newcommand{\Var}{\operatorname{Var}}

\newcommand{\Cov}{\operatorname{Cov}}
\newcommand{\TV}{\mathrm{TV}}
\renewcommand{\d}{\mathbf{d}}
\newcommand{\ind}[1]{\mathbf{1}_{\{#1\}}}

\newcommand{\fr}{\frac}
\newcommand{\eps}{\varepsilon}
\newcommand{\Exp}[1]{\exp\left(#1\right)}
\newcommand{\Econd}[2]{\E\left[ #1 \,\middle|\, #2 \right]}
\newcommand{\Pcond}[2]{\P\left[ #1 \,\middle|\, #2 \right]}
\WithSuffix\newcommand\ind*[1]{\mathbf{1}_{#1}}
\newcommand{\sse}{\subseteq}

\newcommand{\allball}{B_\eta^\square(p^*)}
\newcommand{\ball}{B_\eta^\square}
\newcommand{\allhalfball}{B_{\eta/2}^\square(p^*)}
\newcommand{\halfball}{B_{\eta/2}^\square}
\newcommand{\dtv}{\d_\TV}
\renewcommand{\dh}{\d_\mathrm{H}}
\newcommand{\dl}{\d_\Lambda}
\newcommand{\dlb}{\d_{\bar{\Lambda}}}
\newcommand{\db}{\d_\square}

\newcommand{\dto}{\overset{d}{\longrightarrow}}

\newcommand{\edgeset}{\binom{[n]}{2}}

\let\temp\phi
\let\phi\varphi
\let\varphi\temp

\title{Concentration via metastable mixing, with applications to the supercritical exponential random graph model}
\author{Vilas Winstein \\ University of California, Berkeley \\ \texttt{vilas@berkeley.edu}}

\begin{document}

\maketitle

\begin{abstract}
Folklore belief holds that metastable wells in low-temperature statistical mechanics models exhibit high-temperature behavior.
We make this rigorous in the exponential random graph model (ERGM) through the lens of concentration of measure.
We make use of the supercritical (low-temperature) \emph{metastable mixing} which was recently proven
by Bresler, Nagaraj, and Nichani \cite{bresler2024metastable} and obtain a novel concentration inequality for Lipschitz
observables of the ERGM in a large metastable well, answering a question posed \cite{bresler2024metastable}.
To achieve this, we prove a new connectivity property for metastable mixing in the ERGM
and introduce a new general result yielding concentration inequalities, which extends a result of Chatterjee
\cite{chatterjee2005concentration}.
We also use a result of Barbour, Brightwell, and Luczak \cite{barbour2022long} to cover all cases of interest.
Our work extends a result of Ganguly and Nam \cite{ganguly2024sub} from the subcritical (high-temperature) regime to metastable
wells, and we also extend applications of this concentration, namely a central limit theorem for small
subcollections of edges and a bound on the Wasserstein distance between the ERGM and the Erd\H{o}s--R\'enyi random graph.
Finally, to supplement the mathematical content of the article, we present a simulation study of metastable wells
in the supercritical ERGM.
\end{abstract}

\setcounter{tocdepth}{2}
\tableofcontents

\section{Introduction}
\label{sec:intro}

Exponential random graph models (ERGMs) are exponential families of random graph models.
In the present article, we focus on the case where the sufficient statistics are the counts of certain small subgraphs.
In other words, for us the ERGM is a Gibbs measure generalizing the Erd\H{o}s--R\'enyi model, where the probability of a graph
is tilted according to the counts of particular subgraphs.
For instance, the presence of triangles in the graph
may be encouraged or discouraged, depending on the parameters of the model.

More precisely, given a fixed sequence of finite graphs $G_i = (V_i, E_i)$ for $0 \leq i \leq K$ (where we always take
$G_0$ to be a single edge), and a vector of parameters $\bm\beta = (\beta_0, \dotsc, \beta_K) \in \R^{1+K}$, the ERGM
is a probability distribution $\tilde{\mu}$ on the space of simple graphs $x$ on $n$ vertices given by $\tilde{\mu}(x) \propto \Exp{H(x)}$,
where the \emph{Hamiltonian} $H$ is defined as
\begin{equation}
    H(x) = \sum_{i=0}^K \fr{\beta_i}{n^{|V_i|-2}} N_{G_i}(x).
\end{equation}
Here $N_G(x)$ denotes the number of \emph{labeled} subgraphs of $x$ isomorphic to $G$.
The normalization by $n^{|V_i|-2}$ ensures that all terms in the Hamiltonian are of the same order, namely $n^2$, which
is the exponential order of the total number of simple graphs on $n$ vertices.
When $K = 0$, so that the only graph in the ERGM specification is a single edge,
one recovers the Erd\H{o}s--R\'enyi model $\cG(n,p)$ with $p = \fr{e^{2 \beta_0}}{1 + e^{2 \beta_0}}$.

The ERGM has been used as a model of complex networks with various clustering properties, and has found applications in fields like
sociology and biology \cite{frank1986markov,holland1981exponential,wasserman1994social}.
Early nonrigorous analysis of the ERGM was also contributed by statistical physicists
\cite{burda2004network,park2004solution,park2005solution}.
From a mathematical perspective, the ERGM stands out as one of the simplest natural models which incorporates higher-order
(i.e.\ not just pairwise) interactions, since the presence of a subgraph isomorphic to $G = (V,E)$ depends on $|E|$ different edges.
As such, the ERGM has attracted the attention of probabilists, statisticians, and computer scientists who have investigated the model
from a variety of perspectives including
large deviations and phase transitions \cite{chatterjee2013estimating,radin2013phaseComplex,radin2013phaseERGM},
distributional limit theorems \cite{mukherjee2023statistics,fang2024normal},
mixing times and sampling \cite{bhamidi2008mixing,demuse2019mixing,reinert2019approximating,bresler2024metastable},
concentration of measure \cite{sambale2020logarithmic,ganguly2024sub}, and stochastic localization
\cite{eldan2018decomposition,eldan2018exponential,mikulincer2024stochastic}.
Moreover, the ERGM is useful as a tool for studying random graphs \emph{constrained} to have a certain abnormal
number of triangles or other subgraphs \cite{radin2014asymptotics,lubetzky2015replica}.
Our definition of the ERGM gives dense graphs, but some recent work has also studied
modifications which yield sparse graphs instead \cite{yin2017asymptotics,cook2024typical}.
This list of citations is meant only to display the breadth of the subject and is not comprehensive; many important
works have been omitted.

As with many other Gibbs measures, the macroscopic behavior of the ERGM can be explained via a large deviations principle
which characterizes the system to first order.
Large deviations principles, however, leave open the question of the microscopic behavior, i.e.\ the fluctuations of
various observables of the system at finite size, in the same way that the standard law of large numbers
\begin{equation}
    \sum_{i=1}^n Y_i = n \cdot \E[Y_1] + o(n)
\end{equation}
for i.i.d.\ square-integrable random variables $(Y_i)_{i=1}^\infty$
gives no indication that the error term $o(n)$ is typically of order $\sqrt{n}$.
Since we do not have direct control over the correlations,
it is often a delicate matter to extract fluctuation information from a Gibbs measure, and there are many cases where
this is still open.

Broadly speaking, the present article is about the order of fluctuations of certain observables of the ERGM
in the supercritical (low-temperature) regime of parameters.
This regime, which will be defined in the next subsection, has
been largely unexplored as compared with the subcritical (high-temperature) regime.

\subsection{Macroscopic and microscopic behavior}
\label{sec:intro_macroscopic}

In \cite{chatterjee2013estimating}, a large deviations principle for the ERGM was demonstrated, using the language of \emph{graphons},
which are limit points of finite graphs in a certain metric space which has been completed with respect to the \emph{cut distance}.
These notions will be defined precisely in Section \ref{sec:review_ldp} below.
The large deviations principle of \cite{chatterjee2013estimating} says that a sample from the ERGM is typically close (in the cut distance)
to the set of maximizers of a certain functional $\mathcal{F}$ on the space of graphons.

In the present work, we focus on the \emph{ferromagnetic} regime where $\beta_i \geq 0$ for all $i \geq 1$ (note that we still allow
$\beta_0$ to be negative).
In this regime, the large deviations principle of \cite{chatterjee2013estimating} simplifies, and it can be shown that the maximizers
of the functional $\mathcal{F}$ are all \emph{constant graphons} $W_p$ for $p \in [0,1]$, which are the limit points of Erd\H{o}s--R\'enyi
random graphs $\cG(n,p)$.
When restricted to constant graphons, the functional $\mathcal{F}$ can be written as
\begin{equation}
\label{eq:Ldef}
    \mathcal{F}(W_p) =
    L_{\bm\beta}(p) \coloneqq \sum_{i=0}^K \beta_i p^{|E_i|} - I(p),
\end{equation}
where $I(p) = \tfrac{1}{2} p \log p + \tfrac{1}{2} (1-p) \log(1-p)$.
In essence, then,  the large deviations principle states that a sample from the ferromagnetic ERGM
is typically close in cut distance to a sample from $\cG(n,p^*)$ for some $p^* \in M_{\bm\beta}$,
the set of global maximizers of $L_{\bm\beta}$ (see Figure \ref{fig:ergm_pstar}).

\begin{figure}
\begin{center}
    \includegraphics[align=c,scale=0.15]{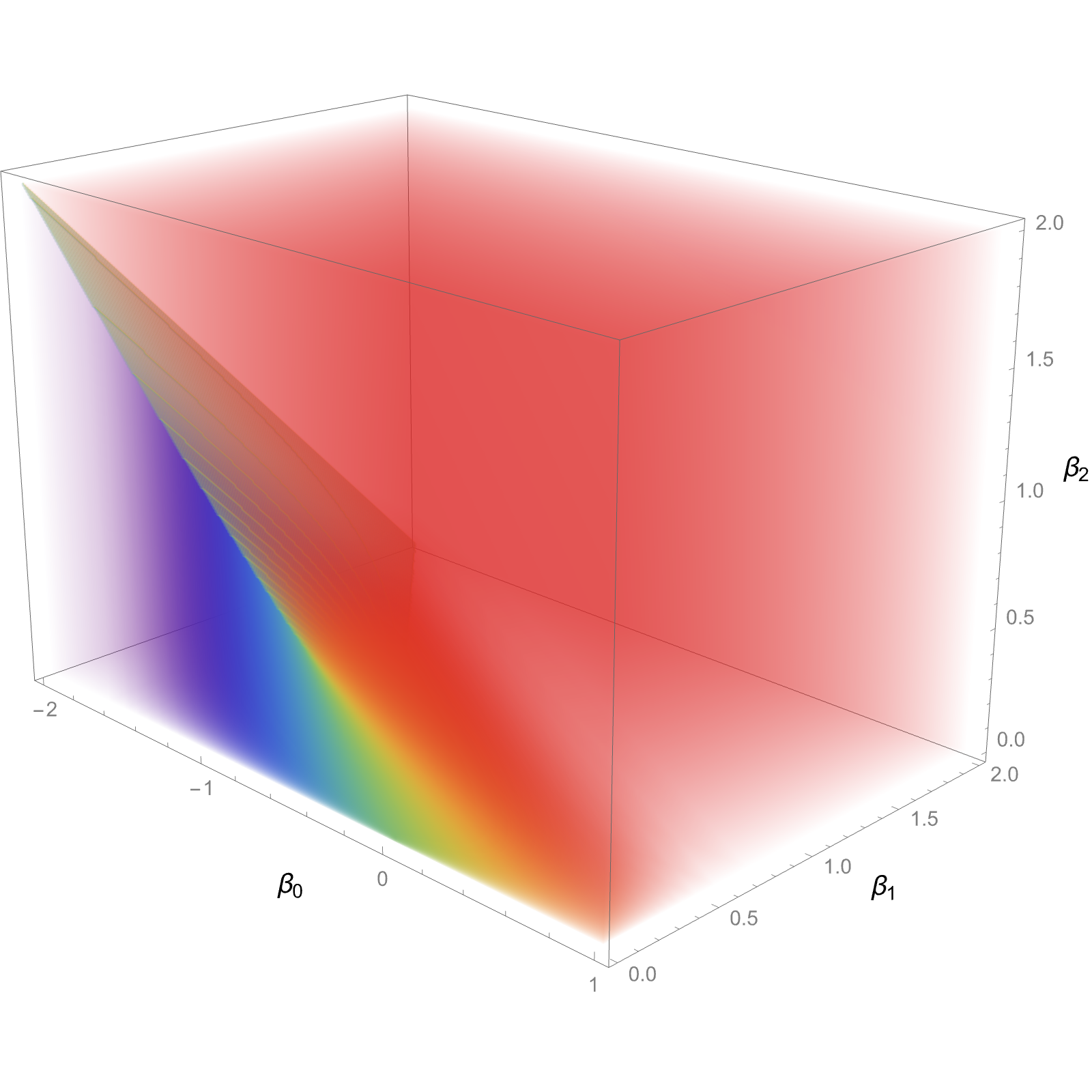}
    \hspace{5mm}
    \includegraphics[align=c,scale=0.15]{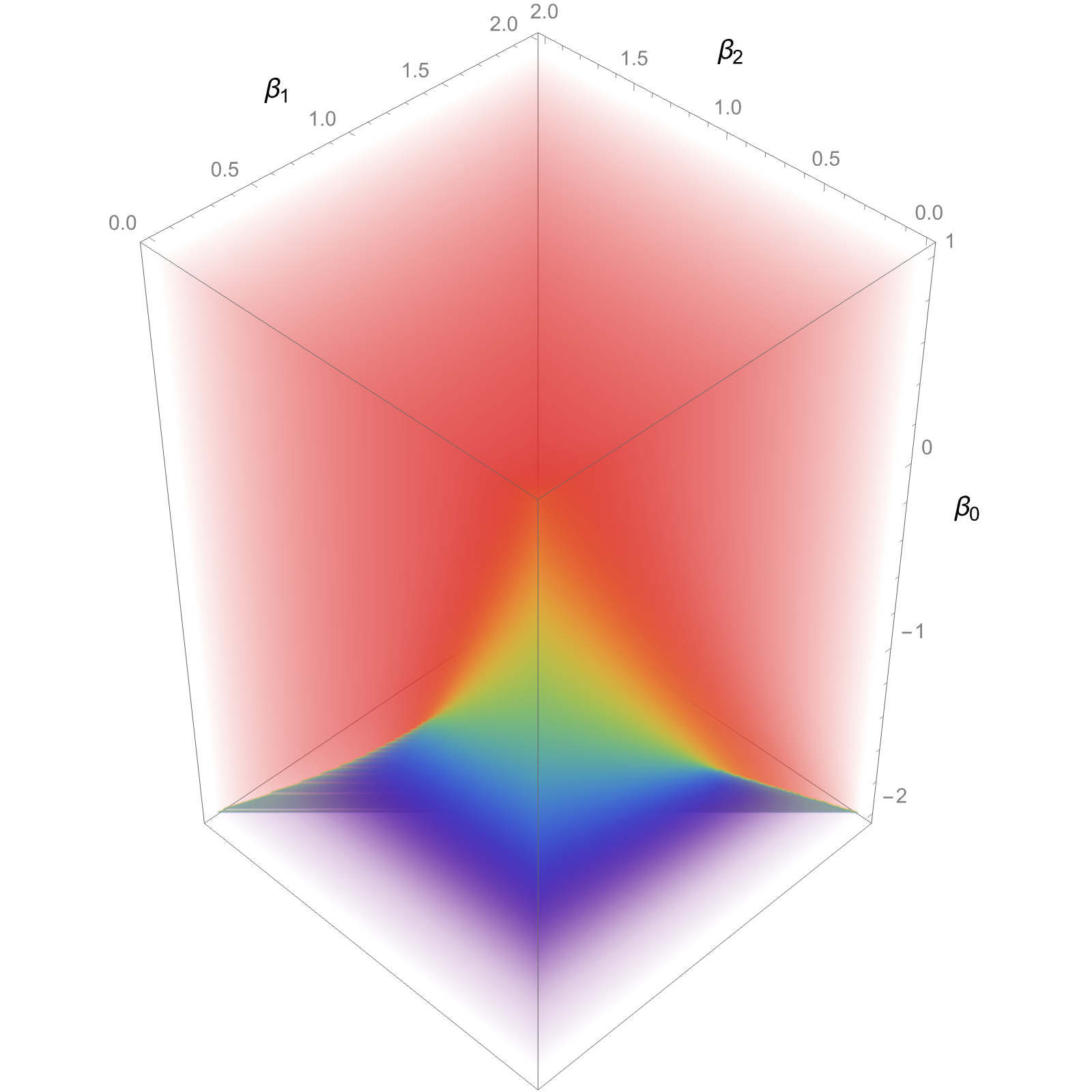}
    \hspace{5mm}
    \includegraphics[align=c,scale=0.15]{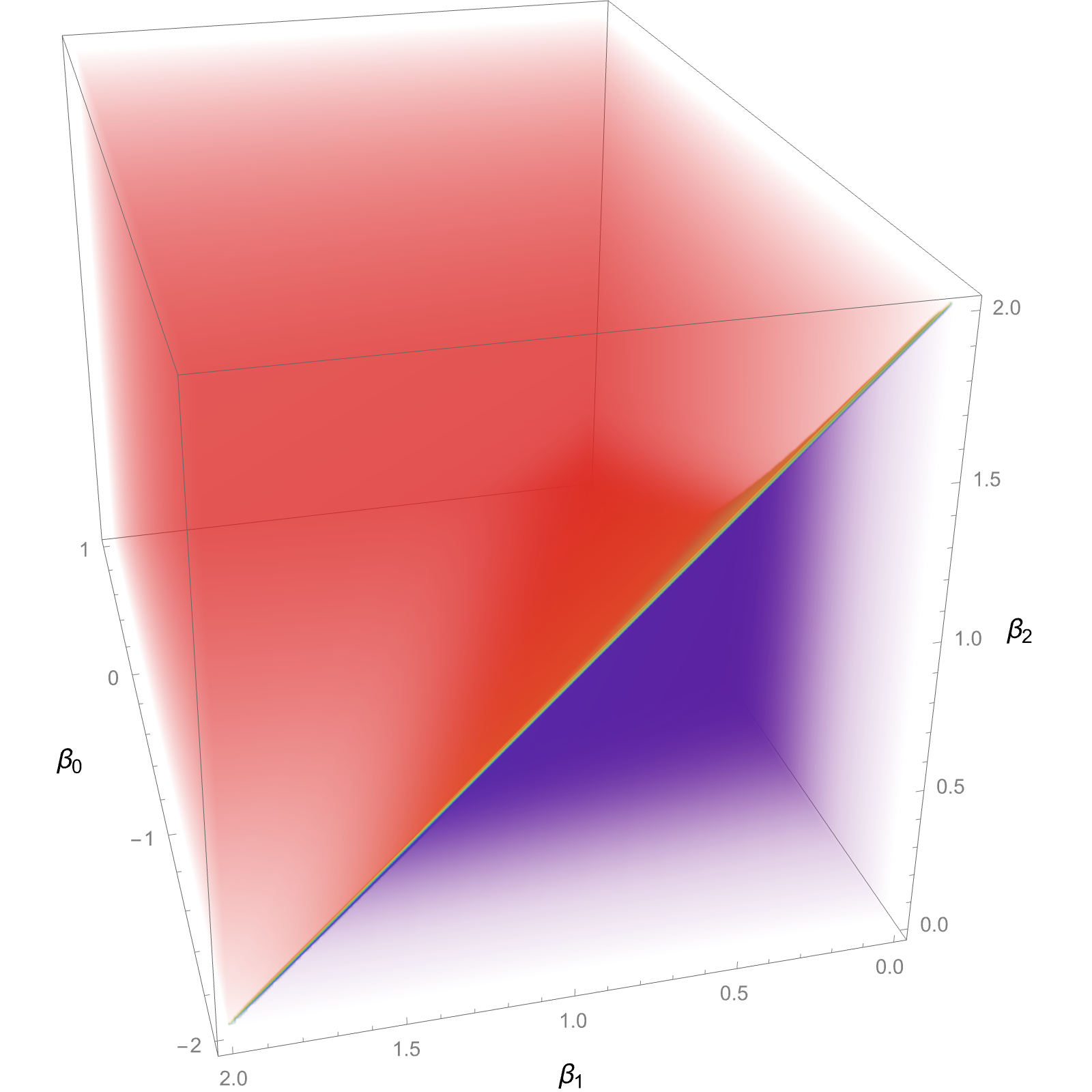}
    \hspace{5mm}
    \includegraphics[align=c,scale=0.4]{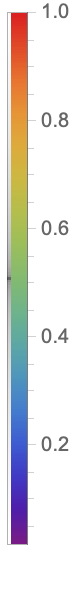}
\end{center}
\caption{Three views of a numerical plot of the maximizer $p^*$ of $L_{\bm\beta}$, as a function of $\beta_0 \in [-2,1]$ and
$\beta_1, \beta_2 \in [0,2]$. The graphs chosen for this ERGM are $G_1 = $ two-star (or disjoint union of two edges),
and $G_2 = $ triangle (or any other graph on $3$ edges, as all give the same values of $p^*$).
The opacity of $p^*$ values near $0$ and $1$ has been decreased so that the interior detail can be seen:
notice the two-dimensional surface across which there is a first-order phase transition.}
\label{fig:ergm_pstar}
\end{figure}

For reasons which will become clear in Section \ref{sec:review} below, the behavior of the ERGM depends not just on the global maximizers
but also the local maximizers of $L_{\bm\beta}$.
We will say that $\bm\beta$ lies in the \emph{subcritical} or \emph{high-temperature} regime of parameters if $L_{\bm\beta}$
has a unique local maximizer $p^* \in [0,1]$, and if $L_{\bm\beta}''(p^*) < 0$ (see Figure \ref{fig:ergm_regimes}, top row).
On the other hand, $\bm\beta$ is said to be in the \emph{supercritical} or \emph{low-temperature} regime if $L_{\bm\beta}$ has
at least two local maximizers $p^* \in [0,1]$ with $L_{\bm\beta}''(p^*) < 0$ (see Figure \ref{fig:ergm_regimes}, bottom row).
Values of $\bm\beta$ outside of these two regimes are said to be \emph{critical}.

\begin{figure}
\begin{center}
    \includegraphics[align=c,scale=0.15]{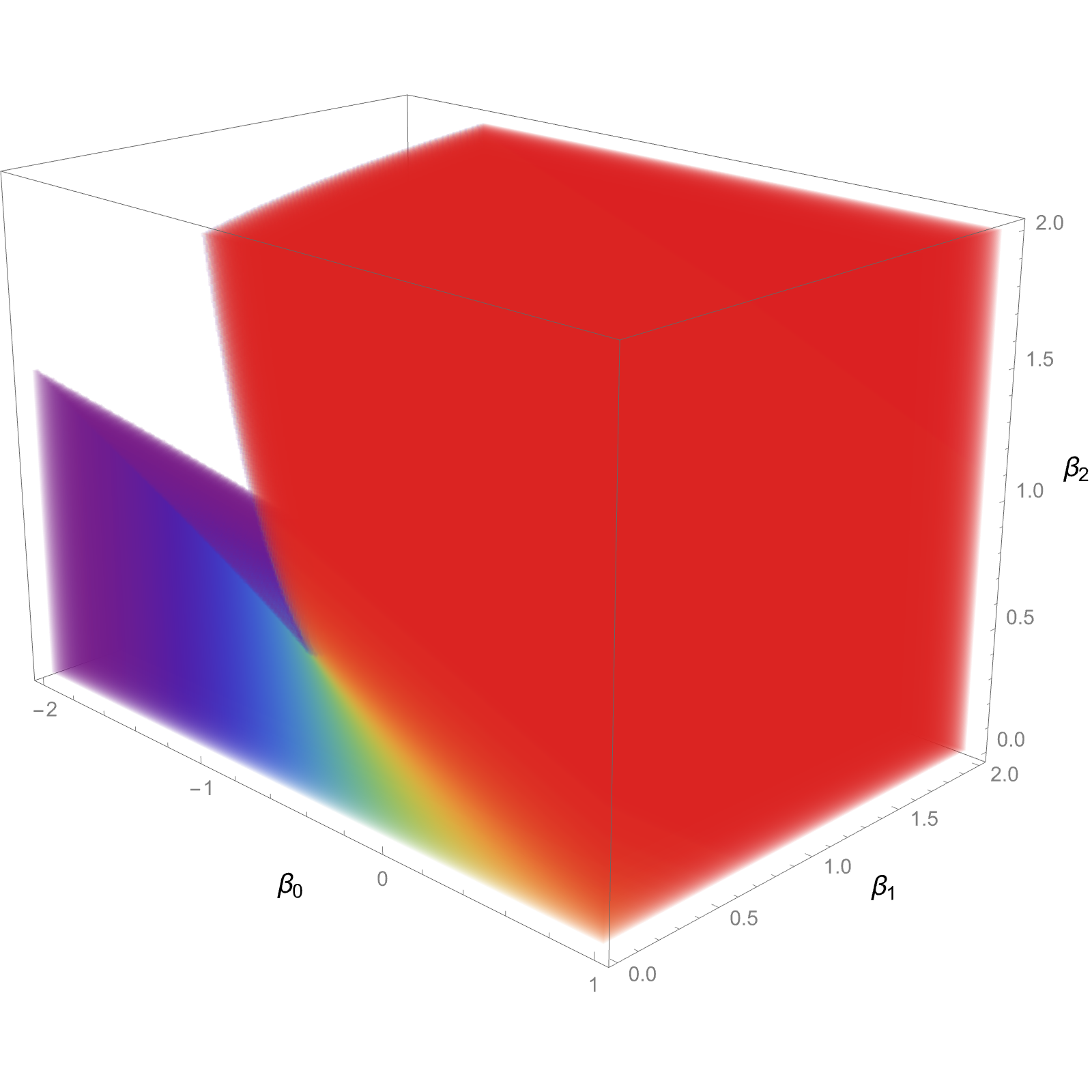}
    \hspace{5mm}
    \includegraphics[align=c,scale=0.15]{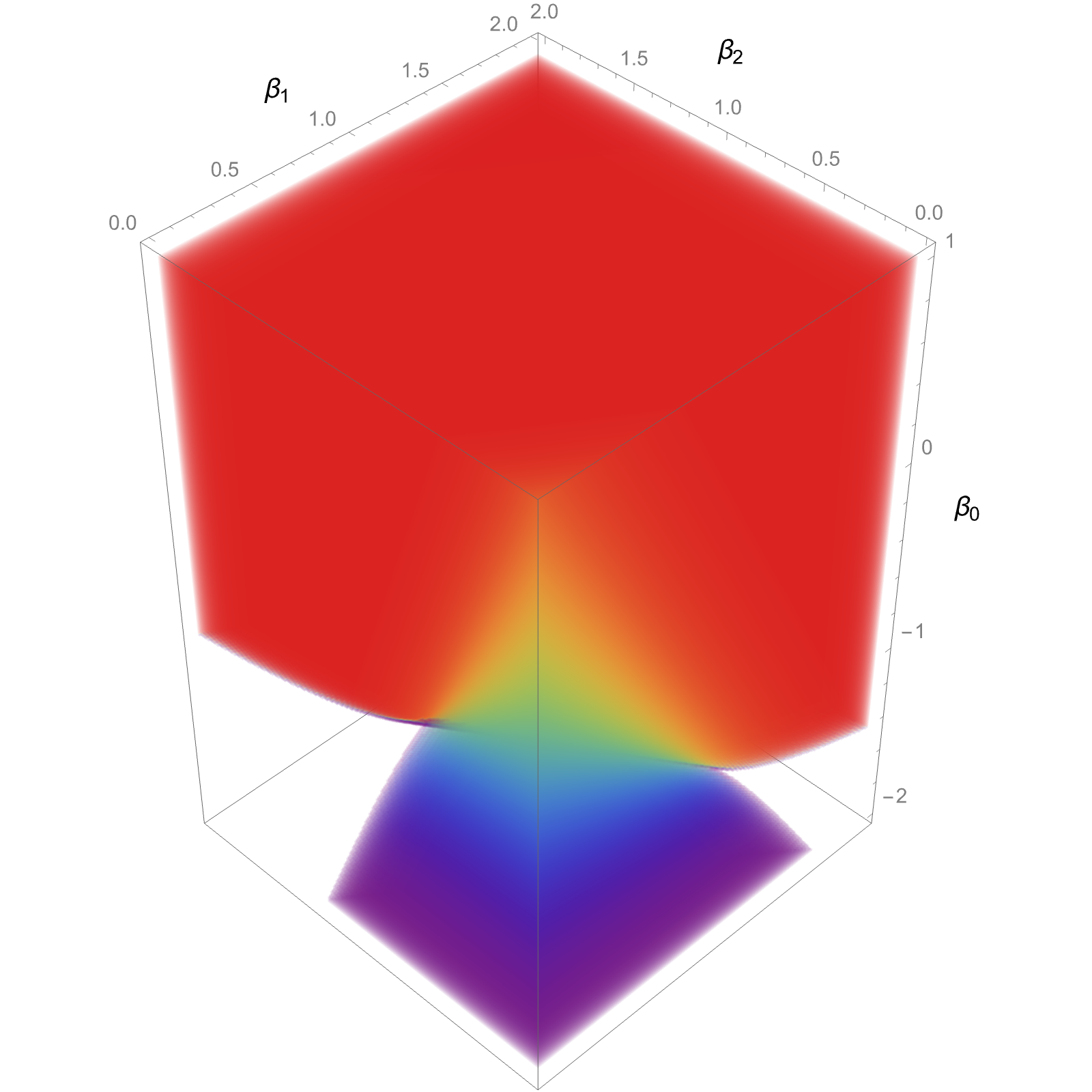}
    \hspace{5mm}
    \includegraphics[align=c,scale=0.15]{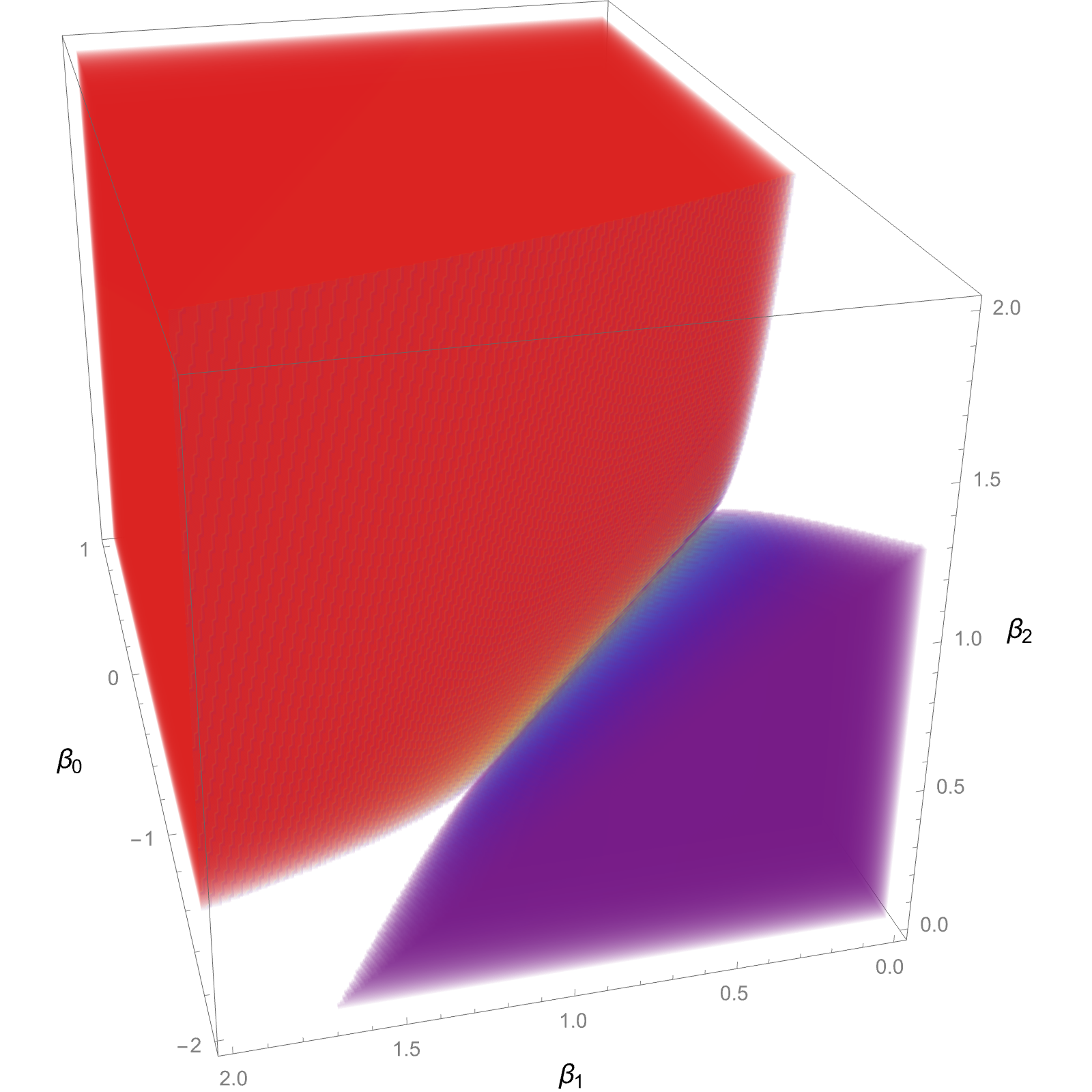}
    \hspace{5mm}
    \includegraphics[align=c,scale=0.4]{figures/ergm/colorbar.png}
\end{center}
\begin{center}
    \includegraphics[align=c,scale=0.15]{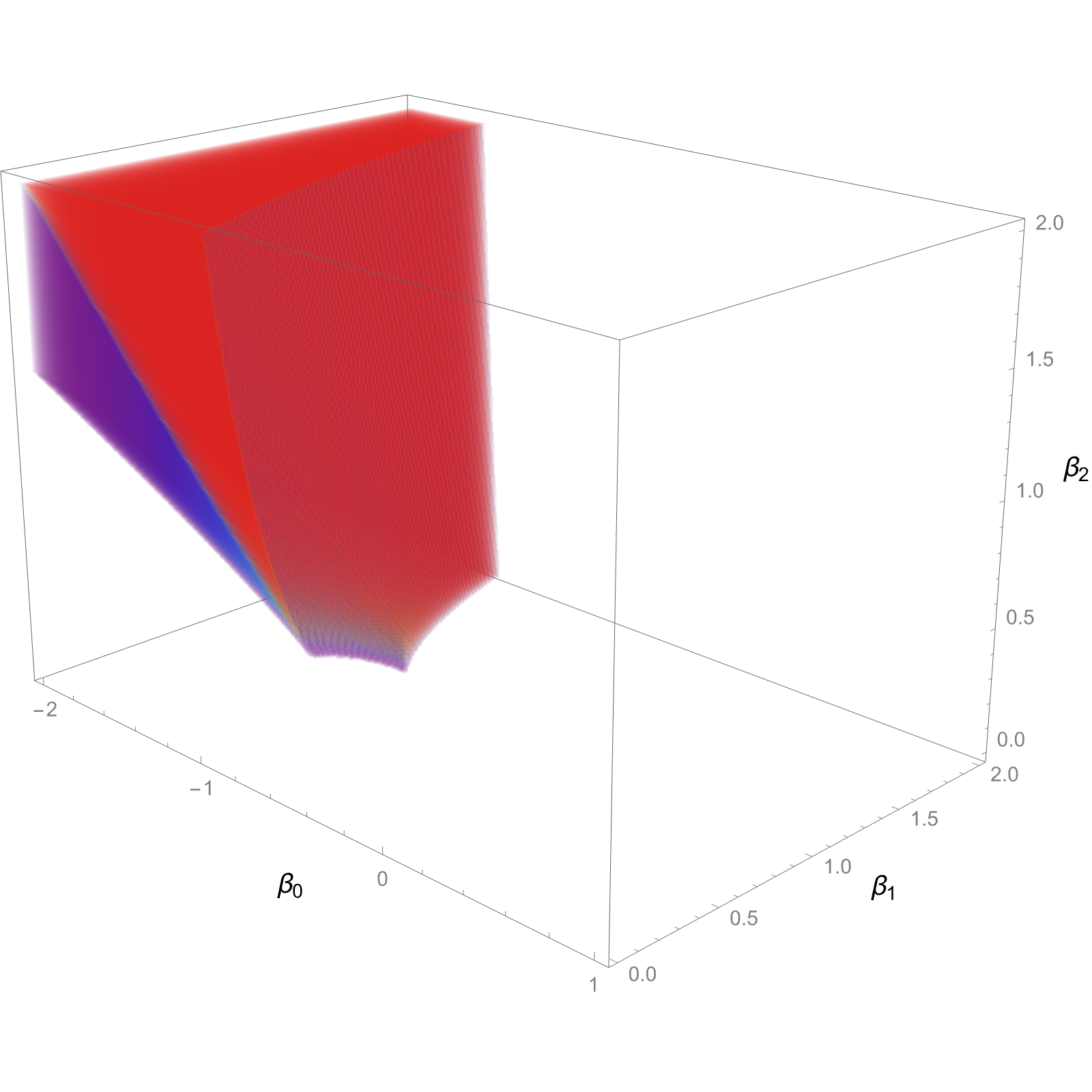}
    \hspace{5mm}
    \includegraphics[align=c,scale=0.15]{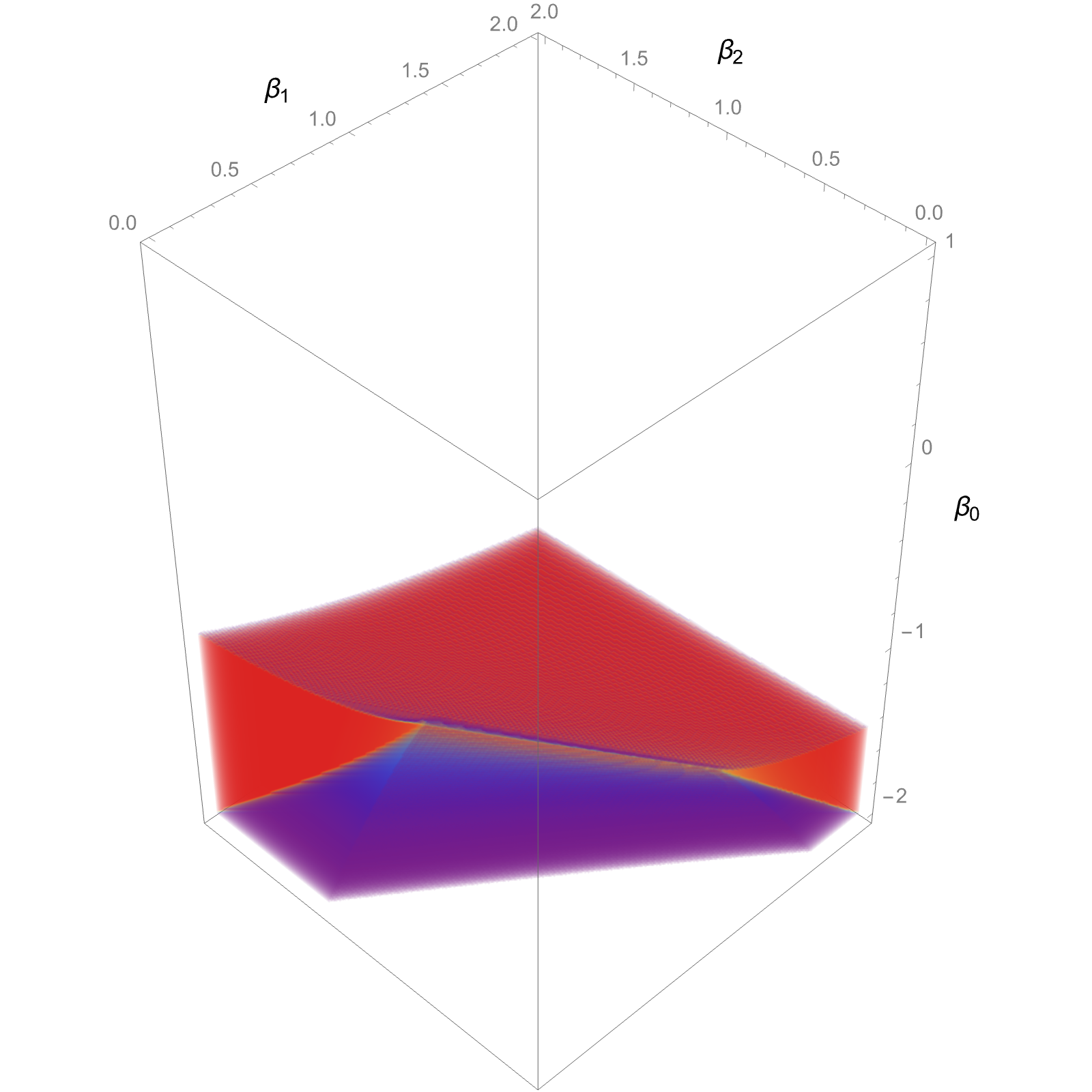}
    \hspace{5mm}
    \includegraphics[align=c,scale=0.15]{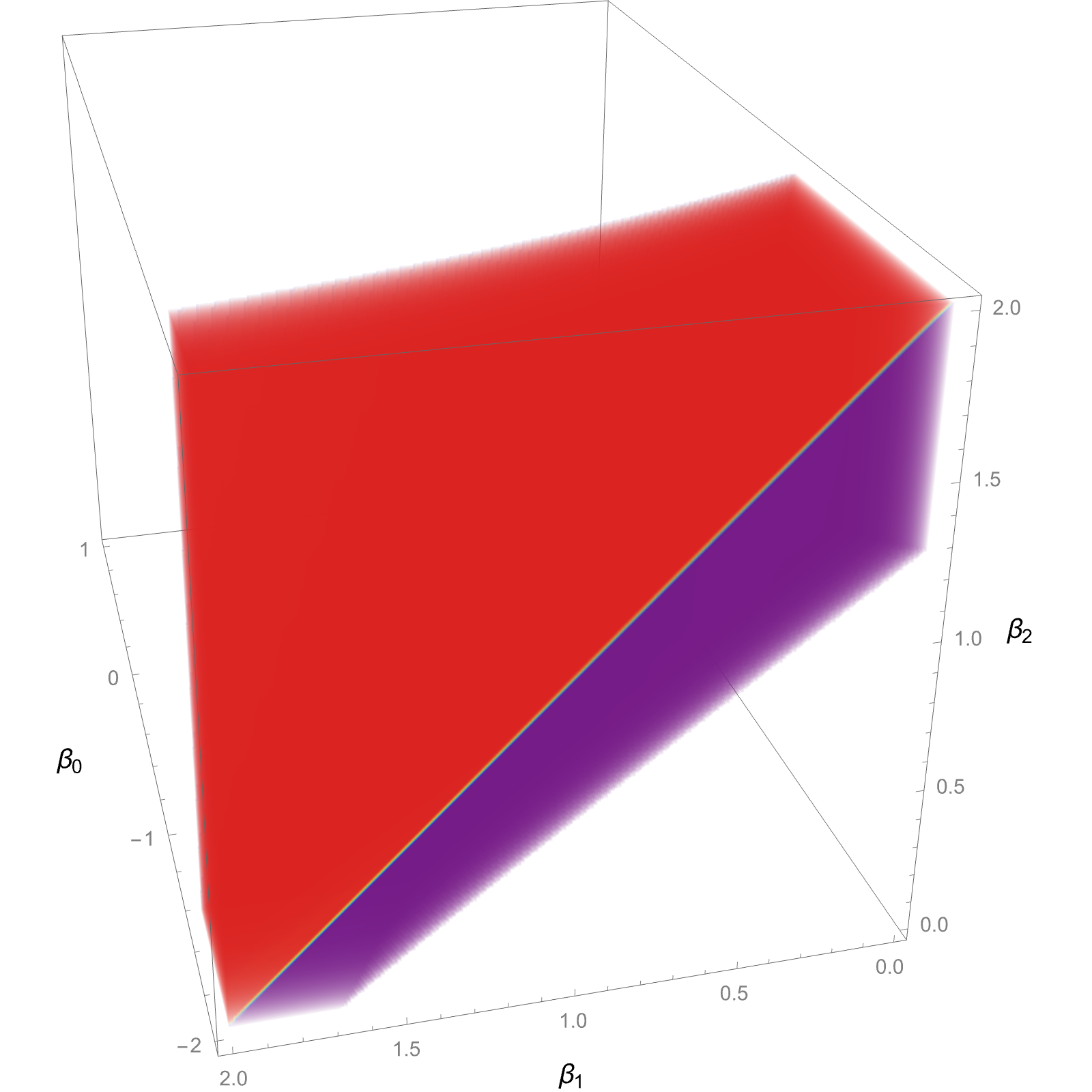}
    \hspace{5mm}
    \includegraphics[align=c,scale=0.4]{figures/ergm/colorbar.png}
\end{center}
\caption{Three views of the two parameter regimes in the ERGM with the same specification as in Figure \ref{fig:ergm_pstar}.
Top row: subcritical regime. Bottom row: supercritical regime.}
\label{fig:ergm_regimes}
\end{figure}

The large deviations principle thus states that in the subcritical regime, the ERGM is close in cut distance to
$\cG(n,p^*)$ for some \emph{unique} $p^* \in [0,1]$, while in the supercritical regime it can be expressed as a finite mixture of measures
which are each close to some Erd\H{o}s--R\'enyi model.
On first sight, this might lead one to believe that the ferromagnetic ERGM is somehow trivial, especially in the subcritical regime.
However, the cut distance is quite coarse, and only captures the behavior of the ERGM to first order.
At the microscopic scale, the ERGM is very different from the Erd\H{o}s--R\'enyi random graph: specifically,
in \cite{mukherjee2023statistics,bresler2018optimal}, it was shown that the total variation distance between the two distributions tends to $1$
in the case where $K=1$ and $G_1$ is a two-star, and this is expected to hold more generally as well.

This leads one to investigate the microscopic structure of the ERGM in more detail, and ask how it compares to
the Erd\H{o}s--R\'enyi random graph in a sense coarser than the total variation distance but finer than the cut distance,
such as the $1$-(Hamming)-Wasserstein distance.
This was first done in \cite{reinert2019approximating}, where it was found that under the optimal coupling between the
two distributions, the subcritical ferromagnetic ERGM differs from the Erd\H{o}s--R\'enyi random graph $\cG(n,p^*)$
by at most $O(n^{3/2})$ edges.
A key input for this result was the rapid mixing of the Glauber dynamics for the subcritical ERGM, as proved in
\cite{bhamidi2008mixing}.

Another consequence of the rapid mixing result of \cite{bhamidi2008mixing} is a powerful Gaussian concentration
inequality for Lipschitz observables of the subcritical ERGM due to \cite{ganguly2024sub}.
In that work, the rapid mixing was combined with a technique from \cite{chatterjee2005concentration} used to
derive such concentration inequalities from various markers of a rapidly mixing chain with the correct stationary distribution.
Applications of the concentration result of \cite{ganguly2024sub} include a central limit theorem for the edge count
in certain sparse collections of possible edges, as well as an alternative proof of the previously mentioned 
Wasserstein distance bound of \cite{reinert2019approximating} (although the bound obtained is worse by a factor of
$\sqrt{\log n}$).
\subsection{Decomposing low-temperature models into metastable wells}
\label{sec:intro_lowtemp}

In the supercritical regime, the ERGM is in general \emph{not} well-approximated by any particular Erd\H{o}s--R\'enyi model, as
there can be phase coexistence which leads to a drastic failure of independence.
Moreover, each of the multiple local maximizers of $L_{\bm\beta}$ gives rise to a distinct \emph{metastable well}
for the Glauber dynamics---the dynamics will take an exponentially long time to escape a small cut-distance ball around $W_{p^*}$
whenever $p^*$ is a local maximizer of $L_{\bm\beta}$ \cite{bhamidi2008mixing}.
This prevents the \emph{global} rapid mixing of the supercritical ERGM Glauber dynamics, as well as a direct application 
of the technique of \cite{chatterjee2005concentration},
necessitating a deeper study of the energy landscape of the ERGM.

Such analysis of the energy landscape of low-temperature models has garnered much recent interest, and has been carried out
in settings such as the lattice Ising model as well as partially in some disordered models such as spin glasses
\cite{gheissari2022low,gamarnik2021overlap,huang2024weak,sellke2024threshold}.
These analyses often lead to improved sampling algorithms via sampling in stages, and can also lead to a better
understanding of why other sampling algorithms fail.
In the case of the supercritical ERGM, it was proved in \cite{bresler2024metastable} via the cavity method that the Glauber dynamics
exhibits metastable mixing, a form of rapid mixing \emph{within} a metastable well when initialized with a ``warm start'' from the
corresponding Erd\H{o}s--R\'enyi measure.

Understanding the energy landscape in a low-temperature model can be a challenging endeavor, as it often
requires specialized analysis that may not easily generalize to other models.
However, one common theme is a decomposition of the state space into pockets corresponding to the metastable wells of natural dynamical
versions of the model, wherein high-temperature behavior can be observed \cite{helmuth2019algorithmic,pirogov1975phase,subag2018free}.
This has led to a folklore belief that such decompositions abound and offer a natural way to understand many low-temperature systems
as mixtures of systems with high-temperature behavior.

In the present work, we formalize this belief in the setting of the supercritical ERGM, through the lens of concentration of measure.
Specifically, we extend the concentration result of \cite{ganguly2024sub} from the subcritical regime to \emph{each relevant metastable well}
in the supercritical regime, using as an input the metastable mixing result of \cite{bresler2024metastable}, combined with a new modification
of the technique of \cite{chatterjee2005concentration} suited for metastable mixing, which may be of independent interest.
We also make use of a result of \cite{barbour2022long} which patches a gap left by the aforementioned technique.
We apply this novel concentration result to gain a better understanding of the structure of the supercritical ERGM, extending results
first derived in the subcritical regime in \cite{ganguly2024sub}.
Specifically, we show that, when conditioned on a metastable well (cut-norm ball) around $W_{p^*}$,
certain $n$-vertex ERGM measures can be coupled to $\cG(n,p^*)$ such that samples differ by at most $O(n^{3/2} \sqrt{\log n})$ edges
in expectation.
Additionally, we derive a central limit theorem for certain small enough subcollections of edges in any supercritical ERGM conditioned
on a metastable well.
\subsection{Statement of results and some proof ideas}
\label{sec:intro_results}

Our results concern the ERGM measure conditioned on a metastable well, i.e.\ a small enough cut distance ball around a
constant graphon $W_{p^*}$, for some $p^* \in U_{\bm\beta}$,
the set of all global maximizers of the function $L_{\bm\beta}$ for which the second derivative
$L_{\bm\beta}''(p^*)$ is strictly negative.
Throughout this section we informally let $\mu$ denote such a conditioned ERGM measure; the precise definition may
depend on the particular result being considered and will be specified in all formal statements after this section.

\subsubsection{Concentration inequality}

In order to state our main result, let us introduce some notation that will be used throughout the article.
Let $[n] = \{1,\dotsc,n\}$, which will be the vertex set of graphs sampled from $\mu$.
These graphs can be represented by elements in $\{0,1\}^{\binom{[n]}{2}}$, where $\binom{[n]}{2}$ is the edge
set of the complete graph $K_n$, namely all size-$2$ subsets of $[n]$.
For a function $f : \{0,1\}^{\binom{[n]}{2}} \to \R$ and a vector $v \in \R_{\geq 0}^{\binom{[n]}{2}}$,
we say that $f$ is $v$-Lipschitz if $|f(x) - f(y)| \leq v_e$ whenever $x$ and $y$ differ only at index $e \in \binom{[n]}{2}$.

The main theorem of this work is the following concentration inequality for $v$-Lipschitz observables under $\mu$.
It gives a Gaussian tail bound at fluctuation scale $n \| v \|_\infty$ in all cases, as well as a tighter
bound $\sqrt{\| v \|_1 \| v \|_\infty}$ in cases where the former bound is suboptimal, which will be essential
in our applications.

\begin{theorem}[Informal, see Theorem \ref{thm:lipschitzconcentration}]
\label{thm:lipschitzconcentration_informal}
There are some constants $C, c > 0$ depending only on the ERGM specification and the choice of $p \in U_{\bm\beta}$
such that, for all $v$-Lipschitz $f$ and all $\lambda \geq 0$ satisfying $\lambda \leq c \| v \|_1$,
\begin{equation}
	\mu[| f - \E_\mu[f] | > \lambda] \leq 2 \Exp{
	- c \cdot \max \left\{
		\fr{\lambda^2}{n^2 \| v \|_\infty^2},
		\fr{\lambda^2}{\| v \|_1 \| v \|_\infty}
		- \Exp{\fr{C\lambda}{\| v \|_\infty} - c n}
		\right\}
	} + e^{-c n}.
\end{equation}
\end{theorem}

The two expressions inside the maximum in the above bound arise from two different general results which will be presented
in Section \ref{sec:concentration} below, one which is a modification of a theorem of \cite{chatterjee2005concentration},
and the other of which is due to \cite{barbour2022long}.
Both results may be phrased as translating certain markers of metastable mixing into a concentration inequality,
but their proofs are quite different and each one has its own benefits and drawbacks.
We emphasize that the ``markers of metastable mixing'' just alluded to are rather subtle and a large portion of the technical
work in the present article goes towards setting up a proper framework for the application of these results.

Let us now comment on the difference between the two expressions in the maximum above.
The first, $\frac{\lambda^2}{n^2 \| v \|_\infty^2}$, comes from the method of \cite{barbour2022long} and is optimal for
\emph{global} observables such as the overall edge count, which has $\| v \|_\infty = 1$.
Indeed, in the case of the Erd\H{o}s--R\'enyi graph, the overall edge count does fluctuate at scale $n$ with Gaussian tails.
However, for more \emph{local} observables, such as the degree of a single vertex, this quantity will give a suboptimal bound,
as we still have $\| v \|_\infty = 1$ but the degree should fluctuate only at scale $\sqrt{n}$.
This is where the second expression,
\begin{equation}
	\frac{\lambda^2}{\| v \|_1 \| v \|_\infty} - \Exp{\frac{C \lambda}{\| v \|_\infty} - cn},
\end{equation}
comes in; this comes from the modification of the result of \cite{chatterjee2005concentration}.
By replacing $n^2 \| v \|_\infty^2$ with $\| v \|_1 \| v \|_\infty$, this second expression allows us to capture the optimal
fluctuation scale for more local observables such as the degree.
However, the subtracted exponential term can ruin the bound if $\lambda \gtrsim n \| v \|_\infty$,
meaning the second expression gives a vacuous bound if the true fluctuation scale is of order $n \| v \|_\infty$.
Luckily, this is exactly the case where the first expression $\frac{\lambda^2}{n^2 \| v \|_\infty^2}$ is optimal,
and so Theorem \ref{thm:lipschitzconcentration_informal} is effective for all observables,
and gives a sufficiently tight rate for the observables which arise in our applications.

We remark that one should expect fluctuations at scale $\| v \|_2$, as is the case for for $v$-Lipschitz
functions of i.i.d.\ variables; our methods do not seem to be able to get this more optimal scale in general.
However, many quantities of interest have ``flat'' Lipchitz vectors with, for instance, $v_e \in \{0,1\}$ for all $e \in \edgeset$,
and in this case we have $\| v \|_1 \| v \|_\infty = \| v \|_2^2$.
Even if the Lipschitz vector is not completely flat, these two rates are often comparable, leading to an optimal (up to constants)
bound via Theorem \ref{thm:lipschitzconcentration_informal}.

Finally, the external $e^{-c n}$ term arises from the possibility that metastable regions exist \emph{within}
the cut distance ball around $p^* \in U_{\bm\beta}$.
For instance, see \cite[Theorem 3.3]{bresler2024metastable}, which exhibits a metastable region of size $e^{-\Theta(n)}$ that
consists of graphs which have the correct edge density everywhere except for \emph{around one particular vertex}.
It should be noted that this anomalous behavior seems like it could be controlled, using for example \cite[Theorem 8.9]{bresler2024metastable}
to remove the external $e^{-c n}$ term, but we do not pursue this presently
as in most applications it just suffices to get an exponentially small upper bound.

\subsubsection{Applications of Theorem \ref{thm:lipschitzconcentration_informal}}

As a consequence of Theorem \ref{thm:lipschitzconcentration_informal}, we can extract various structural information about the ERGM
conditioned on one of these metastable wells.
These results extend \cite[Theorems 2 and 3]{ganguly2024sub} from the subcritical regime to the supercritical regime.
The first result gives an upper bound on the $1$-Wasserstein distance between $\mu$ and $\cG(n,p^*)$, which is the minimal
expected number of differing edges between two samples from these two distributions, under the optimal coupling.
Recall that $M_{\bm\beta}$ denotes the set of all global maximizers of $L_{\bm\beta}$, with no condition on the second derivative.

\begin{theorem}[Informal, see Theorem \ref{thm:wasserstein}]
\label{thm:wasserstein_informal}
Suppose $|M_{\bm\beta}| = 1$, or that all graphs $G_0,\dotsc,G_K$ in the specification of the ERGM are forests.
Then the $1$-Wasserstein distance between $\mu$ and $\cG(n,p^*)$ is $\lesssim n^{3/2} \sqrt{\log n}$.
\end{theorem}

The somewhat strange hypothesis of this theorem is an artifact of our proof strategy, which is adapted from that of \cite{ganguly2024sub}.
In that work, a key input to the corresponding result was the FKG inequality, which does not hold for the conditioned measure $\mu$.
This is an issue especially in the phase coexistence regime, when $|M_{\bm\beta}| > 1$.
As such, in the present work we introduce a new technique to get around the use of the FKG inequality in the phase coexistence regime,
but we must introduce the second hypothesis, that none of the specifying graphs have cycles, in order to derive the Wasserstein distance bound.
Note that in follow-up work by Mukherjee and the author \cite{fkg}, an approximate FKG inequality will be proved allowing for this condition
to be removed.
Nevertheless, we view the workaround found in the present article as an interesting contribution due to the novel technique it introduces;
see the proof of Proposition \ref{prop:Eproduct} for details.

We remark that it is also possible to remove this condition 
by instead using an argument similar to the proof of \cite[Theorem 1.13]{reinert2019approximating}, which would have
the added benefit of removing the $\sqrt{\log n}$ factor.
Tools developed in the current work, in particular Proposition \ref{prop:goodset} below, would be helpful for this argument.
While we do not pursue this presently, follow-up work by the author \cite{was} yields the same result, using a different strategy
which actually uses Theorem \ref{thm:wasserstein_informal} as an input to a bootstrapping argument.

The second application we present is a central limit theorem for the edge counts in small subcollections of edges.
Fortunately, our new technique allows us to completely bypass the FKG inputs used in \cite{ganguly2024sub} for this theorem,
and we don't require any additional hypotheses beyond what were used in that work; in fact we are able to show slightly more.

\begin{theorem}[Informal, see Theorem \ref{thm:clt}]
\label{thm:clt_informal}
Let $S_n(x)$ denote the edge count within a growing but sublinear (in $n$) collection of edges in $\binom{[n]}{2}$ which are either vertex-disjoint
or all incident on a single vertex.
Then $S_n(X)$ obeys a central limit theorem under $X \sim \mu$, as $n \to \infty$.
\end{theorem}

The condition that the edges must be vertex-disjoint was present in \cite{ganguly2024sub} for the subcritical regime,
but the alternative condition that all edges be incident on a single vertex was not covered there, although essentially the same
proof goes through.
Recently, it was shown in \cite{fang2024normal} that these conditions can be removed in the perturbative \emph{Dobrushin uniqueness regime},
which should be thought of as a very high temperature regime.
There, it was shown using Stein's method that in fact the full edge count has a central limit theorem.
In addition, after the first version of the present article was posted on arXiv, \cite{fang2024normal} was updated to cover the full
subcritical regime, and this was later extended to metastable wells in the supercritical regime by the author \cite{winstein2025quantitative}.
We also mention an earlier work \cite{mukherjee2023statistics}, which characterizes the limiting distribution of the edge count
in \emph{all} parameter regimes for the \emph{two-star} ERGM, i.e.\ in the case where $K = 1$ and $G_1$ is a two-star.
However, their argument is specialized to this case as it has only pairwise interactions and so can be compared
with an appropriate Ising model.

\subsubsection{Some proof ideas}

In order to prove Theorem \ref{thm:lipschitzconcentration_informal}, we will make use of two general results,
Theorem \ref{thm:chatterjee} and Corollary \ref{cor:barbour} below, which translate certain markers of metastable mixing into 
concentration inequalities.
Corollary \ref{cor:barbour} follows directly from \cite[Theorem 2.1]{barbour2022long} and is useful when
considering global observables, whereas Theorem \ref{thm:chatterjee} is a novel adaptation of \cite[Theorem 3.3]{chatterjee2005concentration},
and is useful when considering local observables, as discussed below the statement of Theorem \ref{thm:lipschitzconcentration_informal}.
As the statements of these results are rather long, we do not state informal versions here,
but instead give a short discussion in the next few paragraphs.

Suppose we have some function $f$ of a random variable $X$, and there is a Markov chain $(X_t)$ on the state space which allows one to
obtain approximate samples of $X$.
If the Markov chain mixes rapidly, and the value of $f(X_t)$ does not change by much from one time step to the next, then it is intuitive
that multiple samples of $X$ should yield similar values of $f(X)$: in other words, $f(X)$ is concentrated.
This intuition was first made formal in the context of traditionally mixing Markov chains by \cite[Theorem 3.3]{chatterjee2005concentration}.

As opposed to traditional forms of mixing, metastable mixing often only supposes that the Markov chain has suitable mixing properties
within a metastable well when given a warm start, or when started from a specific collection of states in the well.
Moreover, as the mixing is only metastable, one cannot in general expect the error to decrease without end as in traditional mixing
(although in some special cases such as the $p$-spin Curie--Weiss model \cite{samanta2024mixing} this actually can be done).
These features are incorporated in \cite[Theorem 2.1]{barbour2022long}, which is well-suited for the metastable mixing setup but
unfortunately gives a suboptimal rate of concentration for local observables.
Our new result, Theorem \ref{thm:chatterjee}, solves this problem by modifying \cite[Theorem 3.3]{chatterjee2005concentration}
to work with metastable mixing, but this modification leads to its own drawbacks limiting the applicability for global observables.
The proof techniques for Theorem \ref{thm:chatterjee} and Corollary \ref{cor:barbour} are different, and it remains an interesting
problem to combine these techniques in a way which retains both benefits and eliminates both drawbacks.

As mentioned, we use the metastable mixing results of \cite{bresler2024metastable} as inputs for Theorem \ref{thm:chatterjee}
and Corollary \ref{cor:barbour} in order to prove Theorem \ref{thm:lipschitzconcentration}.
However, note that the results of \cite{bresler2024metastable} assume a \emph{warm start}, whereas our techniques
seem limited in general to the ``good starting set'' form of metastable mixing.
In the present setting, we are able to translate between the two, but this seemingly innocuous distinction actually leads to a rather
outsized portion of the technical work in proving Theorem \ref{thm:lipschitzconcentration_informal}, due to the difficulty of accessing
certain connectivity properties required for a path-coupling argument, which can be mostly overlooked in the warm start setting.
It is therefore an interesting question whether a result can be proved which is similar to Theorem \ref{thm:chatterjee}
or Corollary \ref{cor:barbour}, but which can be used directly with metastable mixing using a warm start.
This would alleviate some of the technicalities faced in the present work, allowing for easier and possibly wider applications.
\subsection{Outline of the paper}
\label{sec:intro_outline}

Before delving into the mathematical content of the article, in Section \ref{sec:simulations} we discuss a simulation study
of metastable wells in a variety of supercritical ERGMs.
This provides some context for our results and also leads to some interesting new questions.
We then begin the mathematical discussion in Section \ref{sec:review} with a review of the literature on the ERGM,
highlighting results which will be relevant for the proofs to follow.
In Section \ref{sec:concentration} we present the results deriving concentration from metastable mixing in general,
including our novel result, Theorem \ref{thm:chatterjee}, which we hope is of independent interest.
This does not rely on any special knowledge of the ERGM, so the reader only interested in the general concentration
result may skip directly to Section \ref{sec:concentration}.
Then, in Section \ref{sec:lipschitz}, we apply Theorem \ref{thm:chatterjee} as well as Corollary \ref{cor:barbour}
to Lipschitz observables of the ERGM and obtain our main theorem, Theorem \ref{thm:lipschitzconcentration_informal}.
Finally, in Section \ref{sec:applications} we derive Theorems \ref{thm:wasserstein_informal} and \ref{thm:clt_informal} as applications
of Theorem \ref{thm:lipschitzconcentration_informal}.
\subsection{Acknowledgements}
\label{sec:intro_acknowledgements}

I was partially supported by the NSF Graduate Research Fellowship grant DGE 2146752.
I would like to thank my advisor, Shirshendu Ganguly, for suggesting this project and for helpful discussions.
I would also like to thank Sourav Chatterjee, Persi Diaconis, Sumit Mukherjee, Charles Radin, and Rikhav Shah
for helpful feedback on an earlier draft of this article.
Finally, I would like to thank the anonymous referees for suggestions which greatly improved this article.
\section{Simulations}
\label{sec:simulations}

\def\trilo{$\color{blue}\vartriangle$}
\def\trihi{$\color{cyan}\blacktriangle$}
\def\tetlo{$\color{red}\square$}
\def\tethi{$\color{magenta}\blacksquare$}
\def\hexlo{$\color{green}\fullmoon$}
\def\hexhi{$\color{yellow}\newmoon$}

\begin{figure}
    \centering

    \pgfmathdeclarefunction{H}{1}{
    \pgfmathparse{\bZ * #1 ^ \eZ + \bO * #1 ^ \eO + \bT * #1 ^ \eT}%
    }
    \pgfmathdeclarefunction{I}{1}{
    \pgfmathparse{0.5 * #1 * ln(#1) + 0.5 * (1 - #1) * ln(1 - #1)}%
    }
    \pgfmathdeclarefunction{L}{1}{
    \pgfmathparse{H(#1) - I(#1)}%
    }

    \begin{tikzpicture}
    \begin{axis}[
        width=16cm,
        height=8cm,
        axis lines=middle,
        xlabel={$p$},
        ylabel={$L_\beta(p)$},
        domain=0:1,
        samples=500,
        thick,
        grid=both,
        minor tick num=1,
        grid style={gray!10},
        enlargelimits=false,
        clip=false,
        scaled ticks=false,
        yticklabel style={/pgf/number format/fixed}, 
        yticklabel style={/pgf/number format/precision=5},
        axis line style={-,gray},
        label style={gray},
        tick label style={gray},
    ]

    \def\bZ{-1.17}
    \def\bO{1.17}
    \def\bT{0.06}
    \def\eZ{1}
    \def\eO{2}
    \def\eT{32}
    \def\p{0.6}

    \addplot[green, thick] {L(x)};

    \def\bZ{-1}
    \def\bO{0.55}
    \def\bT{0.5}
    \def\eZ{1}
    \def\eO{2}
    \def\eT{3}
    \addplot[blue, thick] {L(x)};

    \def\p{0.18277}
    \addplot[mark=*, black, mark size=1pt] coordinates {(\p,{L(\p)})};
    \addplot[mark=text, text mark=\trilo] coordinates {(\p,{L(\p)+0.005})};

    \def\p{0.93653}
    \addplot[mark=*, black, mark size=1pt] coordinates {(\p,{L(\p)})};
    \addplot[mark=text, text mark=\trihi] coordinates {(\p,{L(\p)+0.005})};

    \def\bZ{-0.6}
    \def\bO{0.25}
    \def\bT{0.5}
    \def\eZ{1}
    \def\eO{2}
    \def\eT{6}
    \addplot[red, thick] {L(x)};

    \def\p{0.28937}
    \addplot[mark=*, black, mark size=1pt] coordinates {(\p,{L(\p)})};
    \addplot[mark=text, text mark=\tetlo] coordinates {(\p-0.01,{L(\p)+0.005})};
    \addplot[mark=text, text mark=\hexlo] coordinates {(\p+0.01,{L(\p)+0.005})};

    \def\p{0.99666}
    \addplot[mark=*, black, mark size=1pt] coordinates {(\p,{L(\p)})};
    \addplot[mark=text, text mark=\tethi] coordinates {(\p-0.01,{L(\p)+0.005})};
    \addplot[mark=text, text mark=\hexhi] coordinates {(\p+0.01,{L(\p)+0.005})};

    \end{axis}
    \end{tikzpicture}
    \caption{
    Plots of $L_{\bm\beta}$ for various ERGM specifications, as well as marks representing the datasets we consider
    in Section \ref{sec:simulations}.
    All specifications have $K=2$ and $G_0, G_1$ as an edge and a $2$-star (wedge).
    For the lower blue curve, $G_2$ is a triangle, and for the upper red curve, $G_2$ is either a tetrahedron or a hexagon,
    both of which yield the same $L_{\bm\beta}$.
    The local maxima have been highlighted and labeled with symbols corresponding to the datasets outlined in Table \ref{table:data}.
    For the lower blue curve, the global maximum is on the left, while for the upper red curve, the global maximum is on the right.
    We also remark that it is possible for $L_{\bm\beta}$ to have more than two local maxima; the green curve without any marked points
    demonstrates this, corresponding to the specification where $G_2$ is any graph with $32$ edges and we take $\bm\beta = (-1.17, 1.17, 0.06)$.
    However, due to computational constraints, we did not generate a dataset for this specification.
    }
    \label{fig:sim_L}
\end{figure}

In this section we supplement our mathematical results with a simulation study of metastable wells in various
supercritical ERGMs.
For all models we consider, we take $K=2$ and set the graphs $G_0$ and $G_1$ to be an edge and a two-star respectively.
The third graph will either be a triangle, a tetrahedron, or a hexagon.
We have selected parameters $\bm\beta$ which lead to two distinct local maxima for $L_{\bm\beta}$ in each case
(see Figure \ref{fig:sim_L}), and considered the model conditioned on both metastable wells.
Each $n$-vertex ERGM sample was generated by running Glauber dynamics for $n^3$ steps, which leads to a total variation
distance of $e^{-\Omega(n)}$ from the true ERGM distribution conditioned on a metastable well, by the results of
\cite{bhamidi2008mixing,bresler2024metastable}.
Note that due to technicalities in the hypotheses of \cite{bresler2024metastable}, this mixing bound has actually only
been proven in the case where we consider the \emph{global} maxima of $L_{\bm\beta}$, and our results in Section \ref{sec:intro_results}
are also only stated for such metastable wells, but we expect that everything should also hold for the smaller metastable wells.

Table \ref{table:data} describes the datasets we collected.
For each dataset, we considered graph sizes of the form $n = \lfloor b^s \rfloor$ for some base $b$
and some range of integer scales $s$.
As the efficiency of simulating Glauber dynamics depends on the graphs in the specification of an ERGM,
we considered slightly different bases and scale ranges for different graphs $G_2$.
For each value of $n = \lfloor b^s \rfloor$ obtained in this way, we collected a number of samples, either $128$ or $8192$.
We have characterized each metastable well specification (i.e.\ the set of graphs, $\bm\beta$, and $p^*$)
with a unique symbol in the plots;
these symbols also appear in Figure \ref{fig:sim_L} to pictorially represent each specification.
The first 6 datasets are used in Sections \ref{sec:simulations_concentration} and \ref{sec:simulations_wasserstein},
which are related to our Theorems \ref{thm:lipschitzconcentration_informal} and \ref{thm:wasserstein_informal} respectively.
Validating the CLT of Theorem \ref{thm:clt_informal} required more samples as will be discussed in Section \ref{sec:simulations_clt},
and we use the last two datasets for this purpose.

\begin{table}
\centering
\begin{tabular}{c|c|c|c|c|c|c|c|}
     & symbol & $G_2$ & $\bm\beta$ & $p^*$ & samples & base $b$ & scale $s$ range \\
    \hline
    dataset 1 & \trilo & triangle & $(-1,0.55,0.5)$ & $\approx 0.18299$ & $128$ & $\sqrt{2}$ & $6$--$16$ \\
    \hline
    dataset 2 & \trihi & triangle & $(-1,0.55,0.5)$ & $\approx 0.93653$ & $128$ & $\sqrt{2}$ & $11$--$16$ \\
    \hline
    dataset 3 & \tetlo & tetrahedron & $(-0.6,0.25,0.5)$ & $\approx 0.28937$ & $128$ & $\sqrt{2}$ & $6$--$15$ \\
    \hline
    dataset 4 & \tethi & tetrahedron & $(-0.6,0.25,0.5)$ & $\approx 0.99666$ & $128$ & $\sqrt{2}$ & $10$--$15$ \\
    \hline
    dataset 5 & \hexlo & hexagon & $(-0.6,0.25,0.5)$ & $\approx 0.28937$ & $128$ & $\sqrt[4]{2}$ & $12$--$25$ \\
    \hline
    dataset 6 & \hexhi & hexagon & $(-0.6,0.25,0.5)$ & $\approx 0.99666$ & $128$ & $\sqrt[4]{2}$ & $19$--$25$ \\
    \hline
    dataset 7 & \trilo & triangle & $(-1,0.55,0.5)$ & $\approx 0.18299$ & $8192$ & $\sqrt{2}$ & $6$--$14$ \\
    \hline
    dataset 8 & \tetlo & tetrahedron & $(-0.6,0.25,0.5)$ & $\approx 0.28937$ & $8192$ & $\sqrt{2}$ & $6$--$13$ \\
    \hline
\end{tabular}
\caption{descriptions of the datasets considered in this section.}
\label{table:data}
\end{table}

\subsection{Fluctuations of Lipschitz observables}
\label{sec:simulations_concentration}

To supplement Theorem \ref{thm:lipschitzconcentration_informal}, we calculated
a few Lipschitz observables of the samples in datasets 1 through 6.
We considered \emph{global} edge and triangle counts (the total number of edges or triangles in the graph), as well as
\emph{local} edge and triangle counts (the number of edges or triangles incident on a particular vertex).
Note that the local edge count is the same as the degree of a vertex.
As a proxy for the fluctuation scale of these observables, we calculated the sample standard deviation of the observed values;
the results are plotted in Figure \ref{fig:sim_concentration}, along with polynomial functions which match the data closely.
No attempt was made to find the true curves of best fit, we simply aim to demonstrate the power laws with reasonable constants.

\begin{figure}
\centering
\def\length{4.5cm}
\def\spacing{1mm}
\def\markscale{0.5}
    \begin{tikzpicture}
        \begin{loglogaxis}[
            title={global edge count},
            width=\length,
            height=\length,
            font=\small,
            grid=none
        ]
        \addplot[mark=text, text mark={\scalebox{\markscale}{\hexhi}}, only marks] table[col sep=comma] {data/small_hexagon_high/total_edge_count_stds.csv};
        \addplot[mark=text, text mark={\scalebox{\markscale}{\tethi}}, only marks] table[col sep=comma] {data/small_tetrahedron_high/total_edge_count_stds.csv};
        \addplot[mark=text, text mark={\scalebox{\markscale}{\trihi}}, only marks] table[col sep=comma] {data/small_triangle_high/total_edge_count_stds.csv};
        \addplot[mark=text, text mark={\scalebox{\markscale}{\hexlo}}, only marks] table[col sep=comma] {data/small_hexagon_low/total_edge_count_stds.csv};
        \addplot[mark=text, text mark={\scalebox{\markscale}{\tetlo}}, only marks] table[col sep=comma] {data/small_tetrahedron_low/total_edge_count_stds.csv};
        \addplot[mark=text, text mark={\scalebox{\markscale}{\trilo}}, only marks] table[col sep=comma] {data/small_triangle_low/total_edge_count_stds.csv};
        \addplot[
            thick,
            domain=5:300,
            samples=2
        ] {0.4*x^(1)}
        node[pos=0.6, above left] {$0.4 \cdot n$};
        \end{loglogaxis}
    \end{tikzpicture}
    \hspace{\spacing}
    \begin{tikzpicture}
        \begin{loglogaxis}[
            title={local edge count},
            width=\length,
            height=\length,
            font=\small,
            grid=none
        ]
        \addplot[mark=text, text mark={\scalebox{\markscale}{\hexhi}}, only marks] table[col sep=comma] {data/small_hexagon_high/edge_count_around_vertex_stds.csv};
        \addplot[mark=text, text mark={\scalebox{\markscale}{\tethi}}, only marks] table[col sep=comma] {data/small_tetrahedron_high/edge_count_around_vertex_stds.csv};
        \addplot[mark=text, text mark={\scalebox{\markscale}{\trihi}}, only marks] table[col sep=comma] {data/small_triangle_high/edge_count_around_vertex_stds.csv};
        \addplot[mark=text, text mark={\scalebox{\markscale}{\hexlo}}, only marks] table[col sep=comma] {data/small_hexagon_low/edge_count_around_vertex_stds.csv};
        \addplot[mark=text, text mark={\scalebox{\markscale}{\tetlo}}, only marks] table[col sep=comma] {data/small_tetrahedron_low/edge_count_around_vertex_stds.csv};
        \addplot[mark=text, text mark={\scalebox{\markscale}{\trilo}}, only marks] table[col sep=comma] {data/small_triangle_low/edge_count_around_vertex_stds.csv};
        \addplot[
            thick,
            domain=5:300,
            samples=2
        ] {0.5*x^(0.5)}
        node[pos=0.6, above left] {$0.5 \cdot n^{\frac{1}{2}}$};
        \end{loglogaxis}
    \end{tikzpicture}
    \hspace{\spacing}
    \begin{tikzpicture}
        \begin{loglogaxis}[
            title={global triangle count},
            width=\length,
            height=\length,
            font=\small,
            grid=none
        ]
        \addplot[mark=text, text mark={\scalebox{\markscale}{\hexhi}}, only marks] table[col sep=comma] {data/small_hexagon_high/total_triangle_count_stds.csv};
        \addplot[mark=text, text mark={\scalebox{\markscale}{\tethi}}, only marks] table[col sep=comma] {data/small_tetrahedron_high/total_triangle_count_stds.csv};
        \addplot[mark=text, text mark={\scalebox{\markscale}{\trihi}}, only marks] table[col sep=comma] {data/small_triangle_high/total_triangle_count_stds.csv};
        \addplot[mark=text, text mark={\scalebox{\markscale}{\hexlo}}, only marks] table[col sep=comma] {data/small_hexagon_low/total_triangle_count_stds.csv};
        \addplot[mark=text, text mark={\scalebox{\markscale}{\tetlo}}, only marks] table[col sep=comma] {data/small_tetrahedron_low/total_triangle_count_stds.csv};
        \addplot[mark=text, text mark={\scalebox{\markscale}{\trilo}}, only marks] table[col sep=comma] {data/small_triangle_low/total_triangle_count_stds.csv};
        \addplot[
            thick,
            domain=5:300,
            samples=2
        ] {0.02*x^(2)}
        node[pos=0.4, below right] {$0.02 \cdot n^2$};
        \end{loglogaxis}
    \end{tikzpicture}
    \hspace{\spacing}
    \begin{tikzpicture}
        \begin{loglogaxis}[
            title={local triangle count},
            width=\length,
            height=\length,
            font=\small,
            grid=none
        ]
        \addplot[mark=text, text mark={\scalebox{\markscale}{\hexhi}}, only marks] table[col sep=comma] {data/small_hexagon_high/triangle_count_around_vertex_stds.csv};
        \addplot[mark=text, text mark={\scalebox{\markscale}{\tethi}}, only marks] table[col sep=comma] {data/small_tetrahedron_high/triangle_count_around_vertex_stds.csv};
        \addplot[mark=text, text mark={\scalebox{\markscale}{\trihi}}, only marks] table[col sep=comma] {data/small_triangle_high/triangle_count_around_vertex_stds.csv};
        \addplot[mark=text, text mark={\scalebox{\markscale}{\hexlo}}, only marks] table[col sep=comma] {data/small_hexagon_low/triangle_count_around_vertex_stds.csv};
        \addplot[mark=text, text mark={\scalebox{\markscale}{\tetlo}}, only marks] table[col sep=comma] {data/small_tetrahedron_low/triangle_count_around_vertex_stds.csv};
        \addplot[mark=text, text mark={\scalebox{\markscale}{\trilo}}, only marks] table[col sep=comma] {data/small_triangle_low/triangle_count_around_vertex_stds.csv};
        \addplot[
            thick,
            domain=5:300,
            samples=2
        ] {0.03*x^(1.5)}
        node[pos=0.4, below right] {$0.03 \cdot n^{\frac{3}{2}}$};
        \end{loglogaxis}
    \end{tikzpicture}
    
\caption{Log-log plots of the \emph{standard deviations} of four Lipschitz observables of ERGM samples.
The power-law dependences agree with the guarantees of Theorem \ref{thm:lipschitzconcentration_informal}.
Note that the somewhat degenerate behavior of datasets 4 and 6 likely arises from the fact that $p^*$ is very close to $1$.
}
\label{fig:sim_concentration}
\end{figure}

Let us discuss how the fluctuation bounds guaranteed by Theorem \ref{thm:lipschitzconcentration_informal} compare with the
observed fluctuations.
For this, we need to calculate the norms of the various Lipschitz vectors, as Theorem \ref{thm:lipschitzconcentration_informal}
guarantees that a $v$-Lipschitz observable has standard deviation bounded by
$D \cdot \min \{ \sqrt{\| v \|_1 \| v \|_\infty}, n \| v \|_\infty \}$
for some global constant $D$ which depends on the specification of the ERGM, including the parameters $\bm\beta$
and the choice of metastable well $p^*$.

Although we do not attempt to give bounds on the constant $D$ in this article, it is interesting to note from our simulations
that whatever value it may take is not optimal for all Lipschitz observables simultaneously.
Indeed, let us just consider the global edge and triangle counts, which both are controlled via Theorem \ref{thm:lipschitzconcentration_informal}
to have standard deviation bounded by $D n \| v \|_\infty$.
For the global edge count this is $D n$, and for the global triangle count it is $D n(n-2) \approx D n^2$.
However, note that the constants appearing in the first and third plots are an order of magnitude apart from each other.

This discrepancy can be explained by noting that the removal of an edge typically \emph{does not} destroy the maximal number,
$n-2$, of triangles.
This is because the density of the graph is typically close to $p^*$, meaning that removing an edge will instead typically
destroy close to $(p^*)^2 n$ triangles.
Thus, the true fluctuations come from replacing $\| v \|_\infty$ by this true typical change.
The local edge and triangle counts differ by a similar ratio, which can also be understood using the same reasoning.
This phenomenon is utilized in various applications in follow-up work by the author \cite{winstein2025quantitative,was}.
\subsection{Coupling the ERGM and the Erd\H{o}s--R\'enyi model}
\label{sec:simulations_wasserstein}

Next, we turn to Theorem \ref{thm:wasserstein_informal}, the bound on the Wasserstein distance between $\mu$
and $\cG(n,p^*)$.
For this, again use datasets 1 through 6, but for each sample from the ERGM metastable well we have additionally
generated a sample from the corresponding Erd\H{o}s--R\'enyi model $\cG(n,p^*)$ by coupling each edge update in the Glauber
dynamics optimally (see Section \ref{sec:applications_wasserstein} for more details on this coupling).
Figure \ref{fig:sim_wass_graphs} shows a randomly selected pair of these samples from dataset 1 with $n=64$, highlighting the difference
between the two graphs.

\begin{figure}
\begin{center}
    \includegraphics[align=c,scale=0.18]{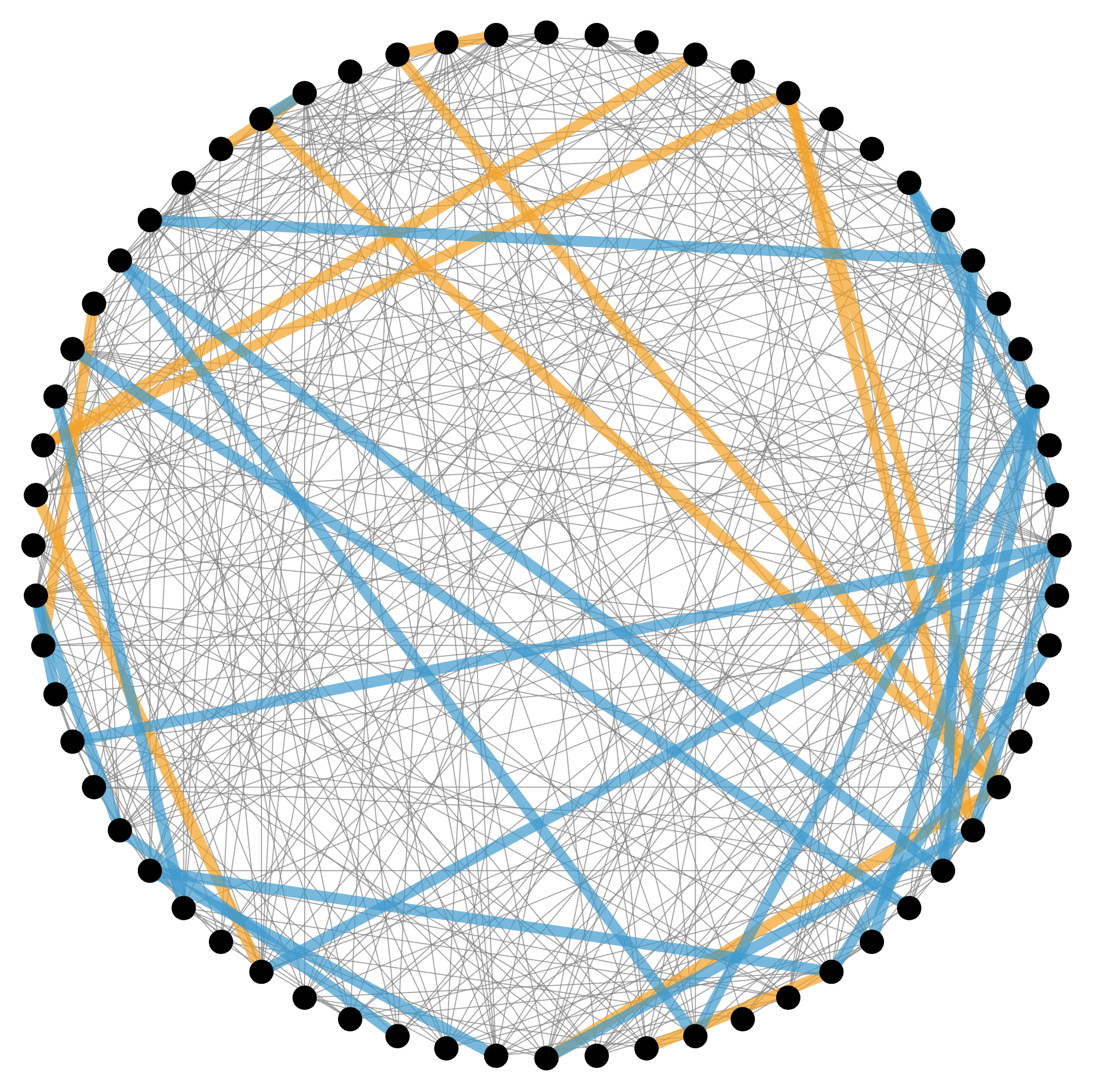}
    \hspace{15mm}
    \includegraphics[align=c,scale=0.18]{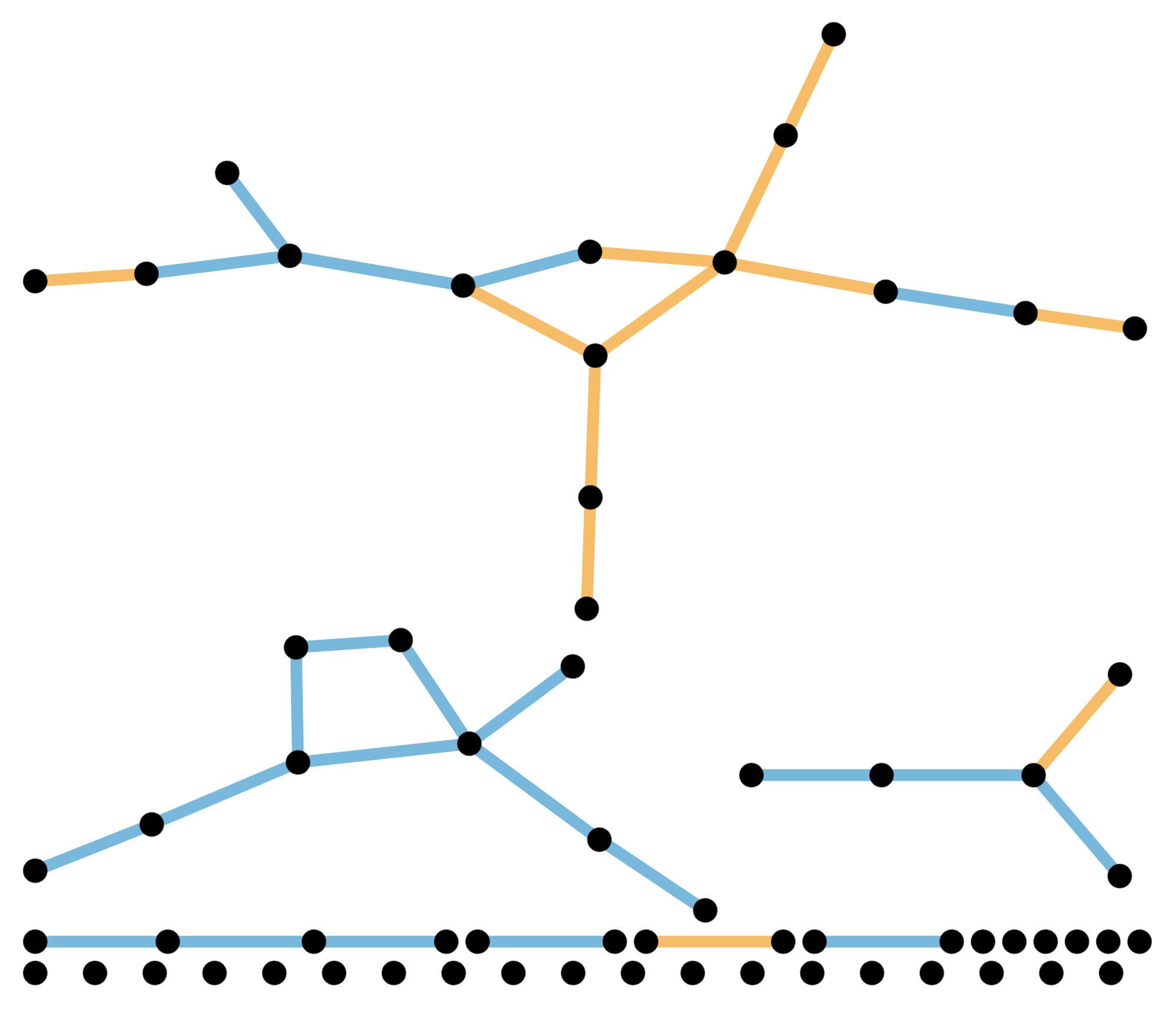}
\end{center}
\caption{Left: a $64$-vertex ERGM sample from dataset 1, overlaid with a coupled sample from $\cG(64,p^*)$.
Grey edges appear in both graphs, blue edges appear only in the ERGM, and orange edges appear only in the
Erd\H{o}s--R\'enyi graph.
Right: the graph structure of the discrepancy edges.
Notice that at first glance it appears more clustered than one might expect from a typical graph with this many vertices
and edges.}
\label{fig:sim_wass_graphs}
\end{figure}

If we let $X$ denote a sample from a metastable well in an ERGM, and let $Y$ denote the corresponding Erd\H{o}s--R\'enyi sample,
both thought of as elements of $\{0,1\}^{\edgeset}$, we can consider the \emph{signed difference graph} $Z = X - Y \in \{-1,0,1\}^{\edgeset}$.
The Hamming distance is exactly the sum $\sum_e |Z(e)|$.
We also consider the sum $\sum_e Z(e)$, which we term the \emph{signed discrepancy}, and which is just the difference in
edge count between $X$ and $Y$.
Finally, we also consider the \emph{difference graph} $|Z| \in \{0,1\}^{\edgeset}$ which is just the signed difference graph
with the signs erased, i.e.\ all $-1$s turned into $1$s.
In Figure \ref{fig:sim_wass_plots}, we plot the average Hamming distance and signed discrepancy, as well as the negative of the 
signed discrepancy, and finally the \emph{clustering coefficient} of the difference graph $|Z|$, which will be discussed below.

\begin{figure}
\centering
\def\length{4.5cm}
\def\spacing{1mm}
\def\markscale{0.5}
    \begin{tikzpicture}
        \begin{loglogaxis}[
            title={Hamming distance},
            width=\length,
            height=\length,
            font=\small,
            grid=none
        ]
        \addplot[mark=text, text mark={\scalebox{\markscale}{\hexhi}}, only marks] table[col sep=comma] {data/small_hexagon_high/hamming_distance_means.csv};
        \addplot[mark=text, text mark={\scalebox{\markscale}{\tethi}}, only marks] table[col sep=comma] {data/small_tetrahedron_high/hamming_distance_means.csv};
        \addplot[mark=text, text mark={\scalebox{\markscale}{\trihi}}, only marks] table[col sep=comma] {data/small_triangle_high/hamming_distance_means.csv};
        \addplot[mark=text, text mark={\scalebox{\markscale}{\hexlo}}, only marks] table[col sep=comma] {data/small_hexagon_low/hamming_distance_means.csv};
        \addplot[mark=text, text mark={\scalebox{\markscale}{\tetlo}}, only marks] table[col sep=comma] {data/small_tetrahedron_low/hamming_distance_means.csv};
        \addplot[mark=text, text mark={\scalebox{\markscale}{\trilo}}, only marks] table[col sep=comma] {data/small_triangle_low/hamming_distance_means.csv};
        \addplot[
            thick,
            domain=5:300,
            samples=2
        ] {0.05*x^(1.5)}
        node[pos=0.6, above left] {$0.05 \cdot n^{\frac{3}{2}}$};
        \end{loglogaxis}
    \end{tikzpicture}
    \hspace{\spacing}
    \begin{tikzpicture}
        \begin{loglogaxis}[
            title={$+$signed discrepancy},
            width=\length,
            height=\length,
            font=\small,
            grid=none
        ]
        \addplot[mark=text, text mark={\scalebox{\markscale}{\hexlo}}, only marks] table[col sep=comma, x=x, y expr=(\thisrow{y})] {data/small_hexagon_low/added_minus_removed_means.csv};
        \addplot[mark=text, text mark={\scalebox{\markscale}{\tetlo}}, only marks] table[col sep=comma, x=x, y expr=(\thisrow{y})] {data/small_tetrahedron_low/added_minus_removed_means.csv};
        \addplot[mark=text, text mark={\scalebox{\markscale}{\trilo}}, only marks] table[col sep=comma, x=x, y expr=(\thisrow{y})] {data/small_triangle_low/added_minus_removed_means.csv};
        \addplot[
            thick,
            domain=5:300,
            samples=2
        ] {0.25*x^(1)}
        node[pos=0.6, above left] {$0.25 \cdot n$};
        \end{loglogaxis}
    \end{tikzpicture}
    \hspace{\spacing}
    \begin{tikzpicture}
        \begin{loglogaxis}[
            title={$-$signed discrepancy},
            width=\length,
            height=\length,
            font=\small,
            grid=none
        ]
        \addplot[mark=text, text mark={\scalebox{\markscale}{\hexhi}}, only marks] table[col sep=comma, x=x, y expr=-(\thisrow{y})] {data/small_hexagon_high/added_minus_removed_means.csv};
        \addplot[mark=text, text mark={\scalebox{\markscale}{\tethi}}, only marks] table[col sep=comma, x=x, y expr=-(\thisrow{y})] {data/small_tetrahedron_high/added_minus_removed_means.csv};
        \addplot[mark=text, text mark={\scalebox{\markscale}{\trihi}}, only marks] table[col sep=comma, x=x, y expr=-(\thisrow{y})] {data/small_triangle_high/added_minus_removed_means.csv};
        \addplot[mark=text, text mark={\scalebox{\markscale}{\tetlo}}, only marks] table[col sep=comma, x=x, y expr=-(\thisrow{y})] {data/small_tetrahedron_low/added_minus_removed_means.csv};
        \addplot[
            thick,
            domain=5:300,
            samples=2
        ] {0.25*x^(1)}
        node[pos=0.5, above left] {$0.25 \cdot n$};
        \end{loglogaxis}
    \end{tikzpicture}
    \hspace{\spacing}
    \begin{tikzpicture}
        \begin{loglogaxis}[
            title={clustering coefficient},
            width=\length,
            height=\length,
            font=\small,
            grid=none,
            y tick label style={
                /pgf/number format/fixed,
                /pgf/number format/precision=0
            },
            scaled y ticks=false
        ]
        \addplot[mark=text, text mark={\scalebox{\markscale}{\hexhi}}, only marks] table[col sep=comma] {data/small_hexagon_high/average_clustering_both_means.csv};
        \addplot[mark=text, text mark={\scalebox{\markscale}{\tethi}}, only marks] table[col sep=comma] {data/small_tetrahedron_high/average_clustering_both_means.csv};
        \addplot[mark=text, text mark={\scalebox{\markscale}{\trihi}}, only marks] table[col sep=comma] {data/small_triangle_high/average_clustering_both_means.csv};
        \addplot[mark=text, text mark={\scalebox{\markscale}{\hexlo}}, only marks] table[col sep=comma] {data/small_hexagon_low/average_clustering_both_means.csv};
        \addplot[mark=text, text mark={\scalebox{\markscale}{\tetlo}}, only marks] table[col sep=comma] {data/small_tetrahedron_low/average_clustering_both_means.csv};
        \addplot[mark=text, text mark={\scalebox{\markscale}{\trilo}}, only marks] table[col sep=comma] {data/small_triangle_low/average_clustering_both_means.csv};
        \addplot[
            thick,
            domain=5:300,
            samples=2
        ] {0.01*x^(-0.5)}
        node[pos=0.25, above right] {$0.01 \cdot n^{-\frac{1}{2}}$};
        \end{loglogaxis}
    \end{tikzpicture}
    
\caption{
Left: log-log plot of the average Hamming distance between samples $X$ of the ERGM and $Y$ from $\cG(n,p^*)$,
validating the $n^{3/2}$ power law guaranteed by Theorem \ref{thm:wasserstein_informal}.
Note that the degenerate behavior of datasets 4 and 6 likely arises due to the closeness of $p^*$ to $1$.
Middle: log-log plot of the average signed discrepancy between samples, as well as its negative on a separate log-log plot.
Right: log-log plot of the average clustering coefficient of the difference graph $|Z|$.
Note that some data points from certain datasets are missing as the corresponding graphs had a clustering coefficient of
$0$ due to the absence of triangles.}
\label{fig:sim_wass_plots}
\end{figure}

In accordance with Theorem \ref{thm:wasserstein_informal}, we observe the $n^{\frac{3}{2}}$ power law in average Hamming distance
under the coupling we consider, which should be close to optimal and give a good approximation of the \emph{Wasserstein distance}
between the ERGM and the corresponding Erd\H{o}s--R\'enyi model.
Note that Theorem \ref{thm:wasserstein_informal} actually only gives an upper bound on the Wasserstein distance of the form
$n^{\frac{3}{2}} \sqrt{\log n}$.
However, as mentioned below the statement of Theorem \ref{thm:wasserstein_informal},
in follow-up work by the author \cite{was} the $\sqrt{\log n}$ factor will be removed and a matching lower bound will be proved.

This improvement to theorem \ref{thm:wasserstein_informal} immediately implies via the triangle
inequality that the signed discrepancy is of order at most $n^{\frac{3}{2}}$ as well.
However, in our plots we can see that the signed discrepancy appears to be of order $n$ at most, and thus the marginal probability
of an edge should actually be $p^* + O(n^{-1})$ as opposed to the bound $p^* + O(n^{-1/2} \sqrt{\log n})$ which we obtain
in Proposition \ref{prop:marginal} below as a consequence of Theorem \ref{thm:wasserstein_informal}.
We state this as the following conjecture for consistency with the literature, but note that
since the first version of this article was posted on arXiv, this conjecture was resolved in the subcritical regime
by \cite{fang2025conditional}, and it will also be resolved in metastable wells in the supercritical regime in follow-up work by 
the author \cite{was}.

\begin{conjecture}[Resolved!]
For any $e \in \edgeset$, the probability that $X \sim \mu$ contains $e$ is $p^* + O(n^{-1})$.
\end{conjecture}

Finally, we turn to the structure of the difference graph $|Z|$.
We can observe in the right side of Figure \ref{fig:sim_wass_graphs} a difference graph from dataset 1
which appears to the eye more ``clustered'' than a typical graph with the same number of vertices and edges.
In the right plot of Figure \ref{fig:sim_wass_plots}, we have calculated the average \emph{clustering coefficient} of the
difference graph, which is the probability that when a uniform vertex $u$ is chosen and two random edges $\{u,v\}$
and $\{u,w\}$ are chosen, that $\{v,w\}$ is an edge in the graph.
Importantly, the edges $\{u,v\}$ and $\{v,w\}$ are chosen only from among the edges in the graph; not just any \emph{potential}
edge may be chosen.
This statistic is commonly used as a measure of the clustered nature of the graph, since intuitively a cluster is a set of
vertices with lots of mutual connections within the cluster, leading to more triangles.

Since there are $O(n^{3/2})$ edges in the difference graph, its density is $O(n^{-1/2})$.
In a model Erd\H{o}s--R\'enyi graph $\cG(n,n^{-1/2})$ with density of the same order, one would expect the average
clustering coefficient to be of order $n^{-1/2}$, since it is just the probability that a particular edge (the one
completing a triangle) is present in the graph and all edges behave independently.
However, the right plot of Figure \ref{fig:sim_wass_plots} suggests that in some cases the average clustering coefficient of the
difference graph may be much larger than $n^{-1/2}$.
This indicates that under the optimal coupling, the ERGM and Erd\H{o}s--R\'enyi model differ by the addition or removal
of \emph{many} edges in a \emph{very small} region of the graph.
However, this behavior is not consistent across the datasets and it is possible that this depends more sensitively on the
details of the ERGM specification.
Moreover, the clustering coefficient may not be the correct statistic to capture this behavior, as it is $0$ when there
are no triangles in the graph.
In any case, an explanation of this phenomenon seems out of reach of our current techniques, so we leave this as a
somewhat vague question for future work.

\begin{conjecture}
For certain ERGM specifications, under some coupling with a corresponding Erd\H{o}s--R\'enyi model,
the difference between the two sampled graphs is typically constrained to a localized region.
\end{conjecture}
\subsection{Approximating edge counts by normal distributions}
\label{sec:simulations_clt}

Finally we turn to Theorem \ref{thm:clt_informal}, and to the subject of central limit theorems for
edge counts more broadly.
To get a more accurate picture of the distribution of the edge count, we generated many more samples in datasets 7 and 8
($8192$ as compared to $128$ in the prior datasets).
We considered the global edge count as well as the local edge count in our simulations, but we remark that
Theorem \ref{thm:clt_informal} does not actually apply to either statistic, since it only applies for the edge count in
a set of potential edges with size much smaller than $n$.
However, since the first version of this article was posted, \cite{fang2024normal} showed a quantitative
central limit theorem for the global edge count in the \emph{subcritical} regime.
This was extended by the author \cite{winstein2025quantitative} to metastable wells in the supercritical regime, and in addition a
quantitative CLT was derived for the local edge count (vertex degree); moreover, the bounds given in \cite{fang2024normal} were improved in
\cite{winstein2025quantitative}.

Note however that neither of these works achieves the optimal bounds.
According to our simulations (see Figure \ref{fig:sim_clt}), the Kolmogorov--Smirnov statistic for the empirical distribution
of the global edge count should be of order $n^{-1}$, while \cite{fang2024normal,winstein2025quantitative} can only achieve bounds of order $n^{-1/2}$.
In addition, for the local edge count (vertex degree) one should obtain a bound of order $n^{-1/2}$ while \cite{winstein2025quantitative} 
can only achieve bounds of order $n^{-1/4}$.
It remains an important open problem to bridge the gap between the rigorous bounds and the expected behavior which is closer to the behavior guaranteed
(for instance by the Berry--Esseen theorem) for sums of i.i.d.\ random variables.

\begin{figure}
\centering
\def\length{5cm}
\def\spacing{1mm}
\def\markscale{0.5}
    \begin{tikzpicture}
        \begin{axis}[
            title = {global edge count ($n=128$)},
            ybar,
            bar width=0.5mm,
            font=\small,
            width=\length + 0.5cm,
            height=\length,
        ]
        \addplot+[
            hist={
                bins=25
            },
            fill=blue!30,
            fill opacity=0.5,
            draw=blue
        ] table[
            col sep=comma,
            y = value
        ] {data/large_triangle_low/total_edge_count_normalized_090.csv};
        \addplot+[
            hist={
                bins=25
            },
            fill=red!30,
            fill opacity=0.5,
            draw=red
        ] table[
            col sep=comma,
            y = value
        ] {data/large_tetrahedron_low/total_edge_count_normalized_090.csv};
        \end{axis}
    \end{tikzpicture}
    \hspace{\spacing + 1mm}
    \begin{tikzpicture}
        \begin{loglogaxis}[
            title={KSS (global edge count)},
            width=\length,
            height=\length,
            font=\small,
            grid=none
        ]
        \addplot[mark=text, text mark={\scalebox{\markscale}{\tetlo}}, only marks] table[col sep=comma] {data/large_tetrahedron_low/total_edge_count_kss.csv};
        \addplot[mark=text, text mark={\scalebox{\markscale}{\trilo}}, only marks] table[col sep=comma] {data/large_triangle_low/total_edge_count_kss.csv};
        \addplot[
            thick,
            domain=5:150,
            samples=2
        ] {1*x^(-1)}
        node[pos=0.8, below left] {$n^{-1}$};
        \end{loglogaxis}
    \end{tikzpicture}
    \hspace{\spacing}
    \begin{tikzpicture}
        \begin{loglogaxis}[
            title={KSS (local edge count)},
            width=\length,
            height=\length,
            font=\small,
            grid=none
        ]
        \addplot[mark=text, text mark={\scalebox{\markscale}{\tetlo}}, only marks] table[col sep=comma] {data/large_tetrahedron_low/edge_count_around_vertex_kss.csv};
        \addplot[mark=text, text mark={\scalebox{\markscale}{\trilo}}, only marks] table[col sep=comma] {data/large_triangle_low/edge_count_around_vertex_kss.csv};
        \addplot[
            thick,
            domain=5:150,
            samples=2
        ] {0.5*x^(-0.5)}
        node[pos=0.8, below left] {$0.5n^{-0.5}$};
        \end{loglogaxis}
    \end{tikzpicture}

\caption{Left: histograms of the global edge counts of $8192$ samples from the $90$-vertex ERGM, shifted by the sample mean and scaled by
the sample standard deviation; the blue histogram is dataset 7 and the red histogram is dataset 8.
Middle: log-log plot of the Kolmogorov--Smirnov statistic (distance to a standard normal distribution)
for the rescaled empirical distribution of the global edge count.
Note that the deviation from the trendline in dataset 7 merely indicates a lack of expressitivty arising from the sample size.
Right: log-log plot the Kolmogorov--Smirnov statistic for the rescaled empirical distribution of the local edge count.}
\label{fig:sim_clt}
\end{figure}
\section{Review of the exponential random graph model}
\label{sec:review}

In this section we provide a more detailed review of past work on the exponential random graph model (ERGM) which
is relevant for the current article.
The state space is $\Omega = \{0,1\}^{\binom{[n]}{2}}$, where $\binom{[n]}{2}$ denotes the set of all edges
in the complete graph $K_n$ on $n$ vertices; in other words, the state space is identified with the
set of simple graphs on $n$ vertices.

As mentioned in the introduction, for a fixed finite graph $G = (V,E)$ and a state $x \in \Omega$, we denote by $N_G(x)$ the number of \emph{labeled}
subgraphs of $x$ which are isomorphic to $G$, i.e.\ the number of maps $\sigma : V \to [n]$ for which $X(\{\sigma(u),\sigma(v)\}) = 1$
whenever $\{u,v\} \in E$.
In particular, when the graph $G$ has symmetries, each \emph{unlabeled} copy of $G$ in $x$ will be counted multiple times.
Again, we fix finite graphs $G_0, \dotsc, G_K$ with $G_0$ being a single edge,
as well as parameters $\beta_0 \in \R$ and $\beta_1, \dotsc, \beta_K \geq 0$.
Then the Hamiltonian of the ERGM is
\begin{equation}
	H(x) \coloneqq \sum_{i=0}^K \frac{\beta_i}{n^{|V_i|-2}} N_{G_i}(x),
\end{equation}
where $G_i = (V_i,E_i)$.
We will use $\tilde{\mu}$ to denote the full Gibbs measure with this Hamiltonian, i.e.\ we have $\tilde{\mu}(x) \propto \Exp{H(x)}$.
Note that in most later sections we will use $\mu$ (without a tilde)
to denote this ERGM Gibbs measure \emph{conditioned} on a particular subset of $\Omega$, namely a small cut distance ball around
a constant graphon, to be discussed in the next subsection.

\subsection{Large deviations principle}
\label{sec:review_ldp}

In \cite{chatterjee2013estimating} a large deviations principle was shown for a more general class of random graph models including the above
definition, in terms of the cut distance of graphs and graphons.
Here we give an extremely brief review of these definitions in order to state the large deviations principle for our present purposes;
for a thorough discussion of graphons and the cut distance, one may consult \cite{lovasz2012large}.

A \emph{graphon} is a symmetric measurable function $W : [0,1]^2 \to [0,1]$, 
meant to generalize the notion of an adjacency matrix $(A_{i,j})_{i,j=1}^n$ of a simple graph $x$, which can be interpreted as such a function
by setting $W_x(x,y) = A_{i,j}$ whenever $(x,y) \in \left[\frac{i-1}{n}, \frac{i}{n} \right) \times \left[\frac{j-1}{n}, \frac{j}{n} \right)$.
Notice however that this interpretation is not injective, since a graphon has no well-defined number of vertices, and, for example, subdividing
the squares leads to a different adjacency matrix for a graph with twice as many vertices which nonetheless has the same graphon representation.
Just as we consider graphs to be isomorphic when we permute their vertices, we identify $W(s,t)$ and $W(\sigma(s), \sigma(t))$, whenever
$\sigma : [0,1] \to [0,1]$ is an isomorphism of measure spaces (up to measure zero).

For a fixed graph $G = (V,E)$, we can naturally define \emph{subgraph densities} for graphons as follows:
\begin{equation}
	t(G,W) \coloneqq \int_0^1 \dotsb \int_0^1 \prod_{\{u,v\} \in E} W(s_u, s_v) \prod_{v \in V} ds_v.
\end{equation}
This naturally extends the notion of subgraph counts defined above for graphs.
Indeed, when $W = W_X$ for some graph $X$ on $n$ vertices, we have $N_G(X) = n^{|V|} t(G,W_X)$.

Next we define the cut distance, which is the following distance on the space of graphons:
\begin{equation}
\label{eq:cutdist}
	\db(W,W') \coloneqq \inf_{\substack{\sigma : [0,1] \to [0,1] \\ \text{measure-preserving} \\ \text{bijection}}}
    \sup_{S, T \sse [0,1]} \left| \int_S \int_T (W(s,t) - W'(\sigma(s), \sigma(t)) )\,dt \,ds \right|.
\end{equation}
For the present work it is not necessary to internalize the above definition.
One should be aware however, that convergence of graphons in the cut distance is equivalent to the convergence of all subgraph densities,
namely
\begin{equation}
	\db(W_n,W) \to 0
	\qquad \Leftrightarrow \qquad
	t(G,W_n) \to t(G,W) \text{ for all finite graphs } G.
\end{equation}
The above equivalence can be made quantitative \cite{lovasz2006limits,borgs2008convergent},
although this will not be relevant for the present work.

The simplest example of a graphon is the \emph{constant graphon} with constant $p \in [0,1]$, namely $W_p(x,y) = p$ for all $x,y \in [0,1]$.
This is the limit of Erd\H{o}s--R\'enyi random graphs $G_n \sim \cG(n,p)$, in the sense that
\begin{equation}
	\db(W_{G_n}, W_p) \xrightarrow{\;\P\;} 0.
\end{equation}
Note that in the cut distance \eqref{eq:cutdist}, when one of the graphons is constant, we may ignore the infimum over measure preserving
bijections, although again the precise formula for the cut distance will not be important for the present work.

The large deviations principle of \cite{chatterjee2013estimating} states that any sample $X$ from the ferromagnetic ERGM is typically close in
cut distance to a constant graphon $W_p$ for some value of $p$.
To heuristically determine which values of $p$ could appear here, notice that in order for $W_X$ to be close to $W_p$ in cut distance,
the subgraph densities $t(G_i,W_X)$ must be close to $p^{|E_i|}$, and so the Hamiltonian $H(X)$ must be close to
\begin{equation}
	n^2 \sum_{i=0}^K \beta_i p^{|E_i|}.
\end{equation}
Now, at the exponential (in $n^2$) scale, the number of graphs $X$ which look like they could have been sampled from $\cG(n,p)$ is
\begin{equation}
	\exp\left( \binom{n}{2} (- p \log p - (1-p) \log (1-p)) \right),
\end{equation}
since there are $\binom{n}{2}$ possible edge locations.
So, the relevant values of $p$ are those which maximize
\begin{equation}
\label{eq:ldef}
	L_{\bm\beta} (p) \coloneqq \sum_{i=0}^K \beta_i p^{|E_i|} - I(p),
\end{equation}
where $I(p) = \tfrac{1}{2} p \log p + \tfrac{1}{2} (1-p) \log(1-p)$, as defined in the introduction.
Let $M_{\bm\beta}$ denote the set of global maximizers $p^*$ of the function $L_{\bm\beta}$ on the interval $[0,1]$
(see Figure \ref{fig:ergm_pstar}).
The following theorem is the large deviations principle for the ferromagnetic ERGM.

\begin{theorem}[Theorems 3.2 and 4.1 of \cite{chatterjee2013estimating}]
\label{thm:LDP} 
Let $X$ be sampled from the unconditioned ferromagnetic ERGM measure $\tilde{\mu}$ and define $M_{\bm\beta}$ as above.
For any $\eta > 0$, there are constants $c(\eta), C(\eta) \geq 0$ for which
\begin{equation}
    \tilde{\mu} \left[\inf_{p^* \in M_{\bm\beta}} \db(W_X,W_{p^*}) > \eta\right] \leq C(\eta) \Exp{-c(\eta) n^2}.
\end{equation}
\end{theorem}
\subsection{Glauber dynamics}
\label{sec:review_dynamics}

Now let us take a more dynamical perspective, resampling edges in $X$ via Glauber dynamics.
At each time step, a uniformly random edge $e \in \binom{[n]}{2}$ is chosen to be resampled, and starting from $X$ the subsequent state
is either $X^{+e}$ or $X^{-e}$, i.e.\ the same configuration but with $X_e$ fixed to be $1$ or $0$ respectively.
On the event that an edge $e$ is selected to be resampled, the probability we end up with $X^{+e}$ is
\begin{equation}
    \phi(\partial_e H(X)),
\end{equation}
where $\phi(z) = \fr{e^z}{1+e^z}$ and $\partial_e H(X) = H(X^{+e}) - H(X^{-e})$.
Now by the definition of $H$, we have
\begin{equation}
	\partial_e H(X) = \sum_{i=0}^K \fr{\beta_i}{n^{|V_i|-2}} N_{G_i}(X,e),
\end{equation}
where $N_G(X,e) = N_G(X^{+e}) - N_G(X^{-e})$, which can also be calculated as the number of labeled copies of
$G$ in $X^{+e}$ which include the edge $e$.

Suppose now that $X$ is close to a sample from $\cG(n,p)$ in the sense that $\fr{1}{n^{|V_i|-2}} N_{G_i}(X,e) \approx 2 |E| p^{|E_i|-1}$ for all
$i = 0, \dotsc, K$ and all $e \in \binom{[n]}{2}$.
Note that this is not the same as $W_X$ being close to $W_p$ in the cut distance, because it requires the closeness of all relevant
\emph{local} subgraph densities.
Then, given that edge $e$ was selected to be resampled, the probability that we end up with $X^{+e}$ after one step of Glauber
dynamics is approximately
\begin{equation}
\label{eq:phidef}
	\phi_{\bm\beta}(p) \coloneqq \phi(\Psi_{\bm\beta}(p)),
\end{equation}
where
\begin{equation}
\label{eq:psidef}
	\Psi_{\bm\beta}(p) \coloneqq \sum_{i=0}^K 2 \beta_i |E_i| p^{|E_i|-1}.
\end{equation}
Since $\tilde{\mu}$ is stationary for the Glauber dynamics, this means that a sample $X \sim \tilde{\mu}$ should be close to $\cG(n,p)$ for some $p$ which
is an attracting fixed point of the function $p \mapsto \phi_{\bm\beta}(p)$, i.e.\ $p = \phi_{\bm\beta}(p)$ and $\phi_{\bm\beta}'(p) < 1$.

It may be easily verified that $p = \phi_{\bm\beta}(p)$ is equivalent to $L_{\bm\beta}'(p) = 0$, and $\phi_{\bm\beta}'(p) < 1$ is equivalent
to $L_{\bm\beta}''(p) < 0$, so $p$ is an attracting fixed point of $\phi_{\bm\beta}$ if and only if it is a local maximum fo
$L_{\bm\beta}$ with strictly negative second derivative.
If there is a unique local maximum of $L_{\bm\beta}$ with negative second derivative, or equivalently if there is a unique attracting fixed
point for $p \mapsto \phi_{\bm\beta}(p)$, then we say that the ERGM is in the \emph{subcritical} regime of parameters.
On the other hand, if there are at least two local maxima of $L_{\bm\beta}$ with negative second derivatives (i.e.\ at least two attracting
fixed points of $p \mapsto \phi_{\bm\beta}(p)$), then we say that the ERGM is in the \emph{supercritical} regime of parameters
(see Figure \ref{fig:ergm_regimes}).
If neither condition holds, the ERGM is in the \emph{critical} regime.
For future reference, we let $U_{\bm\beta} \sse M_{\bm\beta}$ denote the collection of all \emph{global} maxima $p^*$ of $L_{\bm\beta}$
for which $L_{\bm\beta}''(p^*) < 0$.
We remark that even in the supercritical regime, for most values of the parameters $\bm\beta$, both $U_{\bm\beta}$ and $M_{\bm\beta}$
are singletons.

Early studies of the ERGM focused on the subcritical regime; in \cite{bhamidi2008mixing} it was proved that the Glauber dynamics in the
subcritical regime exhibits rapid mixing, using a burn-in argument which allows the Markov chain to enter a set of configurations
in which the attractiveness of the unique fixed point can be used to obtain contraction.
In the same work, it was also shown that there is an exponential slowdown in the supercritical regime, due to the dynamics getting stuck
near an attractive fixed point which is not the unique \emph{global} maximum of $L_{\bm\beta}$.
\subsection{Concentration of measure}
\label{sec:review_concentration}

Since the Glauber dynamics is a local algorithm, the rapid mixing result can be used to obtain structural information
about the subcritical ERGM.
This was done in \cite{ganguly2024sub}, specifically by deriving concentration inequalities for Lipschitz observables of the ERGM
from the mixing result combined with a method of \cite{chatterjee2005concentration} which allows one to translate certain markers
of rapid mixing into concentration inequalities.
We remind the reader that for $v \in \R^{\binom{[n]}{2}}$, we say $f : \Omega \to \R$ is $v$-Lipschitz if
$|f(x^{+e}) - f(x^{-e})| \leq v_e$ for every $e \in \binom{[n]}{2}$ and every $x \in \Omega$.
The main result of \cite{ganguly2024sub} is the following concentration inequality.

\begin{theorem}[Theorem 1 of \cite{ganguly2024sub}]
\label{thm:conc_subcritical}
    For $X \sim \tilde{\mu}$, an unconditioned subcritical ERGM measure,
	There is a constant $c > 0$ depending only on $\tilde{\mu}$
    such that for large enough $n$, any $v$-Lipschitz $f : \Omega \to \R$, and any $\lambda \geq 0$,
	\begin{equation}
		\tilde{\mu}[|f(X) - \E_{\tilde{\mu}}[f(X)]| > \lambda] \leq 2 \Exp{- c \frac{\lambda^2}{\| v \|_\infty \| v \|_1}}.
	\end{equation}
\end{theorem}

This result was then used to obtain various structural information about the ERGM in the subcritical regime.
Namely, it was used to prove a central limit theorem for the number of edges in a sparse subcollections of possible edges, where no pair
of possible edges share a vertex \cite[Theorem 2]{ganguly2024sub}.
Additionally, the concentration result was used to derive a bound of order $n^{3/2} \sqrt{\log n}$ on the Wasserstein distance between 
the ERGM and $\cG(n,p^*)$, where $p^*$ is the unique local maximum of $L_{\bm\beta}$ \cite[Theorem 3]{ganguly2024sub}.

A more intricate study of the \emph{supercritical} regime came only more recently.
In \cite{bresler2024metastable}, it was shown that in this regime, when one restricts to a cut distance ball around one of the (possibly multiple)
global maximizers of $L_{\bm\beta}$, i.e.\ $p^* \in U_{\bm\beta}$, then the Glauber dynamics mixes rapidly \emph{when given a warm start}, 
sampling the initial state from the relevant Erd\H{o}s--R\'enyi model $\cG(n,p^*)$.

Note that the form of mixing derived in \cite{bresler2024metastable} is not the standard notion of rapid mixing which yields a total variation
distance which is exponentially small in terms of the number of multiples of the mixing time.
Indeed, as previously mentioned, even inside of the cut distance ball around $p^* \in U_{\bm\beta}$, in the supercritical regime there may 
still be small regions (with probability $e^{-\Theta(n)}$) which are metastable in the sense that the Glauber dynamics 
takes an exponentially long time to escape (see for instance \cite[Theorem 3.3]{bresler2024metastable}).
The Glauber dynamics started from $\cG(n,p^*)$ also takes exponentially long to \emph{reach} these regions, forcing the total variation
distance after \emph{any} subexponential amount of time to be at least $e^{-O(n)}$.

In the present article, we are nonetheless able to derive the concentration inequality Theorem \ref{thm:lipschitzconcentration_informal}, 
along the lines of Theorem \ref{thm:conc_subcritical}, from the metastable mixing results of \cite{bresler2024metastable}, answering a
question posed in that work.
As mentioned in the outline, the proof of Theorem \ref{thm:lipschitzconcentration_informal} appears in Section \ref{sec:lipschitz}.

Using Theorem \ref{thm:lipschitzconcentration_informal}, we are able to obtain structural information about the ERGM conditioned on a
cut distance ball around some $W_{p^*}$ for any $p^* \in U_{\bm\beta}$.
In particular, in Section \ref{sec:applications_wasserstein} we prove Theorem \ref{thm:wasserstein_informal} which gives a bound on the Wasserstein
distance between this measure an $\cG(n,p^*)$ under certain extra conditions on the ERGM.
In Section \ref{sec:applications_clt} we prove Theorem \ref{thm:clt_informal}, which gives a central limit theorem for certain small subcollections of
edges under any subcritical ERGM measure conditioned on any relevant cut distance ball.
\section{Concentration via metastable mixing}
\label{sec:concentration}

In this section we present two results which allow one to derive concentration inequalities
from various markers of metastable mixing.
The first, Theorem \ref{thm:chatterjee}, is a novel modification of \cite[Theorem 3.3]{chatterjee2005concentration}
that allows it to work with metastable mixing, and the second, Corollary \ref{cor:barbour},
follows easily from \cite[Theorem 2.1]{barbour2022long} which is already relatively suited to metastable mixing.
Each of these results has its own benefits and drawbacks, and it is only through their combination
that we are able to derive our main result, Theorem \ref{thm:lipschitzconcentration_informal}.
Namely, Corollary \ref{cor:barbour} provides an optimal result for global observables such as the total number
of edges, while Theorem \ref{thm:chatterjee} provides a better rate of concentration for local observables
such as the degree of a vertex but gives a vacuous bound for global observables.
The control of observables at both scales, especially the local scale, will be crucial for our applications in
Section \ref{sec:applications}.

We have framed Theorem \ref{thm:chatterjee} and Corollary \ref{cor:barbour} below to look somewhat similar, mainly to
allow for our main techniques in Section \ref{sec:lipschitz} to apply seamlessly with both theorems.
However, note that the proofs of these theorems are rather different and it remains an interesting open
question to combine them in a way that alleviates the drawbacks of both techniques while retaining all of the benefits.

Let us briefly give some intuition for these results.
Heuristically, if one has a fast sampling algorithm for $X$ such that at each step, the value of an observable does not change too much,
then one would expect that observable to exhibit some form of concentration, since multiple trials of the sampling algorithm cannot
yield values of the observable which differ by too much.
This idea can be seen even in the most classical concentration inequalities, such as Azuma's inequality, where the martingale
in the standard proof can be viewed as a sampling algorithm.
Theorem \ref{thm:chatterjee} and Corollary \ref{cor:barbour} are instantiations of this idea when phrased in terms of markov chains
which exhibit metastable mixing.
Throughout this section, $(X_t)$ will denote a stationary Markov chain and $(X_t^x)$
will denote a Markov chain which is started at a particular state, i.e.\ $X_0^x = x$.

Finally, we remark that there is no single agreed-upon definition of ``metastable mixing'' in the literature, and in general
slightly different results will arise in different situations.
For instance, in the $p$-spin Curie--Weiss model, one can obtain standard mixing estimates for the dynamics conditioned to remain
in a metastable well \cite[Theorem 2.2]{samanta2024mixing}.
However, in the ERGM which is our main motivation, it is not possible to obtain standard mixing estimates as the dynamics
may get trapped behind bottlenecks inside of small regions \emph{within} the metastable well \cite[Theorem 3.3]{bresler2024metastable}.
We expect this phenomenon to be more widespread, and we adapt our results specifically to this possibility.

\subsection{A new result which works for local observables}
\label{sec:concentration_chatterjee}

Let us first present the result which will allow us to handle local observables,
which is a modification of the result of \cite{chatterjee2005concentration} displayed below.
Following \cite{chatterjee2005concentration}, we present this result and our modification
using the language of \emph{exchangeable pairs}, which will allow for a slightly cleaner proof.
Note however that the data of an exchangeable pair is equivalent to that of a stationary reversible Markov chain;
indeed, a pair of random variables $(X,X')$ is exchangeable if and only if it has the same distribution as $(X_0, X_1)$,
the first two steps of a stationary reversible Markov chain.
We will identify these objects and often refer to the Markov chain associated with an exchangeable pair, or vice versa.

\begin{theorem}[Theorem 3.3 of \cite{chatterjee2005concentration}]
\label{thm:chatterjee_old}
Let $(X,X')$ be an exchangeable pair of random variables in $\Omega$, and let $f : \Omega \to \R$ be any function.
Suppose there is an antisymmetric function $F : \Omega^2 \to \R$ such that $f(x) = \Econd{F(X,X')}{X=x}$, and
\begin{equation}
    \label{eq:chatterjee_condition}
    g(x) \coloneqq \fr{1}{2} \Econd{|f(X) - f(X')| |F(X,X')|}{X = x} \leq V
\end{equation}
for all $x \in \Omega$.
Then for all $\lambda \geq 0$ we have
\begin{equation}
    \P[ f(X) \geq \lambda ] \leq e^{- \lambda^2 / 2 V}.
\end{equation}
\end{theorem}

Let us comment on the role of the antisymmetric function $F$ in the above statement.
There might be multiple suitable ways to construct a function $F$ to use with Theorem \ref{thm:chatterjee_old},
and indeed a few options are exhibited in \cite{chatterjee2005concentration}, each one leading to slightly different conclusions.
However, given a function $f$, one can always make the following universal choice for $F$, whenever it is well-defined:
\begin{equation}
    \label{eq:chatterjee_F}
    F(x,x') \coloneqq \sum_{t=0}^\infty \left( \E[f(X_t^x)] - \E[f(X_t^{x'})] \right).
\end{equation}
Here, $(X_t^x)$ denotes the Markov chain associated with the exchangeable pair started at $X_0^x = x$
and similarly for $(X_t^{x'})$.
This can easily be seen to satisfy $\Econd{F(X,X')}{X=x} = f(x)$.
By the linearity of expectation, for any coupling between the two chains $(X_t^x)$ and $(X_t^{x'})$, we have
\begin{equation}
    F(x,x') = \sum_{t=0}^\infty \E\left[f(X_t^x) - f(X_t^{x'})\right].
\end{equation}
If there is rapid mixing, then we can find a coupling for which $f(X_t) - f(X_t')$ quickly becomes small, which allows
for the condition \eqref{eq:chatterjee_condition} to be verified.
This strategy was pursued in \cite{ganguly2024sub}, where concentration inequalities for arbitrary Lipschitz observables of
the subcritical ERGM were obtained.

In the supercritical case where we only have rapid metastable mixing and not rapid standard mixing, we will need to adjust the
hypotheses accordingly. 
In particular, we would like to only observe the chain while it is still rapidly mixing inside the metastable well,
before it has had a chance to leave (which would destroy any hope of a bound), and so we will
need to consider a different function $F$ than in \eqref{eq:chatterjee_F}.
In our applications, we will consider the function
\begin{equation}
\label{eq:Fdef}
    F(x,x') = \sum_{t = 0}^{T-1} \E \left[ f(X_t^x) - f(X_t^{x'}) \right],
\end{equation}
for $T$ chosen appropriately.
Now, we no longer have $\Econd{F(X,X')}{X = x} = f(x)$, but instead we have
\begin{align}
    \Econd{F(X,X')}{X = x} &= f(x) - P f(x) \\
    &\quad + P f(x) - P^2 f(x) \\
    &\quad + \dotsb \\
    &\quad + P^{T-1} f(x) - P^T f(x) \\
    &= f(x) - P^T f(x).
\end{align}
If the chain mixes rapidly within the metastable well starting from $x$, and $f$ has expectation zero within the well,
then $P^T f(x)$ will be small, allowing for $F$ to be used in a similar capacity as in Theorem \ref{thm:chatterjee}.
The following theorem encapsulates these changes.
As with Theorem \ref{thm:chatterjee_old}, it is more convenient to state the result in terms of the exchangeable pair $(X,X')$,
rather than the Markov chain directly.

\begin{theorem}
\label{thm:chatterjee}
Let $(X,X')$ be an exchangeable pair of random variables in $\Omega$, and let $f : \Omega \to \R$ be any function.
Suppose that we have subsets $\Lambda \sse \Lambda_+ \sse \Omega$ with $\Pcond{X' \in \Lambda_+}{X \in \Lambda} = 1$,
and an antisymmetric function $F : \Lambda_+^2 \to \R$ such that
the following conditions hold for some nonnegative numbers $\eps, M, \Delta, \delta, V$:
\begin{enumerate}[label=(A\arabic*)]
    \item
    \label{cond:gbound}
    A bound on a proxy for the variance of $f$ in $\Lambda$:
    \begin{equation}
        g(x) \coloneqq \fr{1}{2} \Econd{|f(X)-f(X')| |F(X,X')|}{X = x} \leq V \qquad \text{for all } x \in \Lambda.
    \end{equation}

    \item
    \label{cond:diffbound}
    The function $f$ assigns similar values to the exchangeable pair with one end in $\Lambda$:
    \begin{equation}
        \Pcond{|f(X) - f(X')| \leq \Delta}{X \in \Lambda} = 1.
    \end{equation}

    \item
    \label{cond:lambdalarge}
    The set $\Lambda$ occupies a large portion of the space:
    \begin{equation}
        \P[X \in \Lambda] \geq 1 - \eps.
    \end{equation}

    \item
    \label{cond:approxlift}
    The antisymmetric function $F$ is an approximate lift of $f$ in $\Lambda_+$:
    \begin{equation}
        \left| \Econd{F(X,X')}{X = x} - f(x) \right| \leq \delta \qquad \text{for all } x \in \Lambda_+.
    \end{equation}

    \item 
    \label{cond:aprioribound}
    The function $f$ is not too large within $\Lambda_+$:
    \begin{equation}
        \Pcond{|f(X)| \leq M}{X \in \Lambda_+} = 1.
    \end{equation}
\end{enumerate}
In addition, assume that $\eps \leq \fr{1}{2}$ and that $\delta \leq M$.
Then, for all $\lambda \in [0, \fr{V}{\Delta}]$, we have
\begin{equation}
    \P[f(X) > \lambda] \leq
    \Exp{- \fr{\lambda^2}{2V} + \fr{\Delta \lambda^3}{4 V^2}
    + \fr{\delta \lambda}{V} + 2 \eps e^{2 \lambda M / V}} + \eps.
\end{equation}
\end{theorem}

\begin{proof}[Proof of Theorem \ref{thm:chatterjee}]
Set $m(\theta) = \E[e^{\theta f(X)} | X \in \Lambda]$.
Then we have
\begin{align}
    m'(\theta) &= \Econd{f(X) e^{\theta f(X)}}{X \in \Lambda} \\
    &\leq \Econd{(F(X,X') + \delta) e^{\theta f(X)}}{X \in \Lambda}\\
    &= \delta \cdot m(\theta) \\
    \label{eq:mspf1}
    &\quad +\fr{1}{2} \Econd{F(X,X') \left(e^{\theta f(X)} - e^{\theta f(X')}\right)}{X \in \Lambda} \\
    \label{eq:mspf2}
    &\quad + \fr{1}{2} \Econd{F(X,X') \left( e^{\theta f(X)} + e^{\theta f(X')}\right)}{X \in \Lambda},
\end{align}
using condition \ref{cond:approxlift} in the second line.

Let us first control \eqref{eq:mspf2}.
We have
\begin{align}
    \eqref{eq:mspf2}
    \label{eq:mspf3}
    &= \fr{1}{2} \Econd{F(X,X') \left( e^{\theta f(X)} + e^{\theta f(X')} \right)}{X, X' \in \Lambda} \Pcond{X' \in \Lambda}{X \in \Lambda} \\
    \label{eq:mspf4}
    &\quad + \fr{1}{2} \Econd{F(X,X') \left( e^{\theta f(X)} + e^{\theta f(X')} \right)}{X \in \Lambda, X' \notin \Lambda} \Pcond{X' \notin \Lambda}{X \in \Lambda}.
\end{align}
Since the expression $F(X,X') \left( e^{\theta f(X)} + e^{\theta f(X')} \right)$ is antisymmetric in $X$ and $X'$,
and the condition $X,X' \in \Lambda$ is symmetric in $X$ and $X'$, \eqref{eq:mspf3} is zero.
As for \eqref{eq:mspf4}, first note that
\begin{align}
    \Pcond{X' \notin \Lambda}{X \in \Lambda} &= \fr{\P[X' \notin \Lambda] - \Pcond{X' \notin \Lambda}{X \notin \Lambda} \P[X \notin \Lambda]}{\P[X \in \Lambda]} \\
    &\leq \fr{\eps}{1 - \eps},
\end{align}
using the fact that $\P[X' \notin \Lambda] = \P[X \notin \Lambda]$, and both are at most $\eps$.
Additionally, by conditions \ref{cond:approxlift} and \ref{cond:aprioribound}, we have
\begin{equation}
    \Econd{F(X,X')}{X = x} \leq M + \delta
\end{equation}
for any $x \in \Lambda$.
So, using the tower property, we have
\begin{align}
    \Econd{F(X,X') e^{\theta f(X)}}{X \in \Lambda, X' \notin \Lambda} &=
    \Econd{\Econd{F(X,X')}{X} e^{\theta f(X)}}{X \in \Lambda, X' \notin \Lambda} \\
    &\leq (M+\delta) e^{\theta M},
\end{align}
using condition \ref{cond:aprioribound} again.
We get the same bound for the term involving $e^{\theta f(X')}$ using the tower property on $X'$ instead,
since we also have $\Econd{F(X,X')}{X' = x'} \leq M + \delta$ for any $x \in \Lambda_+$ by the same conditions
as before (plus the antisymmetry of $F$), and so we can conclude that
\begin{equation}
    \label{eq:mspf5}
    \eqref{eq:mspf2} = \eqref{eq:mspf4} \leq \fr{\eps}{1-\eps} (M+\delta) e^{\theta M}.
\end{equation}

Now we turn to \eqref{eq:mspf1}.
Since $|e^a - e^b| \leq \fr{|a-b|}{2}(e^a + e^b)$, as can be checked by expanding $e^a - e^b$ as an integral and applying 
Jensen's inequality pointwise, for $\theta \geq 0$ we have
\begin{align}
    \eqref{eq:mspf1} &\leq \fr{1}{4} \Econd{|F(X,X')| |\theta f(X) - \theta f(X')| (e^{\theta f(X)} + e^{\theta f(X')})}{X \in \Lambda} \\
    &\leq \fr{1+e^{\theta \Delta}}{4} \theta \cdot \Econd{|F(X,X')| |f(X) - f(X')| e^{\theta f(X)}}{X \in \Lambda} \tag*{(condition \ref{cond:diffbound})} \\
    &= \fr{1+e^{\theta \Delta}}{2} \theta \cdot \Econd{g(X) e^{\theta f(X)}}{X \in \Lambda} \tag*{(tower property)} \\
    &\leq \fr{1+e^{\theta \Delta}}{2} V \theta \cdot m(\theta). \tag*{(condition \ref{cond:gbound})}
\end{align}
This combined with \eqref{eq:mspf5} and the first display in the proof gives
\begin{equation}
    m'(\theta) \leq \delta \cdot m(\theta) +
        \fr{1 + e^{\theta \Delta}}{2} V \theta \cdot m(\theta)
        + \fr{\eps}{1-\eps} (M+\delta) e^{\theta M}.
\end{equation}
Now since $m(\theta) \geq e^{-\theta M}$ for $\theta \geq 0$ by condition \ref{cond:aprioribound},
for such $\theta$ the above inequality gives
\begin{equation}
    \fr{d}{d\theta} \log m(\theta) \leq \delta + \fr{1 + e^{\theta \Delta}}{2} V \theta
    + \fr{\eps}{1-\eps} (M+\delta) e^{2 \theta M}.
\end{equation}
Since $m(0) = 1$, integrating $\theta$ from $0$ to $\tau > 0$ we obtain
\begin{equation}
    \log m(\tau) \leq
    \delta \tau +
    \fr{V}{2} \left( \fr{\tau^2}{2} +
        \fr{e^{\tau \Delta} (\tau \Delta - 1) + 1}{\Delta^2} \right) +
    \fr{\eps}{1-\eps} (M+\delta) \fr{e^{2 \tau M} - 1}{2M}.
\end{equation}
Now note that for $a \in [0,1]$ we have
\begin{align}
    e^a (a-1) + 1 &= 1 + a + a^2 + \fr{a^3}{2} + \fr{a^4}{6} + \dotsb \\
    &\quad -1 - a - \fr{a^2}{2} - \fr{a^3}{6} - \fr{a^4}{24} - \dotsb \\
    &= \fr{a^2}{2} + \sum_{k=3}^\infty a^k \left( \fr{1}{(k-1)!} - \fr{1}{k!} \right) \\
    &\leq \fr{a^2}{2} + a^3 \sum_{k=3}^\infty \left(\fr{1}{(k-1)!} - \fr{1}{k!}\right) \\
    &= \fr{a^2}{2} + \fr{a^3}{2}.
\end{align}
So, for $0 \leq \tau \leq \fr{1}{\Delta}$, we have
\begin{align}
    \log m(\tau) &\leq \delta \tau +
    \fr{V}{2} \left( \fr{\tau^2}{2} + \fr{(\tau \Delta)^2
        + (\tau \Delta)^3}{2 \Delta^2} \right) +
    \fr{\eps}{1-\eps} \fr{M+\delta}{2M} (e^{2 \tau M} - 1) \\
    &\leq \delta \tau + \fr{V}{2} \tau^2 + \fr{V \Delta}{4} \tau^3
    + \fr{\eps}{1-\eps} \fr{M+\delta}{2M} (e^{2 \tau M} - 1).
\end{align}
So by Markov's inequality, taking $\tau = \fr{\lambda}{V}$, which is at most $\fr{1}{\Delta}$
since $\lambda \in [0,\fr{V}{\Delta}]$, we have
\begin{align}
    \Pcond{f(X) > \lambda}{X \in \Lambda}
    &\leq \Exp{ \log m\left(\fr{\lambda}{V}\right) - \fr{\lambda^2}{V}} \\
    &\leq \Exp{ - \fr{\lambda^2}{2 V} + \fr{\Delta \lambda^3}{4V^2} + \fr{\delta \lambda}{V}
    + \fr{\eps}{1 - \eps} \fr{M+\delta}{2M} (e^{2 \lambda M / V} - 1)}.
\end{align}
Finally, by using our assumptions that $\eps \leq \fr{1}{2}$
and $\delta \leq M$, we can simplify the last term inside the exponent.
Applying condition \ref{cond:lambdalarge} as well, we obtain
\begin{equation}
    \P[f(X) > \lambda] \leq
    \Exp{- \fr{\lambda^2}{2V} + \fr{\Delta \lambda^3}{4 V^2}
    + \fr{\delta \lambda}{V} + 2 \eps e^{2 \lambda M / V}} + \eps,
\end{equation}
which is the desired bound.
\end{proof}
\subsection{A result which works for global observables}
\label{sec:concentration_barbour}

Next let us introduce the result which will give us optimal rates of concentration for global observables.
It is derived from the following theorem due to \cite{barbour2022long}.
Note that we have adapted the notation to our present setting, and to be similar to the presentation of
Theorem \ref{thm:chatterjee} above, for more streamlined application in the next section.
Unlike Theorem \ref{thm:chatterjee_old}, this result is already suited for metastable mixing in its current form.

\begin{theorem}[Theorem 2.1 of \cite{barbour2022long}]
\label{thm:barbour_old}
Let $\Lambda \sse \Omega$ and $f : \Omega \to \R$ such that the following two conditions hold for some
positive numbers $V$ and $\Delta$, as well as some time $T \in \N$.
\begin{enumerate}[label=(B\arabic*)]
    \item
    \label{cond:variance}
    A bound on a proxy for the variance of $f$: for all $x \in \Lambda$,
    \begin{equation}
        h(x) \coloneqq \sum_{x' \in \Omega} P(x,x') \sum_{t=0}^{T-1} \left(
            \E[f(X_t^x)] - \E[f(X_t^{x'})] \right)^2 \leq V.
    \end{equation}
    \item
    \label{cond:smalljumps}
    A bound on the future expected difference, measured by $f$, between chains started at two adjacent initial states:
    for all $x \in \Lambda$, and $x' \in \Omega$ with $P(x,x') > 0$, and for all $t \geq 0$,
    \begin{equation}
        \left| \E[f(X_t^x)] - \E[f(X_t^{x'})] \right| \leq \Delta.
    \end{equation}
\end{enumerate}
Then, for any $x \in \Lambda$, we have
\begin{equation}
\label{eq:barbourconclusion}
    \P \left[
        X_t^x \in \Lambda \text{ for } 0 \leq t < T,
        \text{ and }
        |f(X_T^x) - \E[f(X_T^x)]| \geq \lambda
    \right] \leq 2 \Exp{- \fr{\lambda^2}{2 V + \fr{4}{3} \Delta \lambda}}.
\end{equation}
\end{theorem}

The main difference between Theorem \ref{thm:chatterjee} and \ref{thm:barbour_old} is captured by the difference between
the conditions \ref{cond:gbound} and \ref{cond:variance}.
Note that condition \ref{cond:gbound} with the choice \eqref{eq:Fdef} for the antisymmetric function $F$
reduces to the following condition for all $x \in \Lambda$:
\begin{equation}
\label{eq:bigbound}
    \sum_{x' \in \Omega} P(x,x') \left| f(x) - f(x') \right| \sum_{t=0}^{T-1} \left| \E[f(X_t^x)] - \E[f(X_t^{x'})] \right| \leq 2V.
\end{equation}
The factor of $| f(x) - f(x') |$ is ultimately what leads to the dependence on $\| v \|_1$ in
Theorem \ref{thm:lipschitzconcentration_informal}, while the expression involving expectations may in general only be bounded in terms of
$\| v \|_\infty$ and the overall mixing rate of the chain.
This difference is the reason that Theorem \ref{thm:chatterjee} leads to the fluctuation rate of $\sqrt{\| v \|_1 \| v \|_\infty}$
for local observables and why Theorem \ref{thm:barbour_old} leds to the fluctuation rate of $n \| v \|_\infty$ for global observables.

Theorem \ref{thm:barbour_old} provides concentration for the chain at some large enough time $T$, but we would like to
have concentration under the stationary distribution itself.
Moreover, in the conclusion there appears the condition that $X_t^x \in \Lambda$ for all
$0 \leq t < T$, which we would like to avoid carrying through the proofs that follow.
Both of these issues may be remedied if the chain mixes rapidly enough and if the chain
is very likely to remain in $\Lambda$ for $t < T$.
So we introduce some extra hypotheses and state a straightforward corollary which allows
for more seamless application of Theorem \ref{thm:barbour_old} in the next section.
In this corollary and below, we use $\dtv$ to denote the \emph{total variation distance} between
probability distributions or random variables.

\begin{corollary}
\label{cor:barbour}
Let $\mu$ denote the stationary measure of the dynamics $(X_t)$, and let $X \sim \mu$.
Suppose there is a set $\Lambda$ and a function $f$ such that the conditions \ref{cond:variance} and \ref{cond:smalljumps} in
the statement of Theorem \ref{thm:barbour_old} hold for positive numbers $V$ and $\Delta$, as well as time $T$.
Suppose also that the following three additional conditions hold for some nonnegative numbers
$\epsilon, \delta$, and $M$, and some $z \in \Lambda$.
\begin{enumerate}[label=(B\arabic*)]
    \setcounter{enumi}{2}
    \item
    \label{cond:setlarge}
    The chain has a high probability of staying in $\Lambda$ for time $T$ when started from $z$:
    \begin{equation}
        \P[X_t^z \in \Lambda \text{ for } 0 \leq t < T] \geq 1 - \epsilon.
    \end{equation}

    \item
    \label{cond:mixing}
    The chain mixes rapidly with respect to $f$ when started from $z$:
    \begin{equation}
        \dtv(f(X_T^z), f(X)) \leq \delta.
    \end{equation}

    \item
    \label{cond:fbound}
    The function $f$ is bounded close to its mean: for all $x \in \Omega$,
    \begin{equation}
        |f(x) - \E_\mu[f]| \leq M.
    \end{equation}
\end{enumerate}
Then for all $\lambda \geq 2 \delta M$ we have
\begin{equation}
\label{eq:barbourcor}
    \mu[|f(X) - \E[f(X)]| > \lambda] \leq 2 \exp \left(
        - \fr{(\lambda-2\delta M)^2}{2V + \fr{4}{3} \Delta (\lambda - 2\delta M)}
    \right) + \epsilon + \delta.
\end{equation}
\end{corollary}

\begin{proof}[Proof of Corollary \ref{cor:barbour}]
First of all, conditions \ref{cond:mixing} and \ref{cond:fbound} imply that
\begin{equation}
    \left| \E[f(X_T^z)] - \E[f(X)] \right| \leq 2 \delta M,
\end{equation}
so if we have $\left| f(X_T^z) - \E[f(X)] \right| > \lambda$ then we must also have
\begin{equation}
\label{eq:lambdaminus2deltam}
    \left|
        f(X_T^z) - \E[f(X_T^z)]
    \right| > \lambda - 2 \delta M.
\end{equation}
By conditions \ref{cond:setlarge} and \ref{cond:mixing}, we may couple $X$ and $(X_t^z)$ such that the event
where $f(X) \neq f(X_T^z)$ or $X_t^z \notin \Lambda$ for some $t$ with $0 \leq t < T$ happens with probability
at most $\epsilon + \delta$.
Thus \eqref{eq:barbourcor} follows from \eqref{eq:barbourconclusion} with $\lambda$ replaced by $\lambda - 2 \delta M$.
\end{proof}
\section{Proof of main theorem}
\label{sec:lipschitz}

In this section we use Theorem \ref{thm:chatterjee} and Corollary \ref{cor:barbour}
to prove our main theorem, providing a concentration inequality for Lipschitz
observables of the ERGM conditioned on a small cut distance ball around a constant graphon $W_{p^*}$ for some $p^* \in U_{\bm\beta}$.
To fix notation, for any $p \in [0,1]$ and $\eta > 0$ we set
\begin{equation}
    B_\eta^\square(p) \coloneqq \{ x \in \Omega : \db(W_x,W_p) < \eta \},
\end{equation}
where the relevant notions of graphons $W_x$ and $W_p$ and the cut distance $\db$ were defined in Section \ref{sec:review_ldp}.
The following result was previously stated informally as Theorem \ref{thm:lipschitzconcentration_informal} in the introduction.

\begin{theorem}
\label{thm:lipschitzconcentration}
For any $p^* \in U_{\bm\beta}$ there is some $\eta > 0$ as well as constants $C, c > 0$
such that the following holds.
Let $\mu$ denote the ERGM measure conditioned on $\allball$, and let $X \sim \mu$.
Let $f : \Omega \to \R$ be a $v$-Lipschitz observable. 
Then for all $\lambda \in [0, c \| v \|_1]$, we have
\begin{equation}
    \mu \left[ \left| f(X) - \E_\mu[f] \right| > \lambda \right]
    \leq 2 \Exp{- c \cdot \max \left\{ \fr{\lambda^2}{n^2 \| v \|_\infty^2},
    \fr{\lambda^2}{\| v \|_1 \| v \|_\infty} - \Exp{\fr{C \lambda}{\| v \|_\infty} - c n} \right\} } + e^{-c n}.
\end{equation}
\end{theorem}

Note that by shifting and scaling $f$ we may assume that $\E_\mu[f] = 0$ and that $\| v \|_\infty = 1$.
This also decreases $\| v \|_1$ by a factor of the original value of $\| v \|_\infty$.
Thus, scaling $\lambda$ by this same factor it suffices to prove that in the case where $\E_\mu[f] = 0$ and $\| v \|_\infty = 1$
we have
\begin{equation}
    \mu[|f(X)| > \lambda] \leq 2 \Exp{- c \cdot \max \left\{ \fr{\lambda^2}{n^2}, \fr{\lambda^2}{\| v \|_1} - e^{C \lambda - c n} \right\} }
    + e^{- cn}
\end{equation}
for all $\lambda \in [0, c \| v \|_1]$.
This in turn will be implied by the following two inequalities both holding for such $\lambda$:
\begin{align}
\label{eq:barbourbound}
    \mu[|f(X)| > \lambda] &\leq 2 \Exp{- c \frac{\lambda^2}{n^2}} + e^{-c n} \\
\label{eq:chatterjeebound}
    \mu[|f(X)| > \lambda] &\leq 2 \Exp{- c \frac{\lambda^2}{\| v \|_1} + e^{C \lambda - n}} + e^{- cn},
\end{align}
adjusting the constants as necessary.
We will apply Corollary \ref{cor:barbour} to prove \eqref{eq:barbourbound} and Theorem \ref{thm:chatterjee}
to prove \eqref{eq:chatterjeebound}.
Both applications will use similar inputs which we give an overview of in the next section before developing
them properly in subsequent sections.

\subsection{Overview of the proof}
\label{sec:lipschitz_overview}

For the rest of this section we will fix some $p^* \in U_{\bm\beta}$, and abbreviate $\allball$ as $\ball$
to avoid notational clutter.
We will also henceforth use $\mu$ to denote the ERGM measure conditioned on $\ball$, but note that this implicitly depends on
$\eta$ which we have not yet chosen.
Additionally, we will use $(X_t)$ to denote the stationary Glauber dynamics with respect to $\mu$, meaning that
there is a hard-wall constraint which prevents the chains from leaving $\ball$.
For technical reasons we also define this chain to simply stay in one place if it starts outside of $\ball$.
This behavior will not be important except for in making the proof of Lemma \ref{lem:initialgoodset} a bit simpler.
Additionally, we denote by $(X_t^x)$ the chain started at $X_0^x = x$.
Finally, we will use $P$ to denote the transition kernel of this Markov chain.

\subsubsection{Proving the bound for local observables}

Let us see how to use Theorem \ref{thm:chatterjee} to prove \eqref{eq:chatterjeebound}.
To use Theorem \ref{thm:chatterjee}, we need to define suitable subsets $\Lambda \sse \Lambda_+ \sse \ball$ and an
antisymmetric function $F : \Lambda_+^2 \to \R$, and verify conditions \ref{cond:gbound}-\ref{cond:aprioribound} of the theorem.
For $F$, as previously mentioned, we will take
\begin{equation}
    F(x,x') = \sum_{t = 0}^{T-1} \E \left[ f(X_t^x) - f(X_t^{x'}) \right],
\end{equation}
for a suitable choice of $T$, which will turn out to be $T = n^3$.
We delay a precise definition of the subset $\Lambda$ (which will be introduced in Section \ref{sec:lipschitz_goodset});
for now, it should be thought of as a connected
collection of configurations where all relevant local subgraph counts are close to those of $G(n,p^*)$, after appropriate normalization.
These local subgraph counts determine the dynamics, and allow for a comparison with a suitable Erd\H{o}s--R\'enyi graph,
which will yield dynamical results useful for the proof---we will expand on this in the following few subsections.
As for $\Lambda_+$, we will simply take the set of configurations reachable from $\Lambda$ in one step of the Glauber
dynamics, i.e.\ those $x \in \ball$ with $\dh(x,\Lambda) \leq 1$.
Here $\dh(x,y)$ denotes the Hamming distance, i.e.\ the number of
edge locations $e \in \binom{[n]}{2}$ on which the two configurations $x$ and $y$ disagree, and $\dh(x,\Lambda)$ denotes
the minimal distance from $x$ to $y$ for $y \in \Lambda$.
This trivially satisfies the condition $\Pcond{X' \in \Lambda_+}{X \in \Lambda} = 1$
when $(X,X')$ is an exchangeable pair corresponding to the stationary reversible Glauber dynamics chain with respect to $\mu$.

As it will turn out, our choice of $\Lambda$ will satisfy condition \ref{cond:lambdalarge} of Theorem \ref{thm:chatterjee}
with $\eps = e^{-\Omega(n)}$.
Specifically, there is some $\eta > 0$ for which we can find a set $\Lambda$ which satisfies certain nice properties to be
discussed later (see Proposition \ref{prop:goodset} for details), with
\begin{equation}
\label{eq:lambdalarge}
    \mu(\Lambda) \geq 1 - e^{-\Omega(n)}
\end{equation}
(recall that $\mu$ implicitly depends on $\eta$).
This will be proved in Section \ref{sec:lipschitz_goodset}.

Conditions \ref{cond:aprioribound} and \ref{cond:diffbound} of Theorem \ref{thm:chatterjee} are immediate from
the assumptions that $\E_\mu[f] = 0$ and $\| v \|_\infty = 1$.
Namely, we can take $\Delta = 1$ since $\| v \|_\infty = 1$ and the Glauber dynamics takes steps of size at most $1$
in Hamming distance.
Additionally, since $\E_\mu[f] = 0$ we can take $M = \| v \|_1$.
Note that both of these upper bounds hold uniformly in $\ball$; the restriction to $\Lambda$ or $\Lambda_+$ in
Theorem \ref{thm:chatterjee} is not necessary in this case.

Conditions \ref{cond:approxlift} and \ref{cond:gbound} are more involved.
For \ref{cond:approxlift}, as previously derived in Section \ref{sec:concentration_chatterjee}, we have
\begin{equation}
    \Econd{F(X,X')}{X = x} = f(x) - P^T f(x),
\end{equation}
and so we need to show that $P^T f(x)$ is small for a suitable choice of $T$, when $x \in \Lambda_+$.
Since $\E_\mu[f] = 0$, we have
\begin{equation}
    |P^T f(x)| \leq M \cdot \dtv (\delta_x P^T, \mu),
\end{equation}
where $\dtv$ denotes the total variation distance between probability distributions and $\delta_x$ is the Dirac
mass at $x$.
We would like to achieve $\delta = e^{-\Omega(n)}$ in condition \ref{cond:approxlift}, and since $M = \| v \|_1 \leq n^2$
by our assumption on $\| v \|_\infty$, it suffices to obtain the following bound, uniformly among $x \in \Lambda_+$:
\begin{equation}
\label{eq:dtvbound}
    \dtv(\delta_x P^T, \mu) \leq e^{-\Omega(n)}.
\end{equation}

Finally, as for condition \ref{cond:gbound}, note first that we have
\begin{align}
    g(x) &\leq \sum_{x'} P(x,x') |f(x) - f(x')| \sum_{t=0}^{T-1} \left| \E \left[ f(X_t^x) - f(X_t^{x'}) \right] \right| \\
    &\leq \sum_{x'} P(x,x') |f(x) - f(x')| \sum_{t=0}^{T-1} \E \left[ \dh(X_t^x, X_t^{x'}) \right],
\end{align}
since $\| v \|_\infty = 1$.
Now, suppose we had the following bound for all $x \in \Lambda$ and $x'$ such that $\dh(x,x') = 1$:
\begin{equation}
\label{eq:contraction}
    \E \left[ \dh(X_t^x, X_t^{x'}) \right] \lesssim \left(1 - \frac{\kappa}{n^2} \right)^t
\end{equation}
for some $\kappa > 0$ and all $t < T$.
Then we would obtain
\begin{equation}
\label{eq:dhbound}
    \sum_{t=0}^{T-1} \E \left[ \dh(X_t^x, X_t^{x'}) \right] \lesssim n^2,
\end{equation}
which yields
\begin{equation}
    g(x) \lesssim n^2 \sum_{x'} P(x,x') |f(x) - f(x')|.
\end{equation}
Now, since under the Glauber dynamics each of the $\binom{n}{2}$ edges have equal probability to be updated, 
and if an edge $e$ is updated then $|f(x) - f(x')| \leq v_e$, we obtain
\begin{equation}
    g(x) \lesssim n^2 \sum_e \fr{1}{\binom{n}{2}} v_e
    \lesssim \| v \|_1,
\end{equation}
meaning we can take $V = O(\| v \|_1)$.

By the above discussion, if there is some $T$ for which both \eqref{eq:dtvbound} and \eqref{eq:contraction} hold, then
we can apply Theorem \ref{thm:chatterjee} with $\eps = e^{-\Omega(n)}$, $M = \| v \|_1$, $\Delta = 1$,
$\delta = e^{-\Omega(n)}$, and $V = O(\| v \|_1)$, obtaining the following bound on the tail of $f$:
for all $\lambda \in [0, O(\| v \|_1)]$,
\begin{equation}
    \mu[ f(X) > \lambda ] \leq
    \Exp{- \fr{\lambda^2}{O(\| v \|_1)} + \fr{\lambda^3}{O(\| v\|_1^2)} + \fr{e^{-\Omega(n)} \lambda}{O(\| v \|_1)}
    + 2 e^{-\Omega(n)} e^{2 \lambda \| v \|_1 / O(\| v \|_1)}} + e^{-\Omega(n)}.
\end{equation}
Now, there is some small enough $c_1$ such that if $\lambda \leq c_1 \| v \|_1$, then the second term in the exponent
is at most half of the first term in magnitude,
so that the first two terms can be combined into one term of the form $- \fr{c_2 \lambda^2}{\| v \|_1}$.
Additionally, since $\lambda \lesssim n^2$ and $\| v \|_1 \geq \| v \|_\infty = 1$,
the third term is simply $\leq e^{-\Omega(n)}$.
Finally, the fourth term is $\leq e^{c_3 \lambda - \Omega(n)}$, which can absorb the third term as well.
Defining suitable constants $c_4$ and $c_5$ to replace the asymptotic notation, we thus obtain
\begin{equation}
    \mu[ f(X) > \lambda ] \leq
    \Exp{- \fr{c_2 \lambda^2}{\| v \|_1} + e^{c_3 \lambda - c_4 n}} + e^{-c_5 n},
\end{equation}
which, together with the corresponding fact for $-f$, is the bound \eqref{eq:chatterjeebound}.

So to finish the proof of \eqref{eq:chatterjeebound}, it remains only to construct $\Lambda$ and verify \eqref{eq:lambdalarge},
as well as to find $T$ for which both \eqref{eq:dtvbound} and \eqref{eq:contraction} hold; this will be achieved at $T = n^3$.

\subsubsection{Proving the bound for global observables}

As it turns out, once we construct $\Lambda$ which satisfies \eqref{eq:lambdalarge}, \eqref{eq:dtvbound}, and \eqref{eq:contraction}
for $T = n^3$ as in the proof of \eqref{eq:chatterjeebound},
these conditions will also imply \eqref{eq:barbourbound} through Corollary \ref{cor:barbour}, finishing the proof of Theorem \ref{thm:lipschitzconcentration}.
Let us now briefly see how this works by verifying the conditions \ref{cond:variance}-\ref{cond:fbound} of Corollary \ref{cor:barbour}
under the assumption of the existence of $\Lambda$ satisfying \eqref{eq:lambdalarge}, \eqref{eq:dtvbound}, and \eqref{eq:contraction}.

First, condition \ref{cond:fbound} follows immediately by definition with $M = \| v \|_1 \leq n^2$.

Next, for conditions \ref{cond:setlarge} and \ref{cond:mixing}, we will need to select a good starting point $z \in \Lambda$.
Since $T = n^3$ and $\mu(\Lambda) \geq 1 - e^{-\Omega(n)}$ by \eqref{eq:lambdalarge}, by a union bound we may select some $z \in \Lambda$
such that
\begin{equation}
    \P \left[ X_t^z \in \Lambda \text{ for } 0 \leq t < T \right] \geq 1 - e^{-\Omega(n)},
\end{equation}
validating condition \ref{cond:setlarge} with $\eps = e^{-\Omega(n)}$.
Further, since $z \in \Lambda$, \eqref{eq:dtvbound} shows that condition \ref{cond:mixing} holds with
$\delta = e^{-\Omega(n)}$ as well.

Now, for condition \ref{cond:smalljumps}, we have
\begin{equation}
    \left|
        \E[f(X_t^x)] - \E[f(X_t^{x'})]
    \right| \leq \E \left[ \dh(X_t^x, X_t^{x'}) \right],
\end{equation}
recalling that we have assumed $\| v \|_\infty = 1$.
Thus \eqref{eq:contraction} shows that we can take $\Delta = 1$ as long as $t < T$.
But for $t \geq T$ we also have that $\P[X_t^x \neq X_t^{x'}] \leq \P[X_T^x \neq X_T^{x'}] = e^{-\Omega(n)}$,
so the bound holds for all $t$ as required.

Finally, for condition \ref{cond:variance}, we have
\begin{align}
    h(x) &= \sum_{x'} P(x,x') \sum_{t=0}^{T-1} \left( \E[f(X_t^x)] - \E[f(X_t^{x'})] \right)^2 \\
    &\leq \sum_{x'} P(x,x') \sum_{t=0}^{T-1} \E \left[ \dh(X_t^x, X_t^{x'}) \right]^2 \\
    &\lesssim \sum_{t=0}^{n^3-1} \left( 1 - \frac{\kappa}{n} \right)^{2t} \\
    &\lesssim n^2,
\end{align}
using \eqref{eq:contraction} at the second to last step.
So we may take $V \lesssim n^2$.

The conclusion of Corollary \ref{cor:barbour} with $\eps = e^{-\Omega(n)}, \delta = e^{-\Omega(n)}, M \leq n^2, \Delta = 1$,
and $V \lesssim n^2$ is thus
\begin{align}
    \P \left[ |f(X)| > \lambda \right] \leq 
    2 \Exp{ - \frac{(\lambda - e^{-\Omega(n)})^2}{2 O(n^2) + \frac{4}{3} (\lambda - e^{-\Omega(n)})}} + e^{-\Omega(n)}
\end{align}
for all $\lambda \geq e^{-\Omega(n)}$.
When $\lambda \lesssim \| v \|_1 \lesssim n^2$, the second term in the denominator is dominated by the first.
Additionally, the condition that $\lambda \geq e^{-\Omega(n)}$ is of no issue as when $\lambda$ is smaller than that the bound
is trivially true upon adjusting the constants appropriately.
So we obtain \eqref{eq:barbourbound} for $\lambda \leq c \| v \|_1$.

This finishes the proof of Theorem \ref{thm:lipschitzconcentration_informal}, modulo the construction of the good starting
set $\Lambda$ which satisfies the conditions \eqref{eq:lambdalarge}, \eqref{eq:dtvbound}, and \eqref{eq:contraction}.
Constructing this set will occupy the remainder of this section, centered around Proposition \ref{prop:goodset}
which is one of the main technical and conceptual contributions of this article.
\subsection{Inputs from metastable mixing}
\label{sec:lipschitz_inputs}

In this subsection we introduce the main inputs from the works \cite{bhamidi2008mixing,bresler2024metastable} about the
mixing behavior of the ERGM Glauber dynamics, which we will use in the construction of the good set $\Lambda$ and the
verification of \eqref{eq:lambdalarge}, \eqref{eq:dtvbound}, and \eqref{eq:contraction}.
Here we use $(\tilde{X}_t)$ to denote the \emph{unconditioned} stationary Glauber dynamics under the full ERGM measure $\tilde{\mu}$,
i.e.\ here there is no hard-wall constraint forcing the dynamics to stay in $\ball$.
Also, we use $(\tilde{X}_t^x)$ to denote this chain started at $\tilde{X}_0^x = x$.
We also still use $\mu$ to denote the ERGM measure conditioned on $\ball$ for some choice of $\eta > 0$.

At a high level, both the rapid mixing of the ERGM at high temperature and the metastable mixing at low temperature
follow from the existence of large attracting sets within which contraction properties arise under the monotone coupling.
To discuss these sets, we need notation, which we will pull from \cite{bhamidi2008mixing,bresler2024metastable},
choosing one convention when those works disagree.
For any fixed graph $G = (V,E)$, we remind the reader that $N_G(x,e)$ denotes the number of labeled subgraphs of $x^{+e}$
which are isomorphic to $G$ and use the edge $e$, or equivalently
\begin{equation}
    N_G(x,e) = N_G(x^{+e}) - N_G(x^{-e}),
\end{equation}
where $N_G(X)$ denotes the homomorphism count of $G$ in $X$ as defined in Section \ref{sec:review}.
Now define
\begin{equation}
    r_G(x,e) \coloneqq \left( \fr{N_G(x,e)}{2 |E| n^{|V|-2}} \right)^{\fr{1}{|E|-1}}.
\end{equation}
To motivate this definition, notice that if $X \sim \cG(n,p)$, then $r_G(X,e) \approx p$ with high probability.
Additionally, inserting the above definition into the discussion of Section \ref{sec:review_dynamics},
the update probabilities of the \emph{unconditioned} ERGM Glauber dynamics can be expressed as
\begin{equation}
    \P\left[\tilde{X}_1^x(e) = 1 \right] = \phi\left(\sum_{i=0}^K 2 \beta_i |E_i| r_{G_i}(x, e)^{|E_i|-1}\right),
\end{equation}
where $\phi(z) = \fr{e^z}{1 + e^z}$.
So, conversely, if $x$ were such that all $r_G(x,e)$ were close to $p$ (for relevant $G$ and all $e$), then
the update probability under the unconditioned dynamics would be close to
\begin{equation}
    \phi_{\bm\beta}(p) = \phi\left(\sum_{i=0}^K 2 \beta_i |E_i| p^{|E_i|-1}\right).
\end{equation}
Therefore if $p^*$ is a fixed point of $\phi_{\bm\beta}$, and if we start with a configuration $x$ with all relevant
$r_G(x,e)$ close to $p^*$, then $\tilde{X}_1^x$ will retain this property.
Moreover, the monotone coupling will turn out to be contractive among graphs with this property.

To make use of this, we define
\begin{equation}
    \Gamma_{p,\eps} \coloneqq \left\{ x \in \Omega : |r_G(x,e) - p| \leq \eps \text{ for all } e \in \binom{[n]}{2}
    \text{ and all } G \in \G_L \right\},
\end{equation}
where $L$ is some constant with $L > \max_{i \leq K} |V_i|$, and $\G_L$ is the collection of all finite graphs
on at most $L$ vertices.
As usual, we will fix some $p^* \in U_{\bm\beta}$ and we will abbreviate $\Gamma_{p^*,\eps}$ as $\Gamma_\eps$
to avoid notational clutter.
The following two lemmas make the previously-mentioned intuition about the Glauber dynamics precise, showing that
$\Gamma_\eps$ is relatively stable under the Glauber dynamics, and that within this set the monotone coupling is contractive.
Note that both of the following lemmas actually hold for any $p^*$ which is a local maximum of $L_{\bm\beta}$, not just those in $U_{\bm\beta}$
(which are all \emph{global} maxima), but later results such as Theorem \ref{thm:goodsetlarge} below only hold for $p^* \in U_{\bm\beta}$.

\begin{lemma}[Lemma 17 of \cite{bhamidi2008mixing}: staying in $\Gamma_{\eps}$ when started in $\Gamma_{\eps/2}$]
\label{lem:staygood}
For all small enough $\eps > 0$, there exists $\alpha > 0$ such that whenever $x \in \Gamma_{\eps/2}$, we have
\begin{equation}
    \P\left[\tilde{X}_t^x \in \Gamma_{\eps} \text{ for all } t < e^{\alpha n}\right]
    \geq 1 - e^{-\alpha n}.
\end{equation}
\end{lemma}

We remark that in the original statement of this lemma in \cite{bhamidi2008mixing}, it is not explicitly mentioned
that the error term on the right-hand side can be taken as uniform over $x \in \Gamma_{\eps/2}$ (i.e.\ written as $e^{-\alpha n}$
for some $\alpha$ not depending on $x$), however this is clear from the proof and was rectified in the restatement of the lemma
as \cite[Lemma 4.8]{bresler2024metastable}.

\begin{lemma}[Lemma 18 of \cite{bhamidi2008mixing}: one-step contraction within $\Gamma_{\eps}$]
\label{lem:contractgood}
For all small enough $\eps > 0$, there exists $\kappa > 0$ such that for all $x,x' \in \Gamma_{\eps}$ with
$\dh(x,x') = 1$, we have
\begin{equation}
    \E\left[\dh(\tilde{X}_1^x, \tilde{X}_1^{x'})\right] \leq \left(1 - \fr{\kappa}{n^2}\right),
\end{equation}
where $\tilde{X}_1^x, \tilde{X}_1^{x'}$ are coupled under the monotone coupling.
\end{lemma}

We remark that a simple path coupling argument allows the above contraction to be extended to all pairs of initial
configurations $x, x' \in \Gamma_\eps$ with $x \preceq x'$, since $\Gamma_\eps$ is an interval in the Hamming order.
However, this \emph{does not} imply that the contraction holds for arbitrary $x, x' \in \Gamma_\eps$, as it is not
a priori clear that the set $\Gamma_\eps$ is connected, with respect to the nearest-neighbor structure of the Hamming cube.
This will be remedied in the next section, where we extract a large connected subset of $\Gamma_\eps$.

The stability statement, Lemma \ref{lem:staygood}, was proved by comparing each $r_G(\tilde{X}_t^x,e)$ to a random walk, which has some
drift towards $p^*$ from all all starting ponits in the high-temperature regime.
On the other hand, in low temperature, there are multiple fixed points $p^* = \phi_{\bm\beta}(p^*)$, and the corresponding random
walks only have drift towards a fixed point if they start within a certain well of attraction.
The main contribution of \cite{bresler2024metastable} is to show that nevertheless, the set $\Gamma_{p^*,\eps}$ is large when conditioned
on a cut distance ball around $W_{p^*}$, for any relevant fixed point $p^*$.

\begin{theorem}[Theorem 4.5 of \cite{bresler2024metastable}: $\Gamma_\eps$ is large]
\label{thm:goodsetlarge}
For any $\eps > 0$, there is some $\eta > 0$ such that
\begin{equation}
    \mu[\Gamma_\eps] \geq 1 - e^{-\Omega(n)},
\end{equation}
where $\mu$ denotes the ERGM measure conditioned on $\ball$.
\end{theorem}

We remark that the above does not mean $\Gamma_\eps$ is contained in $\ball$, and the truth of this statement is
not clear since the definition of $\Gamma_\eps$ involves the local edge counts of only finitely many graphs.
This issue will be resolved by simply considering the intersection of $\Gamma_\eps$ with an appropriate cut distance ball
in the sequel.
As an easy consequence of the above theorem, we also find that smaller sets $\Gamma_\delta$ are also large,
for any $\delta \in (0,\eps)$.
The following corollary follows by an application of the large deviations principle, Theorem \ref{thm:LDP}, which allows us to
expand the cut distance ball exhibited in Theorem \ref{thm:goodsetlarge} for $\delta > 0$ to the potentially larger ball
exhibited for $\eps > \delta$.

\begin{corollary}
\label{cor:goodsetlarge}
Let $\eps, \eta > 0$ be as in Theorem \ref{thm:goodsetlarge}.
Then, for any $\delta \in (0,\eps)$, we have
\begin{equation}
    \mu[\Gamma_\delta] \geq 1 - e^{-\Omega_\delta(n)},
\end{equation}
where $\mu$ denotes the ERGM measure conditioned on $\ball$.
\end{corollary}

As another corollary of Theorem \ref{thm:goodsetlarge} (since Erd\H{o}s--R\'enyi models are also ERGMs), which can also be proved directly,
we obtain the following, which states that an Erd\H{o}s--R\'enyi random graph
with density parameter near $p^*$ will also land in $\Gamma_\eps$ with high probability.

\begin{corollary}[Corollary 4.7 of \cite{bresler2024metastable}]
\label{cor:GNPgoodset} 
Let $\eps, \eta > 0$ be as in Theorem \ref{thm:goodsetlarge}.
Then, for any small enough $\delta \geq 0$,
\begin{align}
    \pi[\Gamma_\eps] \geq 1 - e^{-\Omega_\delta(n)},
\end{align}
where $\pi = \cG(n,p^*+\delta)$.
\end{corollary}

The above corollary is crucial as we will need to compare the ERGM with a sample from $\cG(n,p^*+\delta)$ in order to
obtain appropriate connectivity properties of the good set $\Lambda$, which will allow us to use path coupling.
Another important ingredient regarding the Erd\H{o}s--R\'enyi measure is the following, which states that if we start
with a sample from $\pi = \cG(n,p^*+\delta)$ and evolve it \emph{under the ERGM Glauber dynamics}, then the entire
process is unlikely to leave $\allhalfball$ for a long time.
This is quite similar to \cite[Lemma 4.10]{bresler2024metastable}, although in that lemma the measure
$\cG(n,p^*)$ was considered instead of $\cG(n,p^*+\delta)$.
Nevertheless, essentially the same proof applies, so we simply state the result here.

\begin{lemma}
\label{lem:GNPstayinball}
For all $\eta > 0$, the following holds for all small enough $\delta > 0$.
Let $(\tilde{X}_t^Z)$ be the unconditioned ERGM Glauber dynamics, but started at $Z \sim \cG(n,p^*+\delta)$.
Then the entire trajectory up to time $T$ stays within $B_{\eta/2}^\square(p^*)$ with probability at least
$1 - Te^{-\Omega_\delta(n)}$.
\end{lemma}

The above lemma also holds for the conditioned stationary dynamics $(X_t)$ with $X_0 \sim \mu$, but
that is easier to prove simply by a union bound (we will use this later).
The importance of the conditioned dynamics staying in $\allhalfball$, which we henceforth abbreviate as $\halfball$,
is that it allows us to guarantee that the monotone coupling between various
Glauber dynamics with respect to $\mu$ preserve orderings of configurations.
If one of the chains attempted to leave $\ball$, it may be stopped and allow another chain in the monotone coupling
to break the ordering.
However, since one step in the dynamics only changes one edge, as long as $n$ is large enough, if $x$ is in
$\halfball$, then there is no way for $X_1^x$ to attempt to leave $\ball$.

Finally, we remark that both $\Gamma_\eps$ and $\halfball$ are \emph{intervals} in the Hamming order, in the
sense that if $X \preceq Y \preceq Z$ and $X, Z \in \Gamma_\eps$, then $Y \in \Gamma_\eps$ (and the same holds
for $\halfball$).
To see why $\Gamma_\eps$ is an interval, notice that $N_G(X,e)$ is an increasing function of $X$, and so
$r_G(X,e)$ is as well, and $\Gamma_\eps$ is the intersection of preimages of real intervals under these functions.
As for $\halfball$, we can apply the following lemma.

\begin{lemma}[Lemma 7.3 of \cite{bresler2024metastable}]
For $p \in [0,1]$, let $W_p$ denote the constant graphon with value $p$.
If $X$, $Y$, and $Z$ are graphons such that $X \preceq Y \preceq Z$, then $\db(Y,W_p) \leq \max\{\db(X,W_p), \db(Z,W_p)\}$.
\end{lemma}
\subsection{Construction and properties of a good starting set}
\label{sec:lipschitz_goodset}

In this section we will construct a good starting set $\Lambda$.
Proposition \ref{prop:goodset} below shows the existence of a set $\Lambda$ which satisfies a few crucial properties
that allow us to verify \eqref{eq:lambdalarge}, \eqref{eq:dtvbound}, and \eqref{eq:contraction}, in order to finish
the proof of Theorem \ref{thm:lipschitzconcentration}.
In particular, property \ref{property:lambdalarge} in the proposition statement below is the same as \eqref{eq:lambdalarge}.

The other two properties listed in Proposition \ref{prop:goodset} below will be used in the next subsection to verify
\eqref{eq:dtvbound} and \eqref{eq:contraction}, for which we would like to apply a path coupling argument using the
contraction properties of $\Gamma_\eps$ detailed in Lemma \ref{lem:contractgood}.
As mentioned, the problem with simply taking $\Lambda = \Gamma_\eps$ for suitable $\eps > 0$ is that this set
may not be connected in the Hamming order, which blocks an easy application of path coupling to obtain
contraction throughout $\Gamma_\eps$.
We will remedy this by showing that, even if $\Gamma_\eps$ is not connected, there is a large subset of $\Gamma_\eps$
which is connected, and we will take $\Lambda$ to be this large connected subset.

Here and in the sequel, for any set $A \sse \ball$ and $x, x' \in A$, we will use the notation $\d_A(x,x')$ to denote
the minimum Hamming-length of a path inside $A$ connecting $x$, and $x'$.
Note that this is finite for all $x, x' \in A$ if and only if $A$ is connected in the nearest-neighbor Hamming cube.
Additionally, here we emphasize again that we use $(X_t)$ (with no tilde) for the \emph{conditioned} stationary Glauber dynamics,
i.e.\ the one which has stationary distribution $\mu$ (conditioned on $\ball$), and $(X_t^x)$ for the conditioned dynamics starting at $x$.

\begin{proposition}[Good starting set $\Lambda$]
\label{prop:goodset}
For all small enough $\eps > 0$, there is some $\eta > 0$ such that for all large enough $n$ there is a set
$\Lambda \sse \ball$ satisfying the following properties:
\begin{enumerate}[label=(\Roman*)]
    \item
    \label{property:lambdalarge}
    Letting $\mu$ denote the ERGM measure conditioned on $\ball$, we have
    \begin{equation}
        \mu(\Lambda) \geq 1 - e^{-\Omega(n)}.
    \end{equation}
    \item
    \label{property:lambdaplusstaygood}
    There is some $\alpha > 0$ such that whenever $\dh(x,\Lambda) \leq 1$, we have
    \begin{equation}
        \P \left[ X_t^x \in \Gamma_\eps \cap \halfball \text{ for all } t < e^{\alpha n} \right] \geq 1 - e^{-\Omega(n)}.
    \end{equation}
    \item
    \label{property:lambdadiambound}
    For any $x,x' \in \Lambda$,
    \begin{equation}
        \dl(x,x') \leq 2 n^2.
    \end{equation}
\end{enumerate}
\end{proposition}

To prove Proposition \ref{prop:goodset},
we start with the following lemma which indicates that for most starting positions (with respect to $\mu$),
the Glauber dynamics remains in $\Gamma_\eps \cap \halfball$.
We also need an analogous statement when the starting position is chosen according to $\pi = \cG(n,p^*+\delta)$
for some small $\delta > 0$, still considering the ERGM Glauber dynamics.

\begin{lemma}
\label{lem:initialgoodset} 
For all small enough $\eps > 0$, there is some $\eta > 0$ such that the following holds.
Define
\begin{equation}
    h_T(x) \coloneqq \P \left[ X_t^x \notin \Gamma_\eps \cap \halfball \text{ for some } t < T \right].
\end{equation}
Then there is some $\xi > 0$ such that
\begin{equation}
\label{eq:initialgoodset_mu}
    \mu[h_{e^{\xi n}}(X) \geq e^{-\xi n}] \leq e^{- \xi n}.
\end{equation}
Additionally, for all small enough $\delta > 0$, there is some $\zeta > 0$ such that
\begin{equation}
\label{eq:initialgoodset_pi}
    \pi[h_{e^{\zeta n}}(Z) \geq e^{-\zeta n}] \leq e^{- \zeta n},
\end{equation}
where $Z \sim \pi = \cG(n,p^*+\delta)$.
\end{lemma}

\begin{proof}[Proof of Lemma \ref{lem:initialgoodset}]
Let us first examine the claim about $\mu$.
Consider a trajectory $(X_t)$ of the stationary Glauber dynamics in $\ball$ started at $X_0 \sim \mu$.
Since each $X_t \sim \mu$ as well, by Theorems \ref{thm:goodsetlarge}
and \ref{thm:LDP}, there is some $\gamma > 0$ for which
\begin{equation}
    \P[X_t \notin \Gamma_\eps \cap \halfball] \leq e^{-\gamma n}
\end{equation}
uniformly in $t$.
So, by a union bound, for any $T$ we have
\begin{equation}
    \P[X_t \notin \Gamma_\eps \cap \halfball \text{ for some } t < T] \leq T e^{-\gamma n}.
\end{equation}
Thus, by Markov's inequality, singling out the first step of the chain, if $X \sim \mu$, we have
\begin{equation}
    \mu[h_T(X) \geq r] \leq \fr{1}{r} T e^{-\gamma n}.
\end{equation}
Therefore we can take $T = e^{\gamma n / 3}$ and $r = e^{- \gamma n /3}$, and we obtain
\begin{equation}
    \mu[h_{e^{\gamma n / 3}}(X) \geq e^{-\gamma n / 3}] \leq e^{\gamma n / 3} e^{\gamma n / 3} e^{- \gamma n}
    = e^{- \gamma n / 3},
\end{equation}
meaning we can take $\xi = \fr{\gamma}{3}$ to obtain \eqref{eq:initialgoodset_mu}.

Next we turn to $\pi$, the Erd\H{o}s--R\'enyi measure with parameter $p^*+\delta$ for some small $\delta > 0$.
First, by Lemma \ref{lem:GNPstayinball}, for some $\chi > 0$ we have the following result for the unconditioned ERGM dynamics started
at $Z$: here the starting point is $Z \sim \pi$ and the symbol $\P$ denotes \emph{only the probability with respect to the dynamics}.
\begin{equation}
\label{eq:10or1}
    \pi\left[\P \left[ \tilde{X}_t^Z \notin \halfball \text{ for some } t < e^{\chi n} \right] \geq e^{-\chi n}\right] \leq e^{- \chi n}.
\end{equation}
Next, by Corollary \ref{cor:GNPgoodset} for all small enough $\delta > 0$ there is some $\gamma > 0$ such that
\begin{equation}
    \pi[\Gamma_{\eps/2}] \geq 1 - e^{-\gamma n},
\end{equation}
and by Lemma \ref{lem:staygood} this means that we have
\begin{equation}
\label{eq:10or2}
    \pi\left[\P \left[ \tilde{X}_t^Z \notin \Gamma_\eps \text{ for some } t < e^{\alpha n} \right] \geq e^{-\alpha n}\right] \leq e^{- \gamma n}.
\end{equation}
Now, if the Erd\H{o}s--R\'enyi sample $Z$ is such that both
\begin{equation}
    \P \left[ \tilde{X}_t^Z \notin \halfball \text{ for some } t < e^{\chi n} \right] < e^{- \chi n}
    \qquad \text{and} \qquad
    \P \left[ \tilde{X}_t^Z \notin \Gamma_\eps \text{ for some } t < e^{\alpha n} \right] < e^{- \alpha n}
\end{equation}
then by a union bound we have
\begin{equation}
    \P \left[ \tilde{X}_t^Z \notin \Gamma_\eps \cap \halfball \text{ for some } t <  e^{\min\{ \alpha, \chi \} \cdot n} \right]
    < e^{- \chi n} + e^{- \alpha n} \leq 2 e^{- \min\{ \chi , \alpha \} \cdot n}.
\end{equation}
Now on the above event, $(\tilde{X}_t^Z)$ may be coupled perfectly with $(X_t^Z)$ (the conditioned dynamics)
up to time $t < e^{\min\{\alpha,\chi\} \cdot n}$.
So in particular, for such $Z$, we have
\begin{equation}
    \P \left[ X_t^Z \notin \Gamma_\eps \cap \halfball \text{ for some } t <  e^{\min\{ \alpha, \chi \} \cdot n} \right]
    < 2 e^{- \min\{ \chi , \alpha \} \cdot n}
\end{equation}
as well, since if $X_t^Z \neq \tilde{X}_t^Z$ then $\tilde{X}_t^Z$ must have left $\ball \supseteq \halfball$.
Thus, by \eqref{eq:10or1} and \eqref{eq:10or2}, we have \eqref{eq:initialgoodset_pi}
for any $\zeta$ which is strictly less than $\min\{\gamma, \alpha, \chi\}$, as soon as $n$ is large enough.
\end{proof}

Next, we use the following lemma to upgrade the fact that a set has high probability to the fact that with high probability, 
all \emph{neighbors} of a sampled point lie in that set.

\begin{lemma}
\label{lem:allneighborsgood}
Suppose that $\Xi \sse \ball$ satisfies
\begin{equation}
    \mu[\Xi] \geq 1 - e^{-\Omega(n)}.
\end{equation}
Let $\Xi_-$ denote the set of all $x \in \Xi$ for which $x' \in \Xi$ whenever $\dh(x,x') = 1$.
Then we also have
\begin{equation}
    \mu[\Xi_-] \geq 1 - e^{-\Omega(n)}.
\end{equation}
The same holds for $\mu$ replaced by $\pi = \cG(n,p^*+\delta)$ for all small enough $\delta > 0$.
\end{lemma}

\begin{proof}[Proof of Lemma \ref{lem:allneighborsgood}]
First notice that
\begin{equation}
    \mu[\Xi \cap \Gamma_\eps \cap \halfball] \geq 1 - e^{-\Omega(n)}
\end{equation}
as well, and it suffices to prove the statement for $\Xi$ replaced by $\Xi \cap \Gamma_\eps \cap \halfball$,
so let us assume that $\Xi \sse \Gamma_\eps \cap \halfball$ from now on.

Let $X \sim \mu$ and $X'$ be obtained from $X$ by one step of the $\mu$-Glauber dynamics.
Then
\begin{align}
    \Pcond{X' \notin \Xi}{X \in \Xi}
    &= \fr{\P[X \in \Xi, X' \notin \Xi]}{\P[X \in \Xi]} \\
    &\leq \fr{\P[X' \notin \Xi]}{\P[X \in \Xi]} \\
    &\leq \fr{e^{- \Omega(n)}}{1 - e^{- \Omega(n)}} \\
    &\leq e^{-\Omega(n)}.
\end{align}
Now, since $p^* \notin \{0,1\}$ as can be checked by examining the derivative of $L_{\bm\beta}$,
if $\eps > 0$ is small enough, then for all $x \in \Gamma_\eps \cap \halfball$
and all $x'$ with $\dh(x,x') = 1$, we have
\begin{equation}
\label{eq:highprobmovement}
    \Pcond{X' = x'}{X = x} \geq \fr{c}{n^2}
\end{equation}
for some $c > 0$.
The $\fr{1}{n^2}$ factor comes from the fact that each edge is chosen to be updated with equal probability,
and if an edge is chosen, then since $x \in \Gamma_\eps$, the probability of adding that edge
is close to $p^*$ which is a constant distance away from $0$ or $1$, meaning that both options (adding or removing)
have at least constant probability.
Here we also use the fact that when $x \in \halfball$, it is not possible for the Glauber dynamics to attempt to
leave $\ball$, so none of the probabilities are set to zero by the hard wall mechanism.
Therefore,
\begin{align}
    \P[X \notin \Xi_-] &= 
        \sum_{x \in \Xi} \P[X = x] \cdot
        \ind{\text{some neighbor } x' \text{ of } x \text{ is not in } \Xi}
        + \P[X \notin \Xi] \\
    &\leq \sum_{x \in \Xi} \P[X = x] \cdot \fr{n^2}{c} \Pcond{X' \notin \Xi}{X = x}
        + \P[X \notin \Xi] \\
    &= \fr{n^2}{c} \P[X' \notin \Xi \text{ and } X \in \Xi] + \P[X \notin \Xi] \\
    &\leq \fr{n^2}{c} e^{-\Omega(n)} + e^{-\Omega(n)} \\
    &\leq e^{-\Omega(n)}.
\end{align}
This finishes the proof for $\mu$, and the same proof applies to $\pi$, with the key condition
\eqref{eq:highprobmovement} holding trivially for all small enough $\delta > 0$.
\end{proof}

Finally we can prove Proposition \ref{prop:goodset}; the key point is that since we have a set which
is large under both $\mu$ and $\pi$, we can extract a \emph{connected} subset of it which is still large under $\mu$, 
crucially using monotonicity to connect two different samples from $\mu$ via a sample from $\pi$.

\begin{proof}[Proof of Proposition \ref{prop:goodset}]
Fix some small enough $\delta > 0$ and let $\alpha = \min\{\xi,\zeta\}$, where $\xi$ and $\zeta$ are as in
Lemma \ref{lem:initialgoodset}. 
Define $\Xi$ to be the set of $x \in \ball$ for which $h_{e^{\alpha n}}(x) < e^{-\alpha n}$, and
define $\Xi_-$ as in Lemma \ref{lem:allneighborsgood}.
Now by Lemma \ref{lem:initialgoodset}, we have $\mu(\Xi), \pi(\Xi) \geq 1 - e^{\Omega(n)}$, and so by
Lemma \ref{lem:allneighborsgood} we have $\mu(\Xi_-), \pi(\Xi_-) \geq 1 - e^{\Omega(n)}$.

First notice that $\Xi_-$ is an interval, in the sense that if $X \preceq Y \preceq Z$ and
$X, Z \in \Xi_-$, then $Y \in \Xi_-$ as well.
This can be seen by considering the monotone coupling between the Glauber dynamics started at
$X$, $Y$, and $Z$, noting that both $\Gamma_\eps$ and $\halfball$ are intervals in the same sense,
and using the fact that as long as none of the Glauber dynamics leave $\halfball$, the coupling
preserves the ordering of the configurations.

Now let $X, X' \sim \mu$ be two independent samples.
We will use the above facts to show that $X$ and $X'$ are in the same connected component of $\Xi_-$ with high probability.
To do this, we will resample both $X$ and $X'$ edge-by-edge, and simultaneously sample $Z \sim \cG(n,p^*+\delta)$
for some small enough $\delta > 0$.
This technique was inspired by the proof of \cite[Lemma 4.9]{bresler2024metastable}.

Specifically, choose some arbitrary order $e_1, \dotsc, e_N$ on the $N = \binom{n}{2}$ possible edge locations,
and let $X_0 = X$ and inductively for $i = 1, \dotsc, N$, let $X_i$ be obtained from $X_{i-1}$ by resampling the
edge $e_i$ according to its conditional distribution, given $X_{i-1}(e_j)$ for $j \neq i$, using an independent Bernoulli
random variable to determine each $X_i(e_i)$.
Do this simultaneously to generate $X_i'$ from $X_0' = X'$, using the uniform coupling between the Bernoulli
random variables determining $X_i(e_i)$ and $X_i'(e_i)$.
Additionaly, sample $Z$ by setting $Z(e_i) = 1$ with probability $p^* + \delta$, using again the uniform coupling
between the Bernoulli random variables.

First, since $\phi_{\bm\beta}$ is continuous and $\phi_{\bm\beta}(p^*) = p^*$,
there is some $\delta' > 0$ such that $\phi_{\bm\beta}(p^* + \delta') < p^* + \delta$.
Now since $\mu(\Gamma_{\delta'}) \geq 1 - e^{-\Omega(n)}$, by a union bound over the $\leq n^2$ samples $X_0, \dotsc, X_N$
and $X_0', \dotsc, X_N'$, with probability at least $1 - e^{-\Omega(n)}$, all of these samples are in
$\Gamma_{\delta'}$.
On this event, the chance for all $X_i(e_i) = 1$ and $X_i'(e_i) = 1$ is at most $\phi_{\bm\beta}(p^*+\delta') < p^* + \delta$,
meaning that by the end of the procedure we have
\begin{equation}
    X_N \preceq Z \qquad \text{and} \qquad X_N' \preceq Z.
\end{equation}
Additionally, by another union bound, we have that with probability at least $1 - e^{-\Omega(n)}$, $Z$ and all of the $X_i$
and $X_i'$ are in $\Xi_-$.
On this event, $X$ and $X'$ are in the same connected component of $\Xi_-$, since there is a path within $\Xi$ connecting
$X$ to $X_N$, and then another path connecting $X_N$ to $Z$ by the fact that $\Xi_-$ is an interval, and another path
from $Z$ to $X_N'$ and then to $X'$.
Additionally, the total length of the path from $X$ to $X'$ is at most $4 N$, since the paths from $X_N$ and $X_N'$
to $Z$ are monotone and have length at most $N$.

So define $\Lambda$ to be this connected component of $\Xi_-$.
By the above reasoning, we obtain property \ref{property:lambdalarge}, that $\Lambda$ is large.
Property \ref{property:lambdaplusstaygood}, that starting in $\Lambda_+$, the Glauber dynamics stays in
$\Gamma_\eps \cap \halfball$ for a long time with high probability, follows by the construction of $\Xi$ and the
fact that all neighbors of an element of $\Xi_-$ are in $\Xi$.
Finally, the diameter bound, property \ref{property:lambdadiambound}, follows from the bound on the path length in
the previous paragraph, since $4 N \leq 2 n^2$.
\end{proof}
\subsection{Distance bounds along a trajectory of the right length}
\label{sec:lipschitz_bounds}

In this section we will validate \eqref{eq:dtvbound} and \eqref{eq:contraction} for a common
choice of $T$, finishing the proof of Theorem \ref{thm:lipschitzconcentration}.
We will use the starting set $\Lambda$ constructed in the previous section, and also define
\begin{equation}
    \Lambda_+ = \{ x \in \ball : \dh(x,\Lambda) \leq 1 \},
\end{equation}
so that $\Pcond{X' \in \Lambda_+}{X \in \Lambda} = 1$, and moreover
every $x \in \Lambda_+$ satisfies the bound in property \ref{property:lambdaplusstaygood} of
Proposition \ref{prop:goodset}.
Additionally, we define $\bar{\Lambda}$ to be the largest connected subset of $\Gamma_\eps \cap \halfball$
containing $\Lambda$; this will be the domain in which all of the relevant dynamics happens, with high probability.

\begin{lemma}
\label{lem:Tchoice}
For $T = n^3$, there are constants $C, c > 0$ such that
\begin{equation}
\label{eq:dtvbound_verified}
    \dtv(\delta_x P^T, \mu) \leq e^{-c n} 
\end{equation}
for all $x \in \Lambda_+$, and
\begin{equation}
\label{eq:contraction_verified}
    \E \left[ \dh(X_t^x, X_t^{x'}) \right] \leq C \left(1 - \frac{\kappa}{n^2} \right)^t
\end{equation}
for all $x \in \Lambda$ and $x'$ such that $\dh(x,x') = 1$ and for all $t < T$.
\end{lemma}

\begin{proof}
Let us verify \eqref{eq:dtvbound_verified} first.
Let $(X_t^x)$ and $(X_t)$ be $\mu$-Glauber dynamics, i.e.\ the ERGM Glauber dynamics conditioned on $\ball$, with $X_0^x = x$
and $X_0 \sim \mu$.
Couple the steps the two chains via the monotone coupling.
Then we have
\begin{equation}
    \dtv(\delta_x P^T, \mu) \leq \P[X_T^x \neq X_T],
\end{equation}
since $X_T \sim \mu$ as well.
By property \ref{property:lambdaplusstaygood} of $\Lambda$, since $T = n^3 \ll e^{\alpha n}$,
we have $X_t^x \in \bar{\Lambda}$ for all $t \leq T$
with probability at least $1 - e^{-\Omega(n)}$, since trivially $X_t^x$ remains in the connected component of
$\Gamma_\eps \cap \halfball$ containing $\Lambda$.
By properties \ref{property:lambdalarge} and \ref{property:lambdaplusstaygood}, we also have
$X_t \in \bar{\Lambda}$ for all $t \leq T$ with probability at least $1 - e^{-\Omega(n)}$.
Let $\cB$ denote the bad event that one of the above conditions fail, i.e.\ some $X_t^x$ or $X_t$ is not in $\bar{\Lambda}$
for some $t \in [0,T]$.
Then $\P[\cB] \leq e^{-\Omega(n)}$.

Now, if $X_t^x, X_t \in \bar{\Lambda}$, since $\bar{\Lambda} \sse \Gamma_\eps \cap \halfball$, we have contraction, although
not necessarily in $\dh$.
Instead, we have contraction in $\dlb$ by a path coupling argument applied with Lemma \ref{lem:contractgood}
as the unit-distance contraction input, and obtain
\begin{equation}
    \Econd{\dlb(X_{t+1}^x, X_{t+1})}{X_t^x, X_t} \leq \left(1 - \fr{\kappa}{n^2}\right) \dlb(X_t^x, X_t),
\end{equation}
for some $\kappa > 0$.
Therefore, since $\dh \leq \dlb$, we have
\begin{align}
    \P[X_T^x \neq X_T] &\leq \Pcond{X_T^x \neq X_T}{\cB^\c} + \P[\cB] \\
    &\leq \Econd{\dh(X_T^x, X_T)}{\cB^\c} + e^{-\Omega(n)} \\
    &\leq \Econd{\dlb(X_T^x,X_T)}{\cB^\c} + e^{-\Omega(n)} \\ 
    &\leq \left(1 - \fr{\kappa}{n^2} \right)^T \max\{ \dlb(x,x') : x,x' \in \Lambda_+\} + e^{-\Omega(n)} \\
    &\leq e^{- \kappa T / n^2} \cdot (2 n^2 + 2) + e^{-\Omega(n)} \\
    &\leq e^{- c n}
\end{align}
for some $c > 0$, substituting $T = n^3$ in the last inequality and
using property \ref{property:lambdadiambound} in the second-to-last inequality,
since $\dlb(x,x') \leq \dl(x,x')$ for $x,x' \in \Lambda$.
This completes the verification of \eqref{eq:dtvbound_verified}.

Next we turn to \eqref{eq:contraction_verified}.
Let $(X_t^x), (X_t^{x'})$ be $\mu$-Glauber dynamics started at $X_0^x = x \in \Lambda$ and $X_0^{x'} = x'$ with $\dh(x,x') = 1$,
again coupled via the monotone coupling.
Define the bad event $\cB$ as before, that one of the chains leaves $\bar{\Lambda}$ by time $T$.
Then $\P[\cB] \leq e^{-\Omega(n)}$ as before.
Since either $x \preceq x'$ or $x' \preceq x$, whichever ordering holds at the start persists for all time
$t \in [0,T]$ as long as $\cB^\c$ holds.
In this case, we have $\dh(X_t^x, X_t^{x'}) = \dlb(X_t^x, X_t^{x'})$ for all $t \in [0,T]$ since $\bar{\Lambda}$ is an interval,
and so contraction in $\dlb$ implies the same contraction for $\dh$.
Therefore we obtain
\begin{align}
    \E\left[\dh(X_t^x,X_t^{x'})\right] &\leq \Econd{\dh(X_t^x,X_t^{x'})}{\cB^\c} + n^2 \P[\cB] \\
    &\leq \left(1 - \fr{\kappa}{n^2} \right)^t + e^{-\Omega(n)}
\end{align}
which validates \eqref{eq:contraction_verified} upon changing the constants, since $t$ is at most $n^3$.
This finishes the proof.
\end{proof}

Since conditions \eqref{eq:dtvbound_verified} and \eqref{eq:contraction_verified} are exactly the same as conditions
\eqref{eq:dtvbound} and \eqref{eq:contraction} with $T = n^3$, Lemma \ref{lem:Tchoice} finishes the proof of
Theorem \ref{thm:lipschitzconcentration}.

We remark that in fact the above proof works with any $T < e^{\alpha n}$, where $\alpha$ is defined as in
property \ref{property:lambdaplusstaygood} of $\Lambda$ in Proposition \ref{prop:goodset}, but there is no real
benefit from choosing a different $T$ larger than $n^3$; we could decrease the value of $\delta$
we could take in applying Theorem \ref{thm:chatterjee} and Corollary \ref{cor:barbour}, but the external $e^{-cn}$ term
in Theorem \ref{thm:lipschitzconcentration}, which arises as the
probability under $\mu$ of the complement of the good set $\Lambda$,
cannot be made smaller than $e^{-O(n)}$ due to \cite[Theorem 3.3]{bresler2024metastable}.

Finally, we remark that a precise understanding of the number $\kappa$ appearing in Lemma \ref{lem:contractgood}
would allow us to understand on a similar level of precision the leading-order term in the exponent
in the result of Theorem \ref{thm:lipschitzconcentration}.
However, we do not pursue this here.
\section{Applications}
\label{sec:applications}

In this section we state some applications of Theorem \ref{thm:lipschitzconcentration} to the typical behavior
of the supercritical ERGM, extending results from \cite{ganguly2024sub} in the subcritical regime.
For the most part, essentially the same proofs from the subcritical regime go through, with the relevant concentration inputs
replaced by Theorem \ref{thm:lipschitzconcentration}, and so we will not provide complete proofs but rather give a high-level
overview of the strategies.

One notable exception is that the results of \cite{ganguly2024sub} rely on the FKG inequality, combined with their
concentration result, in order to obtain bounds on the covariance between individual edge variables under the ERGM measure,
as well as estimates of multilinear moments of these variables.
While the FKG inequality does still hold in the supercritical regime for the \emph{entire} ERGM measure, it does not hold for
the measure conditioned on the ball $\ball$.
In the case where $|M_{\bm\beta}| = 1$, the conditioned measure is very close to the full measure by Theorem \ref{thm:LDP},
and we can simply apply the FKG inequality of the full measure.

However, when there is phase coexistence, for each $p^* \in U_{\bm\beta}$ the measure conditioned on $\ball$ might be quite
different from the full measure, being only one of multiple pieces which together form a measure close to the full measure.
In this case, we must find ways around the FKG inequality, and in the bound on the Wasserstein distance to $\cG(n,p^*)$
(Theorem \ref{thm:wasserstein_informal}), this requires us to restrict to a special case of the ERGM measure where all graphs
in the specification are trees.
The central limit theorem for sparse collections of edges (Theorem \ref{thm:clt_informal}) remains unaffected,
as we are able to completely avoid the FKG inequality for the required estimates there.
In fact, we prove a slightly more broad statement, which also extends the work of \cite{ganguly2024sub} within the subcritical regime.

In Section \ref{sec:applications_cov} below, we provide the aforementioned covariance and multilinear moment estimates.
In Section \ref{sec:applications_wasserstein} we state and prove a formal version of Theorem \ref{thm:wasserstein_informal},
and in Section \ref{sec:applications_clt} we do the same for Theorem \ref{thm:clt_informal}.
Throughout this section, we fix $p^* \in U_{\bm\beta}$, as well as some small enough fixed $\eta > 0$, and let
$\mu$ denote the ERGM measure conditioned on $\allball$.

\subsection{Covariance and multilinear moment estimates}
\label{sec:applications_cov}

In this section we provide two lemmas which will be useful for both of Theorems \ref{thm:wasserstein_informal} and
\ref{thm:clt_informal}, namely an upper bound on the covariance between any pair of possible edges, as well as an
extension of this bound to certain higher multilinear moments of edge variables.
These results are the analogues, for the supercritical regime, of \cite[Lemma 6.1]{ganguly2024sub}.

As mentioned above, in that work, a key input was the FKG inequality which holds for the full ERGM measure.
This implied the nonnegativity of the covariances, allowing for an easier proof of their bound by decomposing the variance of the
total edge count.
In addition, an inequality due to \cite{newman1980normal, bulinski1998asymptotical},
which follows from the FKG inequality, was used to bound the higher multilinear moments in terms of the covariances themselves.

For the following covariance bound, we are able to bypass this FKG input entirely and obtain an equivalent bound on the
\emph{absolute value} of the covariances, although we cannot prove that the covariances are nonnegative.

\begin{lemma}
\label{lem:cov}
For any distinct $e, e' \in \binom{[n]}{2}$, we have
\begin{equation}
	\left| \Cov_\mu[X(e), X(e')] \right| \lesssim \fr{1}{n}.
\end{equation}
Moreover, if $e$ and $e'$ do not share a vertex, then the right-hand side
above can be improved to $\fr{1}{n^2}$.
\end{lemma}

\begin{proof}[Proof of Lemma \ref{lem:cov}]
First, consider $e_1, \dotsc, e_{n-1}$ which are all of the edges incident on a particular vertex, i.e.\ these
edges form a maximal star in the complete graph.    
Then by Theorem \ref{thm:lipschitzconcentration}, we have
\begin{equation}
    \Var_\mu\left[\sum_{i=1}^{n-1} X(e_i)\right] \lesssim n.
\end{equation}
Now let us expand the variance as follows:
\begin{align}
    \Var_\mu\left[ \sum_{i=1}^{n-1} X(e_i) \right]
    &= \sum_{i=1}^{n-1} \Var_\mu[X(e_i)] + \sum_{i \neq j} \Cov_\mu[X(e_i),X(e_j)] \\
    &= (n-1) \cdot \Var_\mu[X(e)] + (n-1)(n-2) \cdot \Cov_\mu[X(e),X(e')],
\end{align}
using symmetry in the second step, where $e$ is any edge and $e'$ is any edge sharing a vertex with $e$.
The above implies that
\begin{equation}
\label{eq:cov_share}
    \left| \Cov_\mu[X(e),X(e')] \right| \lesssim \fr{1}{n}
\end{equation}
whenever $e$ and $e'$ share a vertex.
Now consider the sum of \emph{all} edge variables $X(e_1) + \dotsb + X(e_N)$, where $N = \binom{n}{2}$.
Now using the first bound in Theorem \ref{thm:lipschitzconcentration}, we have
\begin{equation}
    \Var_\mu\left[\sum_{i=1}^N X(e_i)\right] \lesssim n^2,
\end{equation}
and by expanding this as before we obtain
\begin{equation}
    \Var_\mu\left[ \sum_{i=1}^N X(e_i) \right]
    = N \cdot \Var_\mu[X(e)] + N (n-1) \cdot \Cov_\mu[X(e),X(e')]
    + N(N-n+1) \cdot \Cov_\mu[X(e),X(e'')],
\end{equation}
where $e$ and $e'$ share a vertex, and $e$ and $e''$ do not.
Using \eqref{eq:cov_share} to bound the second term, this means that
\begin{equation}
\label{eq:cov_noshare}
    \left|\Cov_\mu[X(e), X(e'')]\right| \lesssim \fr{1}{n^2}
\end{equation}
whenever $e$ and $e''$ do not share a vertex.
\end{proof}

We now extend this bound to certain higher multilinear moments of the edge variables.
For this, we cannot completely bypass the FKG inequality in all cases, due to the somewhat more complicated structure
of the expressions that arise, which depend on the type of graph that the edges span.
In the phase uniqueness case $|M_{\bm\beta}| = 1$, we can simply use the FKG inequality for the unconditioned measure,
but in the phase coexistence case, we must restrict collections of edges which form a forest in $K_n$, which allows for a
convenient decomposition argument similar in spirit to the one used for Lemma \ref{lem:cov}.

\begin{proposition}
\label{prop:Eproduct}
For a fixed $k$, let $e_1, \dotsc, e_k \in \binom{[n]}{2}$ be distinct potential edges and suppose that either of the
following two conditions hold:
\begin{enumerate}[label=(\alph*)]
	\item The maximizer $p^* \in U_{\bm\beta}$ is the unique global maximizer of $L_{\bm\beta}$, i.e.\ $|M_{\bm\beta}| = 1$.
	\item The edges $e_1, \dotsc, e_k$ form a forest in $K_n$.
\end{enumerate}
Then we have
\begin{equation}
	\left| \E_\mu\left[ \prod_{j=1}^k X(e_j) \right]
	- \E_\mu[X(e)]^k \right| \lesssim \fr{1}{n},
\end{equation}
where $e$ is an arbitrary edge in $\binom{[n]}{2}$.
\end{proposition}

\begin{proof}[Proof of Proposition \ref{prop:Eproduct}]
We consider each case separately.
\begin{enumerate}[label=(\alph*)]
\item
Since the total variation distance between $\mu$ and the full ERGM measure is $e^{-\Omega(n^2)}$ in this case by Theorem \ref{thm:LDP},
it suffices to prove the result for the full ERGM measure, where the FKG inequality holds by the nonnegativity of all $\beta_i$ for $i \geq 1$.
Following \cite{ganguly2024sub}, we invoke \cite[Equation (12)]{newman1980normal}, which posits the existence of a constant $C$ such that
whenever random variables $(Z_1, \dotsc, Z_m)$ satisfy the FKG inequality, for any $C^1$ functions $f, g : \R^m \to \R$, we have
\begin{equation}
\label{eq:newman}
	\Cov[f(Z_1,\dotsc,Z_m), g(Z_1,\dotsc,Z_m)] \leq C \sum_{i=1}^m \sum_{j=1}^m
	\left\| \fr{\partial f}{\partial z_i} \right\|_\infty \left\| \fr{\partial g}{\partial z_j} \right\|_\infty \Cov[Z_i, Z_j].
\end{equation}
See also \cite{bulinski1998asymptotical,bulinski2007limit} for a proof of this fact which shows that we may take $C = 1$.
In our setting since $|U_{\bm\beta}| = 1$, Lemma \ref{lem:cov} implies that each covariance in the
\emph{unconditioned} measure is also $\lesssim \fr{1}{n}$.
So using \eqref{eq:newman} for the unconditioned measure and then restricting to the conditioned measure, we obtain
\begin{equation}
\label{eq:covcondition}
	\left| \Cov_\mu[X(e_1) \dotsb X(e_j), X(e_{j+1})] \right| \lesssim \fr{1}{n}
\end{equation}
for each $j = 1, \dotsc, k-1$, since the range of each $X(e_i)$ is $\{0,1\}$ which allows us to appropriately modify the multilinear function
$X(e_1) \dotsb X(e_j)$ away from $[0,1]^j$ in order that its partial derivatives are all bounded by some constant, say $2$.
Then we obtain the telescoping sum
\begin{align}
	\left| \E_\mu\left[ \prod_{j=1}^k X(e_j) \right]
	- \E_\mu[X(e)]^k \right| &= \left| \sum_{j=1}^{k-1}
	\Cov_\mu[X(e_1) \dotsb X(e_j), X(e_{j+1})] \cdot
	\E_\mu[X(e)]^{k-1-j} \right| \\
	&\leq \sum_{j=1}^{k-1} \left| \Cov_\mu[X(e_1) \dotsb X(e_j), X(e_{j+1})] \right| \lesssim \fr{1}{n},
\end{align}
since $\E_\mu[X_e] \in [0,1]$.

\item
By the same telescoping sum as above, it suffices to prove condition \eqref{eq:covcondition} for each $j = 1, \dotsc, k-1$.
Let us first consider the case where the edges $e_1, \dotsc, e_k$ are vertex-disjoint, since it is conceptually simpler.

Consider the sum
\begin{equation}
	S(X) = \sum_{i=1}^m X(e_1^i) \dotsb X(e_j^i),
\end{equation}
where $\{ e_\ell^i : 1 \leq i \leq m, 1 \leq \ell \leq j \}$ is some enumeration of a collection of pairwise vertex-disjoint edges.
We can find exactly $\lfloor \fr{n}{2} \rfloor$ such edges, and so since $j$ is fixed,
we can ensure that the number $m$ of summands in $S$ is $\gtrsim n$.
Since $S(X)$ is $v$-Lipschitz with $\| v \|_\infty = 1$ and $\| v \|_1 \lesssim n$, we have
\begin{equation}
	\Var_\mu[S(X)] \lesssim n
\end{equation}
by Theorem \ref{thm:lipschitzconcentration}.
Now splitting up the variance as in the proof of Lemma \ref{lem:cov}, we find by symmetry that
\begin{equation}
	\Var_\mu[S(X)] = m \cdot \Var_\mu\left[X(e_1^1) \dotsb X(e_j^1)\right]
	+ m(m-1) \Cov_\mu\left[X(e_1^1) \dotsb X(e_j^1), X(e_1^2) \dotsb X(e_j^2)\right].
\end{equation}
Thus, since $m = \Omega(n)$ and $\Var_\mu\left[X(e_1^1) \dotsb X(e_j^1)\right] \leq 1$, we must have
\begin{equation}
\label{eq:covdoublebound}
	\left| \Cov_\mu\left[X(e_1^1) \dotsb X(e_j^1), X(e_1^2) \dotsb X(e_j^2)\right] \right| \lesssim \fr{1}{n}.
\end{equation}
Now let us consider the sum
\begin{equation}
	S^+(X) = \sum_{i=1}^m X(e_1^i) \dotsb X(e_j^i)
	+ \sum_{i=1}^m X(e_{j+1}^i),
\end{equation}
where $\{ e_\ell^i : 1 \leq i \leq m, 1 \leq \ell \leq j+1 \}$ is another collection of pairwise vertex-disjoint edges.
We can ensure that $m \gtrsim n$ as before, and as before
$S^+(X)$ is $v$-Lipschitz with $\| v \|_\infty = 1$ and $\| v \|_1 \lesssim n$, so
\begin{equation}
	\Var_\mu[S^+(X)] \lesssim n
\end{equation}
as well by Theorem \ref{thm:lipschitzconcentration}.
Now expanding the variance as before and using symmetry, we obtain
\begin{align}
	\Var_\mu[S^+(X)] &= m \Var_\mu\left[X(e_1^1) \dotsb X(e_j^1)\right]
	+ m \Var_\mu\left[ X(e_{j+1}^1) \right] \\
	&\qquad + m(m-1) \Cov_\mu\left[X(e_1^1) \dotsb X(e_j^1), X(e_1^2) \dotsb X(e_j^2)\right] \\
	&\qquad + m(m-1) \Cov_\mu\left[ X(e_{j+1}^1), X(e_{j+1}^2) \right] \\
	&\qquad + m^2 \Cov_\mu\left[ X(e_1^1) \dotsb X(e_j^1), X(e_{j+1}^1) \right].
\end{align}
Now by \eqref{eq:covdoublebound} and Lemma \ref{lem:cov}, all terms on the right-hand side, other than the last one, are $\lesssim n$,
and so we must also have
\begin{equation}
\label{eq:covprod}
	\left| \Cov_\mu\left[ X(e_1^1) \dotsb X(e_j^1), X(e_{j+1}^1) \right] \right| \lesssim \fr{1}{n},
\end{equation}
which is the same as \eqref{eq:covcondition}, finishing the proof in the case where all edges are vertex-disjoint.

For the general case where the edges form a forest but are not all vertex-disjoint, we first need to reorder the edges
so that when thinking of the forest ``growing'' by adding each edge in sequence, each connected component of the forest is
established by a single edge before any components grow, and then the components grow from the leaves only.

More precisely we require that for each $j=1,\dotsc,k-1$, the edge $e_{j+1}$ is a leaf in the forest spanned by
$\{e_1, \dotsc, e_{j+1}\}$.
Let us also ensure that there is some $j_0$ for which $e_{j+1}$ is not isolated in the forest spanned by $\{e_1,\dotsc,e_{j+1}\}$,
whenever $j \geq j_0$, and for which the edges $\{e_1, \dotsc, e_{j+1}\}$ are vertex-disjoint when $j < j_0$.

Then the cases $j < j_0$ follow exactly as above, and we only need to consider the cases $j \geq j_0$.
Here the proof also follows from essentially the same analysis as above,
but we need to construct the sums $S(X)$ and $S^+(X)$ a bit more carefully.

Let $F$ be the forest spanned by $\{e_1, \dotsc, e_j\}$, and let $F^+$ be the forest spanned by $\{e_1, \dotsc, e_{j+1}\}$
Let $u \in F$ be the vertex where $e_{j+1}$ would be attached to form $F^+$, and let $R$ be the set of indices $\ell \in \{1,\dotsc,j+1\}$
for which $e_\ell$ is incident on $u$.

Now for the first sum we will take
\begin{equation}
	S(X) = \sum_{i=1}^m X(e_1^i) \dotsb X(e_j^i),
\end{equation}
where for each fixed $i$ the edges $e_1^i, \dotsc, e_j^i$ span a forest in $K_n$ of the same type as $F$, and all
$e_\ell^i$ with $\ell \in R$ (across all $i$) are incident on a single vertex, with no other pairs of edges $e_\ell^i$
and $e_{\ell'}^{i'}$ sharing a vertex if $i \neq i'$ and $\ell$ or $\ell'$ is not in $R$ (see Figure \ref{fig:trees}, left).
By the same analysis as in the previous case, we can obtain \eqref{eq:covdoublebound}, since each vertex has $n-1$ edges incident
on it in $K_n$ and so we can take $m \gtrsim n$.

\begin{figure}
\begin{center}
	\resizebox{0.4\textwidth}{!}{\begin{tikzpicture}
	\begin{pgfonlayer}{nodelayer}
		\node [style=black] (0) at (0, 0) {};
		\node [style=none] (1) at (-3.75, 1.5) {};
		\node [style=none] (2) at (-6.25, 0.5) {};
		\node [style=none] (4) at (-2.5, 0.25) {$e_1^1$};
		\node [style=none] (5) at (-5, 2.25) {};
		\node [style=none] (6) at (-7.25, 1.25) {};
		\node [style=none] (7) at (-5.25, 4.75) {};
		\node [style=none] (8) at (-4.75, 0.25) {$e_2^1$};
		\node [style=none] (10) at (-3.25, 2.75) {$e_3^1$};
		\node [style=none] (12) at (-7, 2.25) {$e_4^1$};
		\node [style=none] (13) at (-6.25, 4) {};
		\node [style=none] (14) at (-6.25, 4) {$e_5^1$};
		\node [style=none] (15) at (0, 2.75) {};
		\node [style=none] (16) at (-2.25, 3.75) {};
		\node [style=none] (17) at (0, 3.75) {};
		\node [style=none] (18) at (-2, 4.75) {};
		\node [style=none] (19) at (2, 4.75) {};
		\node [style=none] (20) at (3.25, 1) {};
		\node [style=none] (21) at (3.5, 3.25) {};
		\node [style=none] (22) at (4.5, 1.5) {};
		\node [style=none] (23) at (4.75, 3.5) {};
		\node [style=none] (24) at (6.5, 0.25) {};
		\node [style=none] (25) at (2.75, -2.25) {};
		\node [style=none] (26) at (0, -3) {};
		\node [style=none] (27) at (-3, -2.25) {};
		\node [style=none] (28) at (3.75, -3.25) {};
		\node [style=none] (29) at (5, -2) {};
		\node [style=none] (30) at (3.75, -5.25) {};
		\node [style=none] (31) at (5.75, -3) {};
		\node [style=none] (32) at (-0.25, -4.25) {};
		\node [style=none] (33) at (1.25, -4.25) {};
		\node [style=none] (34) at (-1.75, -5.5) {};
		\node [style=none] (35) at (1.25, -5.5) {};
		\node [style=none] (36) at (-4, -2.75) {};
		\node [style=none] (37) at (-3.75, -4.5) {};
		\node [style=none] (38) at (-6, -1.75) {};
		\node [style=none] (39) at (-5, -5) {};
		\node [style=none] (41) at (0.75, 1.5) {$e_1^2$};
		\node [style=none] (44) at (-1.25, 2.5) {$e_2^2$};
		\node [style=none] (47) at (1.5, 2.75) {$e_3^2$};
		\node [style=none] (48) at (-1, 5.25) {$e_4^2$};
		\node [style=none] (49) at (1, 5.25) {$e_5^2$};
		\node [style=none] (50) at (3.5, -1) {$...$};
		\node [style=none] (51) at (-2, -2.5) {$...$};
		\node [style=none] (52) at (1.25, -2.75) {$...$};
		\node [style=none] (53) at (-4, 4.5) {};
		\node [style=none] (54) at (2, 3.75) {};
		\node [style=none] (55) at (5.25, -0.25) {};
		\node [style=none] (56) at (2.5, -4.5) {};
		\node [style=none] (57) at (-1.5, -4.25) {};
		\node [style=none] (58) at (-5, -1.25) {};
	\end{pgfonlayer}
	\begin{pgfonlayer}{edgelayer}
		\draw [style=red] (1.center) to (0);
		\draw [style=red] (2.center) to (1.center);
		\draw [style=red] (5.center) to (6.center);
		\draw [style=red] (7.center) to (5.center);
		\draw [style=orange] (15.center) to (0);
		\draw [style=orange] (16.center) to (15.center);
		\draw [style=orange] (18.center) to (17.center);
		\draw [style=orange] (19.center) to (17.center);
		\draw [style=yellow] (20.center) to (0);
		\draw [style=yellow] (21.center) to (20.center);
		\draw [style=yellow] (22.center) to (24.center);
		\draw [style=yellow] (23.center) to (22.center);
		\draw [style=green] (0) to (25.center);
		\draw [style=green] (25.center) to (29.center);
		\draw [style=green] (28.center) to (30.center);
		\draw [style=green] (28.center) to (31.center);
		\draw [style=blue] (0) to (26.center);
		\draw [style=blue] (26.center) to (33.center);
		\draw [style=blue] (32.center) to (34.center);
		\draw [style=blue] (32.center) to (35.center);
		\draw [style=purple] (38.center) to (36.center);
		\draw [style=purple] (27.center) to (0);
		\draw [style=purple] (27.center) to (37.center);
		\draw [style=purple] (36.center) to (39.center);
		\draw [style=red] (53.center) to (1.center);
		\draw [style=orange] (15.center) to (54.center);
		\draw [style=yellow] (20.center) to (55.center);
		\draw [style=green] (25.center) to (56.center);
		\draw [style=blue] (26.center) to (57.center);
		\draw [style=purple] (58.center) to (27.center);
	\end{pgfonlayer}
\end{tikzpicture}}
	\hspace{10mm}
	\resizebox{0.4\textwidth}{!}{\begin{tikzpicture}
	\begin{pgfonlayer}{nodelayer}
		\node [style=black] (0) at (0, 0) {};
		\node [style=none] (1) at (-3.75, 1.5) {};
		\node [style=none] (2) at (-6.25, 0.5) {};
		\node [style=none] (4) at (-2.5, 0.25) {$e_1^1$};
		\node [style=none] (5) at (-5, 2.25) {};
		\node [style=none] (6) at (-7.25, 1.25) {};
		\node [style=none] (7) at (-5.25, 4.75) {};
		\node [style=none] (8) at (-4.75, 0.25) {$e_2^1$};
		\node [style=none] (10) at (-3.25, 2.75) {$e_3^1$};
		\node [style=none] (12) at (-7, 2.25) {$e_4^1$};
		\node [style=none] (13) at (-6.25, 4) {};
		\node [style=none] (14) at (-6.25, 4) {$e_5^1$};
		\node [style=none] (15) at (0, 2.75) {};
		\node [style=none] (16) at (-2.25, 3.75) {};
		\node [style=none] (17) at (0, 3.75) {};
		\node [style=none] (18) at (-2, 4.75) {};
		\node [style=none] (19) at (2, 4.75) {};
		\node [style=none] (20) at (3.25, 1) {};
		\node [style=none] (21) at (3.5, 3.25) {};
		\node [style=none] (22) at (4.5, 1.5) {};
		\node [style=none] (23) at (4.75, 3.5) {};
		\node [style=none] (24) at (6.5, 0.25) {};
		\node [style=none] (25) at (2.75, -2.25) {};
		\node [style=none] (26) at (0, -3) {};
		\node [style=none] (27) at (-3, -2.25) {};
		\node [style=none] (41) at (0.75, 1.5) {$e_1^2$};
		\node [style=none] (44) at (-1.25, 2.5) {$e_2^2$};
		\node [style=none] (47) at (1.5, 2.75) {$e_3^2$};
		\node [style=none] (48) at (-1, 5.25) {$e_4^2$};
		\node [style=none] (49) at (1, 5.25) {$e_5^2$};
		\node [style=none] (50) at (3.25, -1) {$...$};
		\node [style=none] (51) at (-2, -2.5) {$e_6^1$};
		\node [style=none] (53) at (-4, 4.5) {};
		\node [style=none] (54) at (2, 3.75) {};
		\node [style=none] (55) at (5.25, -0.25) {};
		\node [style=none] (56) at (0.75, -2.75) {$e_6^2$};
	\end{pgfonlayer}
	\begin{pgfonlayer}{edgelayer}
		\draw [style=red] (1.center) to (0);
		\draw [style=red] (2.center) to (1.center);
		\draw [style=red] (5.center) to (6.center);
		\draw [style=red] (7.center) to (5.center);
		\draw [style=orange] (15.center) to (0);
		\draw [style=orange] (16.center) to (15.center);
		\draw [style=orange] (18.center) to (17.center);
		\draw [style=orange] (19.center) to (17.center);
		\draw [style=yellow] (20.center) to (0);
		\draw [style=yellow] (21.center) to (20.center);
		\draw [style=yellow] (22.center) to (24.center);
		\draw [style=yellow] (23.center) to (22.center);
		\draw [style=green] (0) to (25.center);
		\draw [style=blue] (0) to (26.center);
		\draw [style=purple] (27.center) to (0);
		\draw [style=red] (53.center) to (1.center);
		\draw [style=orange] (15.center) to (54.center);
		\draw [style=yellow] (20.center) to (55.center);
	\end{pgfonlayer}
\end{tikzpicture}}
\end{center}
\caption{The constructions of the sums $S(X)$ (left) and $S^+(X)$ (right) for proof of part (b) of Proposition \ref{prop:Eproduct}, in the
case where the edges span a forest but are not all vertex-disjoint.
Different colors represent different summands in the sums.}
\label{fig:trees}
\end{figure}

As for the second sum, we will take
\begin{equation}
	S^+(X) = \sum_{i=1}^m X(e_1^i) \dotsb X(e_j^i) + \sum_{i=1}^m X(e_{j+1}^i),
\end{equation}
where for each fixed $i$, the edges $e_1^i, \dotsc, e_j^i$ span a forest in $K_n$ of type $F$, and all $e_\ell^i$
with $\ell \in R$ (including $\ell = j+1$) are incident on a single vertex, with no other pairs of edges sharing a vertex
(see Figure \ref{fig:trees}, right).
Again, we can ensure that $m \gtrsim n$, and so the same analysis goes through to derive \eqref{eq:covcondition},
which finishes the proof in the general case.
\qedhere
\end{enumerate}	
\end{proof}

We remark that the proof used in part (b) of Proposition \ref{prop:Eproduct} does not work when the graph spanned
by the edges $e_1, \dotsc, e_k$ contains a cycle.
Indeed, the above strategy relies on the fact that in the construction depicted in the right-hand side of
Figure \ref{fig:trees}, each pair of a partial forest and an edge leads to a forest of the correct type, and
there are $\Omega(n)$ partial forests and $\Omega(n)$ single edges, so there are $\Omega(n^2)$ different choices
for a completed forest.

If instead one tried to remove an edge that was part of a cycle, then for each partial graph, there
would only be \emph{one} possible edge that could complete the graph to the correct type, since both endpoints would
be determined uniquely by the partial graph.
It is therefore an interesting question whether the bound of Proposition \ref{prop:Eproduct} can be extended to all graphs,
or even just a graph with one cycle.
This can be proved using the FKG inequality input under assumption (a), but it seems plausible that the result should
remain true even if there is phase coexistence, i.e.\ $|M_{\bm\beta}| > 1$.

\begin{conjecture}
\label{conj:trianglemoment}
Suppose that the edges $e_1, e_2, e_3$ span a triangle in $K_n$.
Then we still have
\begin{equation}
	\left|
		\E_\mu[X(e_1) X(e_2) X(e_3)]
		- \E_\mu[X(e)]^3
	\right| \lesssim \fr{1}{n},
\end{equation}
where $e$ is an arbitrary edge, even in the phase coexistence regime.
\end{conjecture}
\subsection{Bounding the 1-Wasserstein distance to the Erd\H{o}s--R\'enyi graph}
\label{sec:applications_wasserstein}

In this section we give an upper bound on the $1$-Wasserstein distance between $\mu$ and the Erd\H{o}s--R\'enyi measure
$\cG(n,p^*)$.
This distance is defined for two probability measures $\mu, \pi$ on
$\{0,1\}^{\binom{[n]}{2}}$ as follows:
\begin{equation}
    W_1(\mu,\pi) \coloneqq \inf_{\substack{X \sim \mu \\ Y \sim \pi}} \sum_{e \in \binom{[n]}{2}} \P[X(e) \neq Y(e)],
\end{equation}
where the infimum is taken over all couplings of $\mu$ and $\pi$, instantiated by random variables $X \sim \mu$
and $Y \sim \pi$ on the same probability space, whose joint distribution is being optimized over.
The following theorem is the analog of the subcritical result \cite[Theorem 3]{ganguly2024sub}, although in
the phase coexistence case we must restrict to ERGMs specified by forests.
This result was informally stated in the introduction as Theorem \ref{thm:wasserstein_informal}.

\begin{theorem}
\label{thm:wasserstein}
Suppose that $|M_{\bm\beta}| = 1$, \emph{or} that each graph $G_0, \dotsc, G_K$ in the specification of the ERGM measure is a forest.
Let $p^* \in U_{\bm\beta}$ and let $\mu$ denote the ERGM measure conditioned on $B_\eta^\square(p^*)$,
for some small enough $\eta > 0$.
Additionally, let $\pi$ denote the Erd\H{o}s--R\'enyi measure $\cG(n,p^*)$.
Then we have
\begin{equation}
    W_1(\mu,\pi) \lesssim n^{3/2} \sqrt{\log n}.
\end{equation}
\end{theorem}

As mentioned, the proof is essentially the same as that of \cite[Theorem 3]{ganguly2024sub}, and so we only give a high-level
overview, highlighting especially the places which need to be modified for our present setting.

The proof proceeds via the exhibition of a coupling between the ERGM
measure and the Erd\H{o}s--R\'enyi measure.
To construct this coupling, simply sample $X \sim \mu$ and then resample each edge $e$, one by one, in some 
arbitrary fixed order on all edges, with its conditional distribution given the state of the rest of the graph.
After all edges are resampled, we obtain a new state $X'$, with $X' \sim \mu$ still.
Simultaneously, we sample $Z \sim \pi$ by choosing each edge $e$ to be present with probability $p^*$ independently,
but coupled to the choices of resamplings in $X$ such that each Bernoulli random variable at each step is coupled
optimally.
This is the same resampling procedure as in the proof of Proposition \ref{prop:goodset}, although there we considered
$\cG(n,p^*+\delta)$ instead of $\cG(n,p^*)$ as we do here.

A bound on the Wasserstein distance is then obtained by summing up the probabilities that $X(e) \neq Z(e)$ for each
edge $e$, the Wasserstein distance will thus be small if
\begin{equation}
    \left| \E_\mu[X(e)|X(e') \text{ for } e' \neq e] - p^* \right|
\end{equation}
is small with high probability.
Now recall from Section \ref{sec:review_dynamics} that
\begin{equation}
\label{eq:conditional_hardwall}
    \E_\mu[X(e)|X(e') \text{ for } e' \neq e] = \phi(\partial_e H(X))
\end{equation}
with probability at least $1 - e^{-\Omega(n)}$, since with this probability the configuration lies
well enough within $\ball$ that the conditional distribution at each edge is the same as for the full ERGM measure.

Since $\phi$ is a smooth function it thus suffices to show that $\partial_e H(X)$ is concentrated near a value
$q$ for which $\phi(q) = p^*$.
This is where Theorem \ref{thm:lipschitzconcentration} comes in, and to apply it we need to calculate the norms
of the Lipschitz vector of $\partial_e H(x)$.
First, we have
\begin{equation}
    \partial_e H(x) = \sum_{i=0}^K \beta_i \fr{N_{G_i}(x,e)}{n^{|V_i|-2}},
\end{equation}
and as can be easily calculated for any $e' \in \binom{[n]}{2}$ and any $x \in \Omega$, we have
\begin{equation}
    N_{G_i}(x^{+e'},e) - N_{G_i}(x^{-e'},e) \lesssim n^{|V_i|-3},
\end{equation}
since the left-hand side is the number of subgraphs of $x^{+\{e,e'\}}$ isomorphic to $G_i$ which use both $e$ and $e'$.
Moreover, if $e$ and $e'$ do not share a vertex, then the right-hand side above can be improved to $n^{|V_i|-4}$.
This means that in total, $\partial_e H(x)$ is $v$-Lipschitz with $\| v \|_\infty \lesssim \fr{1}{n}$ and
$\| v \|_1 \lesssim 1$.
So Theorem \ref{thm:lipschitzconcentration} gives
\begin{equation}
    \mu\left[\left|\partial_e H(X) - \E_\mu [\partial_e H(X)]\right| > \fr{\lambda}{n} \right]
    \leq 2 \Exp{- \fr{c_2 \lambda^2}{n} + e^{c_3 \lambda - c_4 n}} + e^{- c_5 n}
\end{equation}
for all $\lambda \leq c_1 n$.
In other words, taking $\lambda = \alpha \sqrt{n \log n}$, with probability at least $1 - O(n^{- c \alpha})$, we have
\begin{equation}
\label{eq:w1bound1}
    \left|\partial_e H(X) - \E_\mu[\partial_e H(X)]\right| \leq \alpha \sqrt{\fr{\log n}{n}}.
\end{equation}

Now we need to estimate $\E_\mu[\partial_e H(X)]$.
For this, we first expand the definition of $N_{G_i}(X,e)$ to see that
\begin{equation}
\label{eq:w1nsum}
    \E_\mu[N_{G_i}(X,e)] = \sum_{\substack{\sigma : V_i \to [n] \\ e = \{\sigma_u, \sigma_v\} \text{ for} \\ \text{some } \{u,v\} \in E_i}}
    \E_\mu\left[\prod_{\substack{\{u,v\} \in E_i \\ \{\sigma_u, \sigma_v\} \neq e}} X(\{\sigma_u, \sigma_v\})\right].
\end{equation}
In the case where $|U_{\bm\beta}| = 1$, we may invoke part (a) of Proposition \ref{prop:Eproduct} and conclude that the expectation inside
the sum differs from $\E_\mu[X_e]^{|E_i|-1}$ by $\lesssim \fr{1}{n}$.
Since there are $2 |E_i| n^{|V_i|-2}$ terms in the sum, this allows us to conclude that
\begin{equation}
\label{eq:w1bound2}
    \left| \E_\mu[\partial_e H(X)] - \Psi_{\bm\beta}(\E_\mu[X_e]) \right| \lesssim \fr{1}{n}
\end{equation}
as well, where we remind the reader that
\begin{equation}
    \Psi_{\bm\beta}(p) = \sum_{i=0}^K 2 |E_i| \beta_i p^{|E_i|-1}.
\end{equation}

Let us now consider the phase coexistence regime, where we assume that each $G_i$ is a forest.
There are at most $O(n^{|V_i|-3})$ terms in the sum \eqref{eq:w1nsum} which correspond to maps $\sigma : V_i \to [n]$ such that the
edges $\{\sigma_u, \sigma_v\}$ for $\{u,v\} \in E_i$ \emph{do not} span a forest, since any such embedding must not be injective,
and thus span only $|V_i|-3$ vertices other than the ones in $e$.
As for the terms for which these edges \emph{do} span a forest, we may invoke part (b) of Proposition \ref{prop:Eproduct} to obtain
\eqref{eq:w1bound2} in this case as well.

Now let us apply $\phi$ to all quantities involved in \eqref{eq:w1bound1} and \eqref{eq:w1bound2} and use
the smoothness of this function.
For convenience, write $Y_e = \phi(\partial_e H(X)) = \E_\mu[X(e) | X(e') \text{ for } e' \neq e]$, so that
$\E_\mu[X(e)] = \E_\mu[Y_e]$.
Then by the preceding discussion we see that, if $\alpha$ is chosen large enough, we have
\begin{equation}
    \left| Y_e - \phi(\Psi_{\bm\beta}(\E_\mu[Y_e]))\right| \leq C \sqrt{\fr{\log n}{n}}
\end{equation}
with probability at least $1 - \fr{1}{n^3}$, for some $C > 0$.

Now, since $p^*$ is a fixed point of $p \mapsto \phi(\Psi_{\bm\beta}(p))$, and moreover it is the \emph{only}
fixed point under consideration since we are restricting to $B_\eta$ with $\eta$ small enough that
no other fixed point (which must be a constant distance away from $p^*$) can play a role, we must have
\begin{equation}
\label{eq:w1bound3}
    \left|Y_e - p^*\right| \leq C \sqrt{\fr{\log n}{n}}
\end{equation}
with probability at least $1 - \fr{1}{n^3}$.
The above difference is exactly the probability that $X(e) \neq Z(e)$ under the coupling described above,
and so taking a union bound over the event that \eqref{eq:w1bound3} fails, which happens with probability at most $\fr{1}{n^3}$,
and the event that even if \eqref{eq:w1bound3} holds the two graphs do not agree at $e$, we have
\begin{equation}
    \P[X(e) \neq Z(e)] \leq C \sqrt{\fr{\log n}{n}} + \fr{1}{n^3} + e^{-\Omega(n)},
\end{equation}
the last term being the probability that condition \eqref{eq:conditional_hardwall} fails.
So, summing over all $\binom{n}{2}$ edges we obtain the result of Theorem \ref{thm:wasserstein}.

One interesting consequence of this proof is the following approximation of the marginal at a single edge.
This is analogous to \cite[Lemma 7.1]{ganguly2024sub}.

\begin{proposition}
\label{prop:marginal}  
If $|M_{\bm\beta}| = 1$ or all graphs $G_0, \dotsc, G_K$ in the ERGM specification are forests, then
for any $p^* \in U_{\bm\beta}$ there is some small enough $\eta > 0$ such that for any edge $e \in \binom{[n]}{2}$,
\begin{equation}
    \left|\E_\mu[X(e)] - p^* \right| \lesssim \sqrt{\fr{\log n}{n}}.
\end{equation}
Here $\mu$ denotes the ERGM measure conditioned on $\allball$.
\end{proposition}
\subsection{A central limit theorem for small subcollections of edges}
\label{sec:applications_clt}

In this section we prove a central limit theorem for certain small subcollections of edges,
extending \cite[Theorem 2]{ganguly2024sub} both to the supercritical regime from the subcritical regime,
and also allowing for the consideration of collections of edges around a single vertex, in addition to
collections of edges not sharing a vertex.
Our proof also has the benefit of not relying on the FKG inequality.
This result was informally stated in the introduction as Theorem \ref{thm:clt_informal}.

\begin{theorem}
\label{thm:clt}
Let $p^* \in U_{\bm\beta}$ and let $\mu$ denote the ERGM measure conditioned on $B_\eta^\square(p^*)$,
for some small enough $\eta > 0$.
Suppose that we have a sequence of collections of distinct edges $\{e_i^n\}_{i=1}^m \sse \binom{[n]}{2}$,
with $m = m_n$ satisfying $1 \ll m_n \ll n$ as $n \to \infty$, such that either of the following conditions hold:
\begin{enumerate}[label=(\alph*)]
    \item For each $n$, the edges $e_1^n, \dotsc, e_m^n$ do not share a vertex.
    \item For each $n$, the edges $e_1^n, \dotsc, e_m^n$ are all incident on a single vertex.
\end{enumerate}
Now define
\begin{equation}
    S_n = \fr{X(e_1^n) + \dotsb + X(e_m^n) - \E_\mu[X(e_1^n) + \dotsb + X(e_m^n)]}
    {\sqrt{\Var_\mu[X(e_1^n) + \dotsb + X(e_m^n)]}}.
\end{equation}
Then we have $S_n \dto \mathcal{N}(0,1)$ under $\mu$ as $n \to \infty$.
\end{theorem}

We remark that the restriction to a simple subcollection of edges has been removed in a certain very-high temperature
setting, the \emph{Dobrushin uniqueness regime}, in \cite{fang2024normal}.
However, it is unclear if the methods there extend even to the full subcritical regime.

The proof of Theorem \ref{thm:clt} follows that of \cite[Theorem 2]{ganguly2024sub} closely,
so we just give an overview and point out where changes must be made.
We remark that the two cases (a) and (b) are proved identically, since the symmetry considerations we will use later hold for both cases.
The proof is based on the moment method, which requires various estimates on multi-way correlation functions.
First, we need the following quantitative independence result, analogous to \cite[Proposition 6.2]{ganguly2024sub}.

\begin{proposition}
\label{prop:62}
For any fixed positive integer $k$ and $a_1, \dotsc, a_k, b_1, \dotsc, b_k \in \{0,1\}$, the following holds
for any distinct edges $e, e_1, \dotsc, e_k$ which form a forest in $K_n$:
\begin{equation}
    \Big|
        \P_\mu[X(e) = 1 | X(e_j) = a_j \text{ for } 1 \leq j \leq k]
        - \P_\mu[X(e) = 1 | X(e_j) = b_j \text{ for } 1 \leq j \leq k]
    \Big| \lesssim \fr{1}{n}.
\end{equation}
\end{proposition}

The proof of \cite[Proposition 6.2]{ganguly2024sub} relied on the FKG inequality as well as an input from \cite{bhamidi2008mixing}
which was proved there only in the subcritical regime.
Therefore we present here an alternative proof using a similar strategy but avoiding these inputs.

\begin{proof}[Proof of Proposition \ref{prop:62}]
By the triangle inequality, it suffices to show that 
\begin{equation}
\label{eq:62goal}
    \left|
        \P_\mu[X(e)=1 | X(e_j) = a_j \text{ for } 1 \leq j \leq k]
        - \P_\mu[X_j = 1]
    \right| \lesssim \fr{1}{n}.
\end{equation}
Now note that Proposition \ref{prop:Eproduct} combined with Theorem \ref{thm:LDP} implies that
\begin{equation}
\label{eq:productlimit}
    \P_\mu[X(e_j) = a_j \text{ for } 1 \leq j \leq k] \xrightarrow[\;n \to \infty\;]{} (p^*)^{\sum_i a_i} (1-p^*)^{k-\sum_i a_i},
\end{equation}
since we can express the probability here as a sum of expected values of products of edge variables, all of which span a forest in $K_n$,
allowing us to apply Proposition \ref{prop:Eproduct} and then Theorem \ref{thm:LDP} which implies that each 
$\E_\mu[X(e)] \to p^*$ as $n \to \infty$.
Using \eqref{eq:productlimit} to bound the denominator implicit in \eqref{eq:62goal} by a constant above and below,
to prove \eqref{eq:62goal} it thus suffices to show that
\begin{equation}
    \left|
        \P_\mu[X(e) = 1, X(e_j) = a_j \text{ for } 1 \leq j \leq k] -
        \P_\mu[X(e) = 1] \cdot \P_\mu[X(e_j) = a_j \text{ for } 1 \leq j \leq k]
    \right| \lesssim \fr{1}{n}.
\end{equation}
Now the left-hand side can be written as the absolute value of
\begin{equation}
        \Cov_\mu\left[X(e), \prod_{j=1}^k Y_j \right],
\end{equation}
where $Y_j = X(e_j)$ if $a_j = 1$ and $Y_j = 1-X(e_j)$ if $a_j = 0$.
We can expand this covariance using linearity in the second input, and find that it is a sum of at most $2^k$ terms of the form
\begin{equation}
    \Cov_\mu\left[X(e), \prod_{i=1}^\ell X(e_{j_i}) \right].
\end{equation}
But all of these terms are $\lesssim \fr{1}{n}$ in absolute value, by the proof of Proposition \ref{prop:Eproduct},
part (b) (recall that the proof involved showing condition \eqref{eq:covcondition}).
This finishes the proof.
\end{proof}

Using Proposition \ref{prop:62} as well as Theorem \ref{thm:lipschitzconcentration}, we can prove the following estimate
on multi-way correlation functions, analogous to \cite[Propostion 6.4]{ganguly2024sub}.
For ease of notation, define $\tilde{X}(e) = X(e) - \E_\mu[X(e)]$.

\begin{proposition}
\label{prop:64}
For any fixed integers $l, m \geq 0$ and $a_1, \dotsc, a_m \geq 1$ with $a_1 + \dotsb + a_m \leq 2 m$, there
is a constant $C_{l,a_1,\dotsc,a_m}$ such that the following holds for sufficiently large $n$.
For any set of edges $e_1, \dotsc, e_m, e_1', \dotsc, e_l'$ which are either vertex-disjoint or all incident on a single
vertex, we have
\begin{equation}
    \left|
        \E_\mu \left[
            \prod_{j=1}^m \tilde{X}(e_j)^{a_j}
            \,\middle|\,
            X(e_1'), \dotsc, X(e_l')
        \right]
    \right| \leq \fr{C_{l,a_1,\dotsc,a_m}}{n^{m - (a_1 + \dotsb + a_m)/2}}.
\end{equation}
\end{proposition}

\begin{proof}[Proof sketch for Proposition \ref{prop:64}]
As essentially the same proof from \cite[Proposition 6.4]{ganguly2024sub} goes through here, we just give a sketch,
pointing the interested reader to that reference for details.

The proof proceeds by induction on $m$.
For $m=1$ it is immediate from Proposition \ref{prop:62} for $a_1 = 1$, and the case of $a_1 = 2$ follows because
$|\tilde{X}(e)| \leq 1$.
For larger values of $m$, if $a_i > 1$ for any $i = 1, \dotsc, m$, then one can use the tower property and condition
on $X(e_i)$: then the sum of $a_j$ for $j \neq i$ is at most $2(m-1)$, allowing the inductive hypothesis to apply.
We obtain
\begin{equation}
    \left|\E_\mu \left[ \tilde{X}(e_i)^{a_i} \cdot \E_\mu \left[ \prod_{j \neq i} \tilde{X}(e_j)^{a_j} \,\middle|\, X(e_i), X(e_1'), \dotsc, X(e_l') \right] \right] \right|
    \leq \fr{C_{l+1,a_1,\dotsc,a_{i-1},a_{i+1},a_m}}{n^{(m-1)-(a_1+\dotsb+a_{i-1}+a_{i+1}+\dotsb+a_m)/2}},
\end{equation}
since $|\tilde{X}(e_i)| \leq 1$.
Since $a_i \geq 2$, the right-hand side is of the desired form, finishing the proof in this case.

The tricky case is when $a_j = 1$ for all $j = 1, \dotsc, m$, in other words showing that
\begin{equation}
\label{eq:64goal}
    \left| \E_\mu \left[\tilde{X}(e_1) \dotsb \tilde{X}(e_m) \,\middle|\, X(e_1'), \dotsc, X(e_l') \right] \right| \lesssim \fr{1}{n^{m/2}}.
\end{equation}
In this case, we begin by instead considering one of the following two sums, depending on whether
our edges are vertex-disjoint (case (a)), or all incident on a single vertex (case (b)).
\begin{equation}
    \text{Case (a):} \quad \sum_{i=l+1}^{\lfloor n/2 \rfloor} \tilde{X}(e_i'),
    \qquad \qquad \qquad
    \text{Case (b):} \quad \sum_{i=l+1}^{n-1} \tilde{X}(e_i').
\end{equation}
Here, in case (a), we assume that $\{ e_i' \}_{i=1}^{\lfloor n/2 \rfloor}$ is a maximal collection of vertex-disjoint
edges containing our set $\{ e_1, \dotsc, e_m, e_1', \dotsc, e_l' \}$ (so that in particular each $e_j$ appears as
some $e_i'$ for $i > l$).
In case (b), we assume that we instead have a maximal collection of edges incident on a single vertex, with the same
other properties.

In either case, we use the initially Gaussian concentration afforded by Theorem \ref{thm:lipschitzconcentration}
for the sum, which is $v$-Lipschitz with $\| v \|_\infty = 1$ and $\| v \|_1 \lesssim n$ in either case, meaning
that we do in fact get Gaussian concentration at scale $\sqrt{n}$, for tails of depth up to scale $n$.
This allows us to estimate the moments of the sum, after using Proposition \ref{prop:62} to compare the probabilities
with and without conditioning on $X(e_1'), \dotsc, X(e_l')$.
This yields
\begin{equation}
\label{eq:gaussianmoments}
    \left|
        \E_\mu \left[
            \left(
                \tilde{X}(e_{l+1}') + \dotsb + \tilde{X}(e_{\lfloor n/2 \rfloor}')
            \right)^m
            \,\middle|\,
            X(e_1'), \dotsc, X(e_l')
        \right]
    \right| \lesssim n^{m/2},
\end{equation}
where we replace the $\lfloor n/2 \rfloor$ by $n-1$ in case (b).
We can then expand the left-hand side into terms of the form
\begin{equation}
    \E_\mu \left[
        \tilde{X}(e_{j_1})^{c_1} \dotsb \tilde{X}(e_{j_L})^{c_L}
        \,\middle|\,
        X(e_1'), \dotsc, X(e_l')
    \right],
\end{equation}
where $c_1 \geq \dotsb \geq c_L \geq 1$ and $c_1 + \dotsb + c_L = m$,
and consider the terms based on the form of $(c_1, \dotsc, c_L)$.

First, if $c_1 \geq 2$ and $2L \geq m$, then $c_2 + \dotsb + c_L \leq 2(L-1)$, and so we can apply the induction
hypothesis to bound each term, and then count the number of such terms to see that their total sum
is $\lesssim n^{m/2}$.
Next, if $m > 2L$ then the number of terms of this form is bounded by $n^L \leq n^{(m-1)/2}$ and we simply bound
each one by $1$, obtaining another bound of $\lesssim n^{m/2}$ for the overall sum of terms of this type.

Finally, we are left with terms with $c_1 = \dotsb = c_L = 1$, meaning that $L=m$ and these terms are the type
that we wish to bound in \eqref{eq:64goal}.
By \eqref{eq:gaussianmoments} as well as the previous paragraph, the total sum of these terms is again $\lesssim n^{m/2}$.
\emph{Here} is where we use the assumption of vertex-disjointness or all edges being incident on a single vertex.
Either of these assumptions means that all of the terms of the form we wish to bound are \emph{the same}, since there is
only one graph structure that could be spanned by the edges considered.
Thus, since there are $\Omega(n^m)$ such terms, each one has the desired bound, finishing the proof.
\end{proof}

Now, as mentioned earlier, Theorem \ref{thm:clt} is proved using the moment method.
In other words, we must estimate the moments of the sum $\tilde{X}(e_1^n) + \dotsb + \tilde{X}(e_m^n)$,
for $e_1^n, \dotsc, e_m^n$ satisfying the hypothesis of Theorem \ref{thm:clt}.
First, Proposition \ref{prop:64} allows us to extract the leading-order parts of the $k$th moment of the sum,
controlling every term where there are powers of edge variables different from $0$ and $2$,
and we obtain the following result:
\begin{equation}
\label{eq:leading_terms}
    \E_\mu\left[\left(\tilde{X}(e_1^n) + \dotsb + \tilde{X}(e_m^n)\right)^k\right] = o(m^{k/2}) +
    \binom{k}{2,\dotsc,2} \sum_{1 \leq j_1 < \dotsb < j_{k/2} \leq m}
    \E_\mu \left[\prod_{l=1}^{k/2} \tilde{X}(e_{j_l}^n)^2\right].
\end{equation}
We skip the details here as they are identical to those in the proof of \cite[Theorem 2]{ganguly2024sub}.
Notably, this is where we use the condition that $m \ll n$.
Our convention here is that for odd $k$, the multinomial coefficient in \eqref{eq:leading_terms} is zero,
and so for odd $k$ the $k$th moment of the sum is simply $o(m^{k/2})$.

One particular case of \eqref{eq:leading_terms} is the variance of the sum, i.e.\ the $k=2$ case, where we get
\begin{equation}
    \Var_\mu[X(e_1^n) + \dotsb + X(e_m^n)] = o(m) + m \Var_\mu[X(e)]
\end{equation}
by symmetry, where $e$ is an arbitrary edge.
By Theorem \ref{thm:LDP}, we have $\Var_\mu[X(e)] \to p^* (1-p^*)$, and so we have
\begin{equation}
    \Var_\mu[X(e_1^n) + \dotsb + X(e_m^n)] = o(m) + m p^* (1-p^*).
\end{equation}
Now, by symmetry and the fact that there are $\binom{m}{k/2} \sim \fr{m^{k/2}}{(k/2)!}$
terms in the sum, this tells us that
\begin{equation}
    \E_\mu[S_n^k]
    = o(1) + \binom{k}{2,\dotsc,2} \fr{1}{(k/2)!} \cdot
    \fr{\E_\mu\left[\prod_{j=1}^{k/2} \tilde{X}(e_j)^2 \right]}{(p^*(1-p^*))^{k/2}},
\end{equation}
where again the multinomial coefficient is zero for odd $k$, meaning that the odd moments are $o(1)$.
Finally, by Proposition \ref{prop:Eproduct} and Theorem \ref{thm:LDP} again, we have
\begin{equation}
    \E_\mu\left[\prod_{j=1}^{k/2} \tilde{X}(e_j)^2 \right] \to (p^*(1-p^*))^{k/2},
\end{equation}
meaning that the limiting moments of $S_n$ are $0$ when $k$ is odd, and equal to
\begin{equation}
    \binom{k}{2,\dotsc,2} \fr{1}{(k/2)!} = \fr{k!}{2^{k/2} (k/2)!} = (k-1)!!
\end{equation}
when $k$ is even.
These are the moments of $\mathcal{N}(0,1)$, and so we obtain Theorem \ref{thm:clt}.

\bibliography{references}
\bibliographystyle{alpha}

\end{document}